\documentclass[opre,sglanonrev]{informs4_mod}
\usepackage{eqndefns-left} 
\RequirePackage{tgtermes}
\RequirePackage{newtxtext}
\RequirePackage{newtxmath}
\RequirePackage{bm}
\RequirePackage{endnotes}

\OneAndAHalfSpacedXII 

\usepackage{algorithm}
\usepackage{algcompatible}

\usepackage{tikz}

\usepackage{statistics}
\usepackage{makecell,multirow}
\usepackage[shortlabels]{enumitem}
\usepackage{subcaption}  
\usepackage{subfiles}
\usepackage{wrapfig}
\def\Var{\text{Var}}
\def\Cov{\text{Cov}}
\def\tr{\text{Tr}}
\def\IFx{{\rm IF}}

\usepackage{natbib}
 \bibpunct[, ]{(}{)}{,}{a}{}{,}%
 \def\bibfont{\small}%

\usepackage{hyperref}       

\usepackage{url}            
\usepackage{booktabs}       
\usepackage{amsfonts}       
\usepackage{nicefrac}       
\usepackage{microtype}      
\usepackage{xcolor}         
\usepackage{amsmath,amssymb,mathtools}
\usepackage{cleveref}

\EquationsNumberedThrough    

\TheoremsNumberedThrough     
\ECRepeatTheorems  %

\MANUSCRIPTNO{}

\begin{document}



\RUNTITLE{First-Order Improvements in Data-Driven Optimization}

\TITLE{Achieving First-Order Statistical Improvements in Data-Driven Optimization: From No-Free-Lunch to Amplified Decision Perturbation}

\ARTICLEAUTHORS{%
\AUTHOR{Henry Lam, Tianyu Wang\footnote{Authors are listed in alphabetical order.}}
\AFF{Department of Industrial Engineering and Operations Research,
Columbia University, \EMAIL{\{henry.lam, tianyu.wang\}@columbia.edu}}
} 

\ABSTRACT{%
    Recent proliferation of data-optimization integration has led to a range of methods that aim to improve the statistical performance of data-driven optimization decisions. However, while many of these methods are motivated intuitively from a robustness or regularization perspective, their resulting statistical benefits are often unclear and, even if available, are established on a case-by-case basis. We provide a systematic dissection of data-driven optimization formulations using the view of ``directionally perturbed'' empirical optimization (EO). Specifically, this umbrella of formulations, which we call ``EO+", covers many existing data-driven optimization methods, including regularization, distributionally robust optimization, transfer learning, and analogous methods for contextual optimization. On the one hand, we argue that without additional, correctly specified, side information, any EO+ method can result in at most second-order improvements. This provides a negative conclusion, namely ``no free lunch is possible", on the statistical power of EO+. On the other hand, we show that when leveraging side information that is geometrically effective, achieving first-order improvements is possible by choosing hyperparameters that are significantly larger than what is typically suggested in the literature. Moreover, we construct a principled methodology based on excess risk estimation, via either system knowledge or bootstrap resampling, to maximize the first-order gain. We demonstrate how this gain connects to the control-variate principle, a variance reduction technique in the Monte Carlo simulation literature, which helps explain why geometrically effective side information is necessary.
}%





\maketitle
\section{Introduction}
We consider a data-driven stochastic optimization problem of the form:
\begin{equation}\label{eq:ddsp}
\min_{\theta \in \Theta}\paran{Z(\theta) := \E_{\xi \sim \P}[\ell(\theta;\xi)]},  
\end{equation}
where $\ell(\theta;\cdot)$ is a
known cost function, $\Theta \subseteq \R^{D_{\theta}}$ 
is the set of feasible decisions,
and $\xi$ is a random variable following the distribution $\P$. 
The decision-maker only has access to
$n$ i.i.d. samples
$\Dscr_n:=\{\xi_i\}_{i = 1}^n$. The goal is to 
use the data $\Dscr_n$ to identify a 
decision with the lowest 
expected cost under the \emph{true} distribution $\P$. This problem setup is widely adopted in practice from empirical risk minimization in machine learning~\citep{hastie2009elements}, to operational applications such as supply chain management~\citep{snyder2019fundamentals}, revenue
management~\citep{talluri2006theory,chen2022statistical} and portfolio
optimization~\citep{ban2018machine}. 

Among all data-driven optimization methods, the most straightforward approach is to replace the unknown $\P$ with the empirical distribution $\hat\P_n:= \frac{1}{n}\sum_{i = 1}^n \delta_{\xi_i}$ in~\eqref{eq:ddsp}, yielding the empirical optimization (EO) solution: $\hat\theta_{EO} \in \argmin_{\theta \in \Theta}\frac{1}{n}\sum_{i = 1}^n \ell(\theta;\xi_i).$ Although the EO solution is asymptotically minimax optimal~\citep{shapiro2021lectures}, it
incorporates little problem-specific structure and can exhibit poor finite-sample performance~\citep{gupta2021small}. To address this limitation, a variety of methods have been developed to improve EO. These methods fall into two main categories based on their provenance. In the optimization literature, \emph{optimization-enhanced} 
methods modify the objective function directly. These methods include replacing $\hat\P_n$ with a worst-case distribution over a so-called uncertainty set, as in distributionally robust optimization (DRO)~\citep{kuhn2025distributionally} or directly adding penalty terms to the objective~\citep{hastie2009elements}. For example, DRO changes the objective of~\eqref{eq:ddsp} to minimize the worst-case expected loss over the uncertainty set, thereby accounting for the misspecification error $\hat\P_n \neq \P$. Such uncertainty sets are constructed via moment constraints or statistical distances~\citep{delage2010distributionally,esfahani2018data,gao2022distributionally,duchi2019variance,lam2019recovering}. In the statistics/machine learning literature, \emph{statistics-enhanced} 
methods leverage related data or information to construct improved estimators. These methods include incorporating additional moment conditions through methods such as GMM~\citep{hansen1982large, bennett2019deep}, leveraging related data sources through transfer learning~\citep{bastani2021predicting,tian2023transfer}, and using shrinkage estimators~\citep{james1961estimation,efron1973stein,tibshirani1996regression}, with extensions to stochastic optimization~\citep{gupta2022data}. Both categories of methods have been empirically and theoretically shown to outperform plain EO in many settings. Yet, as far as we know, existing studies are case-by-case in the sense that the considered method, and the analysis to argue the advantage of the method, are specific. This motivates us to investigate a fundamental gap in understanding data-driven optimization more generally: \emph{Could we create a unified framework encompassing many existing methods, as well as new alternatives, to inform us when and by how much a method outperforms EO? Moreover, can we maximize the amount of improvement?}





\paragraph{Our Contributions.} In this paper, we address these questions via the perspective of ``directionally perturbing" EO solutions. These are data-driven methods that can be cast as essentially a perturbation of an EO solution according to well-chosen directions. We will see that this perspective encompasses many representative data-driven optimization approaches in the literature, which we collectively term \emph{EO+ methods} with solutions~$\hat\theta_{EO+}$. Building on this framework, we systematically characterize the statistical performance of EO+ solutions, which then allows us to dissect whether and when an existing method outperforms, or does not outperform EO, no matter how we ``tune" the hyperparameters. To illustrate, Example~\ref{coro:linear-symmetric} in the sequel demonstrates that DRO using a $\chi^2$-distance-based uncertainty set, a popular DRO formulation, can elicit significant improvements over EO for ordinary least-squares regression if the data distribution is Laplace, but not if it is normal, while for least-absolute-deviation regression the reverse holds. In the case where the method can elicit significant improvements, we will determine what uncertainty set size should be used, and in the case where the method provides no significant improvement, this lack of improvement is intrinsic regardless of how we tune the uncertainty set size.
Our EO+ framework provides a principled machinery to answer questions like these. It deduces, for a given problem instance, whether a data-driven optimization method is worth pursuing, or whether the method is unlikely to be fruitful. Moreover, 
we go one step further by providing a principled methodology to improve these methods and maximize the first-order statistical gain whenever possible. In many cases, our resulting enhancement approach outperforms the recommendations from representative data-driven optimization methods.



In more detail, we measure statistical improvement of a data-driven solution $\theta$ by the impact on the excess risk:
\begin{equation}\label{eq:excess-risk}
    R(\theta) = Z(\theta) - Z(\theta^*),~\text{where}~\theta^* \in \argmin_{\theta \in \Theta} Z(\theta).
\end{equation}
Throughout the main body of the paper, our comparison centers on the expected excess risk, contrasting $\E_{\Dscr_n}[R(\hat\theta_{EO+})]$ with $\E_{\Dscr_n}[R(\hat\theta_{EO})]$ and, more specifically, whether $\hat\theta_{EO+}$ provides improvements over $\hat\theta_{EO}$ in the first-order sense that $\E_{\Dscr_n}[R(\hat\theta_{EO+})]/\E_{\Dscr_n}[R(\hat\theta_{EO})]$ converges to a constant smaller than 1 as the sample size $n$ grows.
In this regard, this paper makes the following theoretical contributions:

\noindent\emph{1. General Theory on Statistical Performance of EO+ Methods.} We exactly characterize the conditions to attain first-order improvements. Specifically, we show that EO+ methods yield first-order improvements if and only if the side information is both correctly specified and geometrically effective and the perturbation magnitude is of constant order (in contrast to vanishing as sample size grows). If any of these conditions fail, any performance gains are at best higher-order and become negligible relative to EO. This result signifies a {no-free-lunch principle} in data-driven optimization: without any additional information, EO cannot be improved in the first-order sense asymptotically. On the other hand, our result also reveals two critical implications for achieving improvement over EO: First, the side information needs to be \emph{correct} and satisfy what we will call the \emph{non-orthogonality} geometric condition. Second, effective leveraging of this side information can only be done by using a perturbation magnitude that is of constant order. This latter calibration regime is \emph{distinct} from all existing suggestions in the data-driven optimization literature, which has predominantly advocated for the use of shrinking perturbations (when translated to specific hyperparameters such as the uncertainty set size in DRO). To the best of our knowledge, our EO+ framework that compares a broad class of solutions for general cost functions in a unified manner, and more importantly the derived insights described above on how improvements over EO can or cannot be achieved, are the first in the literature.



    
    
\noindent\emph{2. Principled Framework to Maximize Improvements.} When correct and geometrically effective side information is available, we develop a principled methodology to exploit it and design the optimally perturbed solution with the largest possible first-order improvement over EO. Our methodology is, in particular, capable of creating new solutions not previously studied in the literature that outperform known methods. 
   Specifically, we first characterize the perturbation that yields the largest first-order improvement based on an optimization problem imposed on the leading terms of the excess risk, and recover the approximate optimal solution through suitable estimation procedures.
    Depending on the curvature of the cost function and the richness of the available side information, we use either an analytical approach that exploits system structure or bootstrap resampling, which is often the more practical choice. Intriguingly, our methodology to achieve the best improvement over EO mirrors the concept of control variates in the Monte Carlo literature, which is a variance reduction technique designed to leverage auxiliary simulation outputs correlated with the target counterparts. 
    Unlike standard Monte Carlo control-variate methods, our procedure estimates and maximizes the variance reduction through data-driven excess-risk estimation, yet the concepts are analogous and create a novel bridge between the data-driven optimization and the Monte Carlo literature.
    
\noindent\emph{3. Generalization to Contextual Stochastic Optimization.} Our first-order improvement framework extends naturally to so-called contextual stochastic optimization problems~\citep{bertsimas2020predictive,sadana2025survey} where some observed features influence the distribution of the randomness. 
    An established approach in contextual stochastic optimization is to apply weights based on features to each individual loss in the empirical objective. As in the non-contextual setting, we characterize when directionally perturbed weighted EO solutions can achieve first-order improvements, and construct a methodology to maximize such improvements from side information. Notably, in the contextual setting, the excess risk of weighted EO solutions decays more slowly due to the nonparametric nature of the weights in approximating the conditional distribution, making first-order improvements even more significant than the non-contextual counterpart.

Finally, we validate our theoretical findings through extensive numerical studies on linear regression and contextual newsvendor problems. The results provide evidence consistent with the theoretical predictions. In the linear regression experiments, we reproduce the predicted problem-dependent improvement patterns; in the contextual newsvendor experiments, our proposed directionally perturbed weighted EO solution significantly reduces cost in 8 of the 9 dataset--estimator combinations across real-world datasets.


\subsection{Other Related Literature}
Our principled framework for performance improvement is most related to two streams of work in the data-driven optimization literature: rigorous comparison of existing EO+ solutions, and the design of better methods motivated by the limited gains existing EO+ solutions offer.

\paragraph{Performance Comparison of Existing EO+ Methods.} Rigorous comparisons between EO+ methods and the base EO solution have been studied alongside the development of these methods. DRO methods, for instance, can be interpreted as regularized empirical optimization objectives~\citep{lam2018sensitivity,duchi2019variance,gao2022finite,wang2024regularization} or as worst-case performance certificates~\citep{delage2010distributionally,esfahani2018data}. Recently, \cite{van2020data,sutter2020general} showed that some variants of DRO achieve the optimal Pareto frontier between excess risk and out-of-sample disappointment. Statistics-enhanced methods, including transfer learning~\citep{bastani2021predicting,tian2023transfer} and semi-supervised learning~\citep{chakrabortty2018efficient,song2024general}, reduce parametric estimation errors under particular structural assumptions (e.g., sparsity or additional information structure). However, estimation error is not equivalent to excess risk, and analyses of this kind do not extend readily to other EO+ methods such as DRO.

The broader landscape of improvements offered by existing EO+ methods is more nuanced when measured in terms of excess risk~\eqref{eq:excess-risk}. 
On the negative side, \cite{lam2021impossibility} establishes an impossibility result: the EO solution stochastically dominates a broad class of data-driven methods, including optimization-enhanced ones, in terms of the excess risk asymptotically. Consistent with this, \cite{gotoh2021calibration,gotoh2023data} analyze the expected cost improvement of $f$-divergence DRO and its optimistic variants, quantifying their trade-offs with the EO solution, and showing that improvements materialize only in a higher-order sense.  Similar findings have been established for other distance-based DRO methods~\citep{anderson2022improving}. Empirically, \cite{gulrajanisearch} (and \cite{wang2023rethinking}) find limited performance gains for these existing transfer learning methods (and DRO methods, respectively), relative to the base EO solution on practical benchmarks.
\paragraph{Design of New Data-Driven Methods.} Motivated by the limited gains of well-established EO+ methods, recent work has sought to design new approaches through direct adjustments to the base data-driven solution. Most closely related to ours is Operational Data Analytics (ODA), developed by \cite{feng2023framework,feng2025contextual}, which derives superior solutions by introducing additional parameters selected via a validation model, so-called ODA boosting. Their framework is shown to uniformly dominate the base data-driven solution. However, it assumes distributional knowledge of the underlying uncertainty, which is different from, and can be more restrictive than the correct side information assumption we adopt. Similarly, \cite{albert2025post} develop a shrinkage adjustment to the standard estimate-then-optimize solution in pricing problems. Both approaches produce meaningful gains in small-sample regimes, but their benefits diminish as the sample size grows. To date, no systematic data-integration framework has been shown to rigorously yield significant improvements in expected cost performance in the asymptotic sense.

Compared to these two complementary streams of work, our paper presents a unified framework that reconciles both views. We provide explicit conditions characterizing when and by how much the base data-driven solution can be improved, and develop a principled framework to maximize improvements over the EO solution in the standard setting and the weighted EO solution in the contextual optimization setting.


In the following, \Cref{sec:preliminary} presents our main comparison setup and definitions of improvement. \Cref{sec:perturb-stat} presents the concept of EO+ methods as directionally perturbed EO solutions and characterizes when and how much they improve over EO. This section also demonstrates how EO+ methods recover many existing data-driven optimization methods. Then, in \Cref{sec:perturb-opt}, we describe our framework to maximize first-order improvements and connect our results with control variates. \Cref{sec:cso} presents generalizations to the contextual stochastic optimization problem. In \Cref{sec:discuss}, we provide further discussions on identifying side information in practice and other generalizations of our framework. Finally, \Cref{sec:numeric} presents numerical results that illustrate our theoretical findings and discusses practical implications. 
Additional theoretical and numerical results are deferred to Appendices~\ref{app:unification}--\ref{app:numerical0}.

\paragraph{Notation.} $\Pi_{\Theta}(\cdot)$ is the projection operator. For a function $f(\theta)$, $\nabla_{\theta}f(\theta)$ denotes the gradient
(or subgradient) and $\nabla_{\theta}^2f(\theta)$ denotes the Hessian
matrix with respect to $\theta$. For a matrix $H$, $\|H\|$ denotes its matrix operator norm and $H^{\dagger}$ denotes the Moore-Penrose pseudoinverse. ``$\Rightarrow$'' denotes convergence in distribution, and ``$\overset{p}{\to}$'' denotes convergence in probability.  $\tr[{\cdot}]$ denotes the trace of a
matrix and $D_{\cdot}$ denotes the dimension of a variable $\cdot$. $\mathbf{1}_{\cdot}$ denotes the indicator function, $I_{D\times D}$ is the $D\times D$ identity matrix and $\mathbf{0}$ denotes the zero vector or matrix. 
$\hat{\P}_n:=\frac{1}{n}\sum_{i = 1}^n\delta_{\xi_i}$ denotes the empirical distribution of $\{\xi_i\}_{i = 1}^n$, where $\delta_{\xi}$ denotes the Dirac measure centered at $\xi$. 
$o(\cdot), O(\cdot), \Omega(\cdot), \Theta(\cdot)$ are used to describe the asymptotic performance of a function.
Specifically, given two functions $f(n)$ and $g(n)$, $g(n) = \Omega(f(n))$ (or $f(n) = O(g(n))$) if there exist positive constants $c_1$ and $n_0$ such that $\|g(n)\| \geq c_1 \|f(n)\|$ for all $n \geq n_0$; $g(n) = \Theta(f(n))$ if $g(n) = \Omega(f(n))$ and $f(n) = \Omega(g(n))$; $f(n) = o(g(n))$ if for any positive constant $\varepsilon$, there exists $n_0$ such that $\|f(n)\| \leq \varepsilon \|g(n)\|$ for all $n \geq n_0$. 
$o_p(\cdot), O_p(\cdot)$ are used to describe the asymptotic convergence of random variables. For a sequence of random variables $X_n$, $X_n = o_p(f(n))$ if $\|X_n / f(n)\| \convp 0$ as $n \to \infty$; $X_n = O_p(f(n))$ if for any $\epsilon > 0$, there exists a finite $M > 0$ and $n_0 > 0$ such that $P(\|X_n / f(n)\| \geq M) \leq \epsilon, \forall n \geq n_0$.
\section{Statistical Improvement Orders and Comparison Framework}\label{sec:preliminary}
We begin by introducing the notion of improvement orders relative to the EO solution.
\begin{definition}[Orders of Improvements]\label{defn:improvement-order}
Suppose that for an EO solution $\hat\theta_{EO}$, its excess risk satisfies:
\begin{equation}\label{eq:regret-eo}
\E_{\Dscr_n}[R(\hat\theta_{EO})] = \frac{C_{EO}}{n^{2\gamma_{EO}}}  + o({n^{-2\gamma_{EO}}}),\label{excess risk EO}
\end{equation}
where $\gamma_{EO} > 0$ is the rate exponent and $C_{EO} > 0$ is the leading constant of the first-order term in the right-hand side of \eqref{excess risk EO}.

Then, for an EO+ solution $\hat\theta_{EO+}$, its excess risk satisfies:
\begin{equation*}
\E_{\Dscr_n}[R(\hat\theta_{EO+})] = \frac{C_{EO+}}{n^{2\gamma_{EO+}}}  + o({n^{-2\gamma_{EO+}}})
\end{equation*}
with corresponding rate exponent $\gamma_{EO+}$ and leading constant $C_{EO+}$. We classify its statistical improvement as follows: 
\begin{enumerate}[label=(\alph*)]
    \item \emph{First-order:} We say $\hat\theta_{EO+}$ achieves a \emph{first-order} improvement over $\hat\theta_{EO}$ if $\gamma_{EO+} = \gamma_{EO}$ and $C_{EO+} < C_{EO}$. Moreover, two EO+ solutions, say $\hat\theta_{EO+^{(1)}}$ and $\hat\theta_{EO+^{(2)}}$, are said to achieve the same amount of first-order improvement if $C_{EO+^{(1)}} = C_{EO+^{(2)}}$.
    \item \emph{Higher-order:} We say $\hat\theta_{EO+}$ achieves a \emph{higher-order} improvement over $\hat\theta_{EO}$ if $C_{EO+} = C_{EO}$ and $\gamma_{EO+} = \gamma_{EO}$, but $\E_{\Dscr_n}[R(\hat\theta_{EO+})] < \E_{\Dscr_n}[R(\hat\theta_{EO})]$ for all sufficiently large $n$.
\end{enumerate}
\end{definition}

Definition \ref{defn:improvement-order} stipulates that, to attain a first-order improvement, the EO+ solution $\hat\theta_{EO+}$ reduces the leading term of the excess risk of the EO solution $\hat\theta_{EO}$, regarding the coefficient in front of the rate $n^{-2\gamma_{EO}}$. Under the two excess risk expansions in \Cref{defn:improvement-order}, the definition of statistical improvements can be written in terms of an equivalent, ratio-based characterization of the two improvement orders:
\begin{enumerate}[label=(\alph*)]
     \item A first-order improvement holds if and only if
    $$\lim_{n \to \infty}\frac{\E_{\Dscr_n}[R(\hat\theta_{EO}) - {R(\hat\theta_{EO+})}]}{\E_{\Dscr_n}[{R(\hat\theta_{EO})}]} \in (0, 1);$$
    \item A higher-order improvement holds if and only if
    $$\lim_{n \to \infty}\frac{\E_{\Dscr_n}[R(\hat\theta_{EO}) - {R(\hat\theta_{EO+})}]}{\E_{\Dscr_n}[{R(\hat\theta_{EO})}]} = 0$$
    and $\E_{\Dscr_n}[R(\hat\theta_{EO}) - {R(\hat\theta_{EO+})}] > 0$ holds for all sufficiently large $n$. 
\end{enumerate}

Looking at the above definitions, it is apparent that there are two other possible cases in comparing $\hat\theta_{EO+}$ with $\hat\theta_{EO}$ that have not been covered. One is the zero-order improvement regime, in which an EO+ solution attains a strictly faster excess risk rate, i.e.,
$\gamma_{EO+} > \gamma_{EO}$~\citep{audibert2007fast,hu2022fast}. However, under the regularity assumptions we will impose (which are arguably standard), this case is ruled out in our comparisons. Another case is when $\hat\theta_{EO+}$ has no improvement (including underperformance) over $\hat\theta_{EO}$, which means $\E_{\Dscr_n}[R(\hat\theta_{EO}) - {R(\hat\theta_{EO+})}] \leq 0$ eventually as $n\to\infty$. While this case is clearly not of our focus, we will reveal in our subsequent developments how this non-improvement is avoided under specific hyperparameter choices in the EO+ methods. 

We also mention that, as a special case of higher-order improvements in Definition \ref{defn:improvement-order}, $\hat\theta_{EO+}$ has a \emph{second-order} improvement over $\hat\theta_{EO}$ if there is a common scale $\zeta_{EO} > 0$ such that the excess risks for $\hat\theta_{EO}$ and $\hat\theta_{EO+}$ can be written as 
\begin{equation}
\E_{\Dscr_n}[R(\hat\theta_{EO})] = \frac{C_{EO}}{n^{2\gamma_{EO}}} + \frac{E_{EO}}{n^{2(\gamma_{EO}+\zeta_{EO})}}  + o(n^{-2(\gamma_{EO} + \zeta_{EO})}),\label{excess risk EO second order}
\end{equation}
$$\E_{\Dscr_n}[R(\hat\theta_{EO+})] = \frac{C_{EO}}{n^{2\gamma_{EO}}} + \frac{E_{EO+}}{n^{2(\gamma_{EO}+\zeta_{EO})}}  + o(n^{-2(\gamma_{EO} + \zeta_{EO})}),$$
respectively, and $E_{EO+} < E_{EO}$. That is, in the second-order improvement, $\hat\theta_{EO+}$ cannot improve the leading term in the excess risk of $\hat\theta_{EO}$, but it improves the next term. Equivalently, $\E_{\Dscr_n}[R(\hat\theta_{EO}) - R(\hat\theta_{EO+})] = o(n^{-2\gamma_{EO}}) = (E_{EO} - E_{EO+})/n^{2(\gamma_{EO} + \zeta_{EO})} + o(n^{-2(\gamma_{EO} + \zeta_{EO})})$.

Next, we describe some regularity assumptions required to establish the excess risk behavior of the EO solution in \eqref{excess risk EO}:


\begin{assumption}[Optimality Condition for Oracle Optimal Solution]\label{asp:theta-star0}
The minimizer $\theta^*$ for $Z(\theta)$ is unique and satisfies the second-order optimality conditions: the gradient $\nabla_{\theta} Z(\theta) = 0$, and the Hessian $I_{\ell}(\theta^*):=\nabla_{\theta}^2 Z(\theta^*)$ is positive definite.
\end{assumption}

\begin{assumption}[Regularity Condition of Cost Function and EO Solution]\label{asp:optimal-condition}
$\ell(\theta;\xi)$ is differentiable at $\theta^*$ for almost every $\xi$ and satisfies $\forall \theta_1, \theta_2 \in \Theta$, $|\ell(\theta_1, \xi) - \ell(\theta_2, \xi)|\leq L(\xi)\|\theta_1 - \theta_2\|$ for some function $L(\cdot)$ with $\E_{\P}[L^3(\xi)]<\infty$. The EO solution $\hat\theta_{EO}$ satisfies $\E_{\hat\P_n}[\ell(\hat\theta_{EO};\xi)] \leq \min_{\theta\in\Theta}\E_{\hat\P_n}[\ell(\theta;\xi)] + o_p\Para{n^{-1}}.$
\end{assumption}
Assumption~\ref{asp:theta-star0} is the standard sufficient optimality condition for unconstrained problems or when $\theta^*$ lies in the interior of $\Theta$. It is adopted here primarily for clarity, and we discuss the constrained case in Section~\ref{subsec:generalization}. Assumption~\ref{asp:optimal-condition} is a standard condition ensuring the consistency of $\hat\theta_{EO}$, while the third moment condition leads to uniform integrability of the estimation error. Assumptions~\ref{asp:theta-star0} and~\ref{asp:optimal-condition} provide the following expansion of the EO solution, and its excess risk behavior depicted in \eqref{excess risk EO}:
\begin{proposition}[Convergence Rate under the EO Solution]\label{prop:eo-conv}
    Suppose Assumptions~\ref{asp:theta-star0} and~\ref{asp:optimal-condition} hold. Then     
     $\hat\theta_{EO}-\theta^* = (1/n)\sum_{i=1}^n \IFx(\xi_i) + o_p\Para{n^{-1/2}}$, where \(\IFx(\xi) :=
-(\nabla_{\theta}^2 Z(\theta^*))^{-1}\nabla_{\theta}\ell(\theta^*;\xi)\). Consequently, the excess risk of $\hat\theta_{EO}$ is given by \eqref{excess risk EO} where $\gamma_{EO} = 1/2$. 

Suppose furthermore \Cref{asp:regular-risk-rate} in Appendix~\ref{app:eo-conv} holds. Then $\zeta_{EO} = 1/2$ in \eqref{excess risk EO second order}.
\end{proposition}

In Proposition \ref{prop:eo-conv}, $\IFx$ is the so-called \emph{influence function}~\citep{hampel1974influence} that signifies the functional derivative of $\hat\theta_{EO}$ with respect to the data distribution. The expansion of $\hat\theta_{EO}-\theta^*$ in terms of this influence function in Proposition \ref{prop:eo-conv} follows from standard M-estimation theory (e.g., Chapter 5 of \cite{van2000asymptotic}). To translate this expansion into the excess risk's expansion \eqref{excess risk EO} with rate exponent $\gamma_{EO} = \frac{1}{2}$, the key observation is a Taylor series argument applied to $Z(\cdot)$ that states
\begin{equation}\label{eq:second-order-expansion}
  Z(\hat\theta_{EO}) - Z(\theta^*) \approx \frac{1}{2}(\hat\theta_{EO} - \theta^*)^{\top} \nabla_{\theta}^2 Z(\theta^*)(\hat\theta_{EO} - \theta^*).  
\end{equation}
This expansion is quadratic in $\hat\theta_{EO} - \theta^*$ 
due to the first-order optimality condition in Assumption \ref{asp:theta-star0}. Therefore, 
$Z(\hat\theta_{EO}) - Z(\theta^*)$ is of order $n^{-1}$ which leads to the rate exponent $\gamma_{EO}=1/2$ in \eqref{excess risk EO}.

We illustrate the excess risk behavior in Definition \ref{defn:improvement-order} with the following simple examples:
\begin{example}[Linear Regression]\label{ex:lr}
    Consider the linear regression model $\xi = (X, Y)^{\top}$ and the ordinary least squares (OLS) loss $\ell(\theta;\xi)=(\theta^{\top}X-Y)^2$, where \(Y=(\theta^*)^{\top}X+\epsilon\) with $\E[\epsilon] = 0$ and $\Var[\epsilon] = \sigma^2$. Recall that $D_{\theta}$ denotes the dimension of $\theta$. Then $C_{EO} = D_{\theta}\sigma^2$ and $\gamma_{EO} = 1/2$ in \eqref{eq:regret-eo}. 
\end{example}

\begin{example}[Newsvendor Problem]\label{ex:newsvendor}
    Consider the single-product newsvendor problem $\ell(\theta;\xi) = c\theta - p \min\{\theta, \xi\}$ where $\xi$ denotes the product demand, $c$ and $p$ denote the cost and price parameters. Suppose $\xi \sim \P$ is continuous with density $f(\cdot)$. Then $C_{EO} = \frac{c(p-c)}{2pf(\theta^*)}, \gamma_{EO} = 1/2$ in~\eqref{eq:regret-eo}.
\end{example}

In the contextual optimization problem, the convergence rate of $\hat\theta_{EO}$ may differ depending on how we derive weights nonparametrically, and therefore other values of $\gamma_{EO}$ can arise. We will provide details in \Cref{sec:cso}.

\section{EO+ Solutions: Improvement Characterization and Scope of Applicability}\label{sec:perturb-stat}
We begin by introducing what we call directionally perturbed EO solutions that serve as the beginning point of our analysis.
\begin{definition}[Directionally Perturbed EO]\label{defn:direction-perturb}
A directionally perturbed EO solution is defined by
\begin{equation}\label{eq:empirical-solution-adjustment}
    \hat\theta_{H, M} = \Pi_{\Theta}\Para{\hat\theta_{EO} + \frac{H}{n}\sum_{i = 1}^n M(\hat\theta_{EO};\xi_i)},
\end{equation}  
where $M(\theta;\xi) \in \R^{D_M}$ is a perturbation function that is differentiable with respect to $\theta$ almost everywhere for each $\xi$, and $H \in \R^{D_{\theta}\times D_M}$ is an adjustment matrix.
\end{definition}

The function $M(\theta;\xi) \in \R^{D_M}$ can be viewed as an encoding of certain available \textit{side information}, as we will make clear in the sequel. For now, building on the setup of \Cref{sec:preliminary}, we establish a main result characterizing the improvement order achievable by $\hat\theta_{H, M}$.
\begin{theorem}[Characterization of Improvement Order]\label{thm:improvement-principle}
    Suppose Assumptions~\ref{asp:theta-star0} and~\ref{asp:optimal-condition} hold.  Denote $\Gamma:=\text{Cov}_{\P}\Paran{\IFx(\xi),M(\theta^*;\xi)}, \Sigma_0 := \E_{\P}[\IFx(\xi) \IFx(\xi)^{\top}]$ and $\mu_M(\theta) := \E_{\P}[M(\theta;\xi)]$. Then if:
    \begin{enumerate}
        \item the side information is correct: $\mu_M(\theta^*) = \mathbf{0}$;
        \item {the non-orthogonality condition} holds: $\Sigma_0 \nabla_{\theta}\mu_M(\theta^*)^{\top} + \Gamma \neq \mathbf{0}$;
    \end{enumerate}
    then there exists some $H$ with $\|H\| = \Theta(1)$ such that 
    $\hat\theta_{H, M}$ achieves a first-order improvement over $\hat\theta_{EO}$, whereas any $H$ with $\|H\| = o(1)$ can achieve at most a second-order improvement.
    
    If Condition 1 fails and $(\Sigma_0 \nabla_{\theta}\mu_M(\theta^*)^{\top} + \Gamma)(I_{D_M\times D_M} - \mu_M(\theta^*)\mu_M(\theta^*)^{\top}/\|\mu_M(\theta^*)\|_2^2) \neq \mathbf{0}$, then $\hat\theta_{H,M}$ achieves a first-order improvement over $\hat\theta_{EO}$ only if $\|H\| = \Theta(1)$ and $H\mu_M(\theta^*) = \mathbf{0}$. 
    
    If Condition 2 fails, then $\hat\theta_{H, M}$ has at most a second-order improvement over $\hat\theta_{EO}$ and the maximal improvement is attained when $\|H\| = \Theta(1/n)$.

\end{theorem}
\Cref{thm:improvement-principle} shows that a first-order improvement requires three conditions to hold simultaneously: (i) correctness of the side information; (ii) the non-orthogonality condition; (iii) a perturbation of significant magnitude. Together, these three conditions rule out degenerate choices of $M$, including cases in which $M$ is a constant in both $\theta$ and $\xi$, or in which $M(\theta;\xi)$ is independent of $\xi$ and $\nabla_{\theta}\mu_M(\theta^*) = 0$. On the other hand, the theorem also states that the combination of conditions $\mu_M(\theta^*)\neq \mathbf{0}$ and $H\mu_M(\theta^*) = \mathbf{0}$ can also lead to a first-order improvement. However, this combined requirement cannot be satisfied in many EO+ methods (we revisit the requirements of $M$ in \Cref{subsec:perturb-stat-connection}). Moreover, it cannot be attained in a fully data-driven way. Appendix~\ref{app:h-correct-impossible} explains this difficulty in detail.


Since Theorem \ref{thm:improvement-principle} provides an important foundation for our framework, we explain its proof idea here, which reveals critical tradeoffs among the bias and variance components. We consider the special case with $\Theta = \R^{D_{\theta}}$ and $D_M = D_{\theta}$. Moreover, we restrict the form of the adjustment matrix $H = \lambda I_{D_{\theta}\times D_{\theta}}$ and only allow changes in $\lambda$. In this case, $H\mu_M(\theta^*)=\mathbf{0}$ if and only if $\mu_M(\theta^*)=\mathbf{0}$. Our analysis extends naturally to general $H$; see Appendix~\ref{app:improvement}. Using \eqref{eq:empirical-solution-adjustment} and Proposition \ref{prop:eo-conv}, we can decompose $\hat\theta_{H, M}$ as:
\begin{equation}
\begin{aligned}
\hat\theta_{H, M} & = \hat\theta_{EO} + \frac{\lambda}{n}\sum_{i = 1}^n M(\hat\theta_{EO};\xi_i)   \\ 
& = \theta^* + \frac{1}{n}\sum_{i  =1}^n \IFx(\xi_i) + \lambda \mu_M(\theta^*) + \frac{\lambda}{n} \sum_{i=1}^n \widetilde M(\theta^*;\xi_i)
+o_p(n^{-1/2}), 
\end{aligned}\label{expansion decomposition}
\end{equation}
where $\widetilde M(\theta^*;\xi) = M(\theta^*;\xi)-\mu_M(\theta^*)+\nabla_\theta\mu_M(\theta^*)\IFx(\xi)$.  
The second equality in~\eqref{expansion decomposition} follows from a first-order Taylor expansion of $M(\hat\theta_{EO};\xi)$ around $\theta^*$ and the convergence $(1/n)\sum_{i=1}^n\nabla_{\theta} M(\theta^*;\xi_i) \convp \nabla_{\theta}\mu_M(\theta^*)$:
\begin{equation}\label{eq:first-order-fluctuation0}
\begin{aligned}
\frac1n\sum_{i=1}^n M(\hat\theta_{EO};\xi_i)
& =\frac1n\sum_{i =1}^n M(\theta^*;\xi_i) + \Para{\frac1n\sum_{i=1}^n\nabla_{\theta} M(\theta^*;\xi_i)}\Para{\frac{1}{n}\sum_{i=1}^n\IFx(\xi_i)} + o_p(n^{-1/2})\\
& = \mu_M(\theta^*)+\frac1n\sum_{i=1}^n
\Big(M(\theta^*;\xi_i)-\mu_M(\theta^*)+\nabla_\theta\mu_M(\theta^*)\IFx(\xi_i)\Big)
+o_p(n^{-1/2}).
\end{aligned}
\end{equation}
As $\hat\theta_{EO} = \theta^* + \frac{1}{n}\sum_{i = 1}^n\IFx(\xi_i) + o_p(1/\sqrt{n})$, the right-hand side of~\eqref{expansion decomposition} separates the perturbation on the EO solution into two components: $\mu_M(\theta^*) = \E_{\P}[M(\theta^*;\xi)]$ captures the additional deterministic bias, and $\widetilde M(\theta^*;\xi)=M(\theta^*;\xi)-\mu_M(\theta^*)+\nabla_\theta\mu_M(\theta^*)\IFx(\xi)$ captures the additional variability, which has mean zero as $\mu_M(\theta^*) = \E_{\P}[M(\theta^*;\xi)]$ and $\E[\IFx(\xi)] = 0$. 

To characterize the improvements over EO, we analyze the expected excess-risk difference between $\hat\theta_{\lambda I_{D_{\theta}\times D_{\theta}}, M}$ and $\hat\theta_{EO}$, i.e., $\E[R(\hat\theta_{\lambda I_{D_{\theta}\times D_{\theta}}, M}) - R(\hat\theta_{EO})]$.
First, analogous to the expansion for EO in \eqref{eq:second-order-expansion}, we can expand the excess risk for the directionally perturbed EO solution $\hat\theta_{H,M}$ in the form
\begin{equation}\label{general expansion}
  R(\hat\theta_{H,M})=Z(\hat\theta_{H,M}) - Z(\theta^*) = \frac{1}{2}(\hat\theta_{H,M} - \theta^*)^{\top} \nabla_{\theta}^2 Z(\theta^*)(\hat\theta_{H,M} - \theta^*)+o_p(n^{-1}).
\end{equation}
The expansion above requires that $\hat\theta_{H,M}$ converges to $\theta^*$; otherwise we can readily argue that $\hat\theta_{H,M}$ cannot achieve an asymptotic improvement over EO.
Focusing on the formulation of $\hat\theta_{\lambda I_{D_{\theta}\times D_{\theta}},M}$ in \eqref{expansion decomposition}, 
we have:
\begin{equation}\label{eq:regret-optimal-h}
\begin{aligned}
    &\E_{\Dscr_n}[R(\hat\theta_{\lambda I_{D_{\theta}\times D_{\theta}}, M}) - R(\hat\theta_{EO})] \\
=& \frac{1}{2} \E_{\Dscr_n}\Paran{\bigg\|\lambda\mu_M(\theta^*)
+ \frac1n\sum_{i=1}^n\Para{\IFx(\xi_i)+\lambda\widetilde M(\theta^*;\xi_i)}\bigg\|_{I_{\ell}(\theta^*)}^2 - \bigg\|\frac{1}{n}\sum_{i = 1}^n \IFx(\xi_i)\bigg\|_{I_{\ell}(\theta^*)}^2} + o\Para{\frac{1}{n}}.   
\end{aligned}
\end{equation}

Then we expand the squared norm $\E[\|\lambda\mu_M(\theta^*)
+ \frac1n\sum_{i=1}^n(\IFx(\xi_i)+\lambda\widetilde M(\theta^*;\xi_i))\|_{I_{\ell}(\theta^*)}^2]$ and compare them with $\E[\|\frac{1}{n}\sum_{i = 1}^n \IFx(\xi_i)\|_{I_{\ell}(\theta^*)}^2]$ in~\eqref{eq:regret-optimal-h}. Since $\frac{1}{n}\sum_{i = 1}^n {\IFx(\xi_i)}$ and $\frac{1}{n}\sum_{i = 1}^n \widetilde M(\theta^*;\xi_i)$ are centered empirical averages, their cross terms with $\lambda \mu_M(\theta^*)$ vanish after taking expectation. Therefore, we have the following: 
\begin{equation}\label{eq:reg-exact-decomp}
    \E[R(\hat\theta_{\lambda I_{D_{\theta}\times D_{\theta}}, M}) - R(\hat\theta_{EO})] = \frac{\lambda^2}{2} \|\mu_M(\theta^*)\|_{I_{\ell}(\theta^*)}^2 + \frac{\rho_1 \lambda}{n} + \frac{\rho_2\lambda^2}{2n} + o\Para{\frac{1}{n}}.
\end{equation}
The first term on the right-hand side of~\eqref{eq:reg-exact-decomp} captures the impact of the perturbation bias on the excess risk, and the remaining terms on the right-hand side of~\eqref{eq:reg-exact-decomp} capture the impact of the perturbation variability on the excess risk. 
More specifically, the term $\rho_1\lambda/n$ linear in $\lambda$ captures the covariance between the original influence function and the perturbation variability, with $\rho_1 = \tr[\Cov_{\P}[\IFx(\xi), \widetilde M(\theta^*;\xi)]I_{\ell}(\theta^*)] =  \tr[(\Sigma_0\nabla_{\theta}\mu_M(\theta^*)^{\top} + \Gamma) I_{\ell}(\theta^*)]$. The term $\rho_2 \lambda^2/(2n)$ quadratic in $\lambda$ captures the variance of the perturbation variability, with $\rho_2 = \tr[\Var_{\P}[\widetilde M(\theta^*;\xi)] I_{\ell}(\theta^*)] > 0$
which is positive because $\Var_{\P}[\widetilde M(\theta^*;\xi)] \succeq 0, \Var_{\P}[\widetilde M(\theta^*;\xi)] \neq \mathbf{0}$ and $I_{\ell}(\theta^*)\succ 0$. 

Since $\E_{\Dscr_n}[R(\hat\theta_{EO})] = \Theta(1/n)$ in \Cref{prop:eo-conv}, a first-order improvement requires the excess risk difference to be negative and of order $\Theta(1/n)$.  This is possible only when the bias term vanishes, namely $\mu_M(\theta^*)=0$, and the covariance term is non-zero, namely $\rho_1\neq 0$. In that case, the optimal $\lambda^*=-\rho_1/\rho_2$ is of constant order as $\rho_2>0$, and the improvement is $\Theta(1/n)$. 

When either condition fails, no choice of $\lambda$ yields an improvement of order $\Theta(1/n)$. More precisely, there are two alternative cases: 
\begin{enumerate}
    \item If $\mu_M(\theta^*) \neq 0$, then the coefficient of $\lambda^2$ is $\Theta(1)$ and the coefficient of $\lambda$ is $\Theta(1/n)$. Then setting $\lambda^* = \Theta(1/n)$ leads to an improvement of order $\Theta(1/n^2)$ for the right-hand side of~\eqref{eq:reg-exact-decomp}, corresponding to a second-order gain; 
    \item If $\rho_1 = 0$, the right-hand side of~\eqref{eq:reg-exact-decomp} is quadratic in $\lambda$ and nonnegative, so the perturbation cannot improve EO in a first-order sense.
\end{enumerate}

\subsection{Implications for Existing Data-Driven Optimization Methods}\label{subsec:perturb-stat-connection}
In this section, we connect common data-driven optimization methods to the directionally perturbed EO solution in \Cref{defn:direction-perturb}. 
To do so, we introduce the class of what we call EO+ methods:
\begin{definition}[EO+ Methods]\label{defn:problemclass}
    Consider EO+ solutions $\hat\theta_{\lambda}$ parameterized by $\lambda$ (which may depend on $n$).
    \begin{enumerate}[(i),leftmargin=*]
        \item Optimization-enhanced methods: 
        \begin{enumerate}
            \item Explicit regularization methods, for a pre-specified function $G(\cdot): \Theta \to \R$:
        \begin{equation}\label{eq:reg-method}
        \hat\theta_{\lambda} \in \argmin_{\theta}\E_{\hat\P_n}[\ell(\theta;\xi)] + \lambda G(\theta).
            \end{equation}
            \item DRO methods, for a discrepancy measure $d$: 
        \begin{equation}\label{eq:dro-method}
        \hat\theta_{\lambda} \in \argmin_{\theta}\max_{\Q: d(\Q, \hat\P_n)\leq \lambda}\E_{\Q}[\ell(\theta;\xi)].
    \end{equation}
        \end{enumerate}
        Choices of $d$ include: (i) Wasserstein distance~\citep{villani2009optimal}; (ii) $f$-divergence~\citep{ben2013robust}:
            $d(\P, \Q) = \E_{\Q}\Paran{f\Para{{d\P}/{d\Q}}}$, which encompasses $\chi^2$-divergence~\citep{duchi2019variance} ($f(x) = (x - 1)^2$) and Kullback-Leibler divergence \citep{lam2013robust}.
        \item Statistics-enhanced methods:
        \begin{enumerate}
            \item Transfer learning: $G(\theta) = \|\theta - \theta_0\|_2^2$ in~\eqref{eq:reg-method} where $\theta_0$ is some fixed anchor parameter.
            \item Shrinkage estimator: $\hat\theta_{\lambda} = \Pi_{\Theta}\Para{\Para{1 + \frac{G(\lambda)}{\|\hat\theta_{EO}\|_2^2}}\hat\theta_{EO}}$ for some function $G(\cdot)$.
        \end{enumerate}
    \end{enumerate}
\end{definition}

The EO+ methods in Definition \ref{defn:problemclass} encompass many common data-driven optimization approaches. Next, we show that EO+ methods are asymptotically equivalent to directionally perturbed EO in~\eqref{eq:empirical-solution-adjustment}.
\begin{theorem}[Unification of EO+ Methods]\label{thm:unification}
Suppose Assumptions~\ref{asp:theta-star0},~\ref{asp:optimal-condition} and~\ref{asp:additional-summary} in Appendix~\ref{app:formulation} hold. Then for any $\lambda > 0$, there exists $r_{\lambda,n} = o_p\Para{\lambda \vee n^{-1/2}}$ such that each EO+ solution $\hat\theta_{\lambda}$ in Definition~\ref{defn:problemclass} admits the representation: 
\begin{equation}\label{eq:eo-plus-representation}
\hat\theta_{\lambda} = \Pi_{\Theta}\Para{\hat\theta_{EO} + \frac{1}{n}\sum_{i = 1}^n H(\lambda) M(\theta^*;\xi_i) + r_{\lambda, n}}. \end{equation}

\end{theorem}

That is, regardless of their specific optimization formulations, each method in \Cref{thm:unification} can be unified as a directional perturbation of the EO solution plus a negligible remainder.
Consequently, whenever $r_{\lambda, n} = o_p(n^{-1/2})$, the first-order improvement analysis developed in Theorem~\ref{thm:improvement-principle} applies directly to EO+ methods. This yields the following characterization:

\begin{corollary}[Characterization of Improvement Order for EO+ Methods]\label{coro:eo-plus-improvement}
Suppose Assumptions~\ref{asp:theta-star0},~\ref{asp:optimal-condition} and~\ref{asp:additional-summary} in Appendix~\ref{app:formulation} hold.
Under the same definitions of $\Gamma$ and $\mu_M(\theta^*)$ in \Cref{thm:improvement-principle}, for each EO+ solution $\hat\theta_{\lambda}$ in Definition~\ref{defn:problemclass} with $r_{\lambda, n} = o_p(n^{-1/2})$ in~\eqref{eq:eo-plus-representation}, if:
\begin{enumerate}
    \item the side information is correct: $\mu_M(\theta^*) = \mathbf{0}$, 
    \item {the non-orthogonality condition} holds: $\Gamma \neq \mathbf{0}$, 
\end{enumerate}
there exists some $\lambda = \Theta(1)$ such that $\|H(\lambda)\| = \Theta(1)$ and $\hat\theta_{\lambda}$ achieves a first-order improvement over $\hat\theta_{EO}$, whereas any $\lambda$ with $\lambda = o(1)$ can achieve at most a second-order improvement. 

If either condition fails, then $\hat\theta_{\lambda}$ has at most a second-order improvement over $\hat\theta_{EO}$ and the maximal improvement is attained when $\|H(\lambda)\| = \Theta(1/n)$.

\end{corollary}
Here, for all the EO+ methods considered in this paper, $H(\lambda)$ has a trivial null space whenever $\|H(\lambda)\| = \Theta(1)$, and  therefore $H(\lambda)\mu_M(\theta^*) = \mathbf{0}$ is equivalent to $\mu_M(\theta^*) = \mathbf{0}$. Also, the non-orthogonality condition takes a simpler form than in \Cref{thm:improvement-principle} because the perturbation term in~\eqref{eq:eo-plus-representation} depends on $\theta^*$ instead of $\hat\theta_{EO}$ in~\eqref{eq:empirical-solution-adjustment}. Consequently, the perturbation does not inherit the additional fluctuation arising from the estimation error of $\hat\theta_{EO}$, and the term $\Sigma_0 \nabla_{\theta}\mu_M(\theta^*)^{\top}$ disappears in the non-orthogonality condition of \Cref{coro:eo-plus-improvement}.


Corollary \ref{coro:eo-plus-improvement} provides a complete characterization of when an EO+ method can achieve a first-order improvement over EO, namely that this is possible only when the induced side information is both correct and non-orthogonal to the EO fluctuation, and the perturbation magnitude measured by $H(\lambda)$ needs to be significant. This characterization implies that, in fact, 
existing EO+ methods are unlikely to achieve first-order improvements over EO in typical scenarios. 
Below, we provide explicit expressions of $M(\cdot;\cdot)$ and $H(\cdot)$
for representative EO+ methods in \Cref{defn:problemclass}. Other examples are presented in Appendix~\ref{app:unification}. From these, we will then specialize in a linear regression to reveal the delicacy in achieving first-order improvements for existing EO+ methods.

\begin{corollary}[Side Information for $\chi^2$-DRO]\label{ex:auxiliary-robust}
    Consider the DRO methods in \Cref{defn:problemclass}(i)(b) and let $\theta_{\lambda}^*$ denote the population minimizer of the DRO formulation~\eqref{eq:dro-method} with $d$ being the $\chi^2$-divergence. Suppose $\theta_{\lambda}^* = \theta^*$. Then: 
        \begin{small}
            \begin{equation}
        \begin{aligned}
            M(\theta;\xi) &= \Para{\ell(\theta^*;\xi) - \E_{\P}[\ell(\theta^*;\xi)]}\nabla_{\theta}\ell(\theta^*;\xi)
    + \nabla_{\theta}\Cov_{\P}[\ell(\theta^*;\xi), \nabla_{\theta}\ell(\theta^*;\xi)]\IFx(\xi) \in \R^{D_{\theta}},\\
        H(\lambda) &= -\sqrt{\lambda}\Para{\sqrt{\Var_{\P}[\ell(\theta^*;\xi)]}\; I_{\ell}(\theta^*) + \sqrt{\lambda}\,\nabla_{\theta}\Cov_{\P}[\ell(\theta^*;\xi), \nabla_{\theta}\ell(\theta^*;\xi)]}^{-1} \in \R^{D_{\theta}\times D_{\theta}}, r_{\lambda,n} = o_p(n^{-1/2}).
        \end{aligned}\end{equation}\vspace{-1em}
        \end{small}
        In particular, $\|H(\lambda)\| \propto \sqrt{\lambda}$ as $\lambda \to 0$, and $\|H(\lambda)\| = O(1)$ for $\lambda = \Theta(1)$.
\end{corollary}
In Corollary \ref{ex:auxiliary-robust}, the general condition $\theta_{\lambda}^* = \theta^*$ means that at the population level, the DRO formulation preserves the minimizer of the original problem. 
This holds when the first-order optimality condition of the worst-case objective $\max_{\Q: d(\Q, \P^*)\leq \lambda}\E_{\Q}[\ell(\theta;\xi)]$ is satisfied at $\theta^*$,
and can be verified in certain settings, for example, when the residual noise in a linear regression is symmetric.  
We will elaborate it in \Cref{coro:linear-symmetric}. However, note that the condition $\theta_\lambda^*=\theta^*$ is typically not easy to satisfy with a non-zero $\lambda$ while, on the other hand, if we take $\lambda=0$ at the population level (i.e., $\lambda\to0$ as $n\to\infty$), then we cannot satisfy the condition $\|H(\lambda)\|=\Theta(1)$ in Corollary \ref{coro:eo-plus-improvement}. Because of this dilemma, it is difficult to achieve first-order improvement generally for the DRO methods depicted in Corollary \ref{ex:auxiliary-robust}.



\begin{corollary}[Side Information for Shrinkage Estimator]\label{ex:auxiliary-stat}
Consider the shrinkage estimator method in \Cref{defn:problemclass}$(ii)(b)$. We have $H(\lambda) = G(\lambda)/\|\theta^*\|_2^2, M(\theta;\xi) = \theta^* + (I_{D_{\theta}\times D_{\theta}} - 2\theta^*\theta^{*T}/\|\theta^*\|_2^2)\IFx(\xi), r_{\lambda,n} = o_p(n^{-1/2})$.
\end{corollary}




Beyond the methods in \Cref{defn:problemclass}, we present two additional data-optimization integration methods, Conditional Value-at-Risk (CVaR) and GMM, that admit representations analogous to~\eqref{eq:eo-plus-representation} and yield essentially the same first-order improvement implications under \Cref{thm:improvement-principle}. We discuss these methods in Appendices~\ref{app:cvardro} and~\ref{app:gmm}, respectively.

Finally, 
we apply Corollaries~\ref{coro:eo-plus-improvement} and~\ref{ex:auxiliary-robust} to examine two classical loss functions in linear regression across a range of noise distributions. This comparison shows when EO+ methods lead to significant improvements:
\begin{corollary}[Characterization of Improvement in Linear Regression]\label{coro:linear-symmetric}
Continuing from \Cref{ex:lr}, define the $\chi^2$-EO+ method with parameter $\lambda \in \R$ by:
\begin{equation}\label{eq:chi2-eo}
    \hat\theta_{\lambda} = \begin{cases}\argmin_{\theta} \max_{\P:\chi^2(\Q, \hat\P_n)\leq \lambda} \E_{\Q}[\ell(\theta;\xi)],~\text{if}~\lambda \geq 0;\\ \argmin_{\theta} \min_{\P:\chi^2(\Q, \hat\P_n)\leq -\lambda} \E_{\Q}[\ell(\theta;\xi)],~\text{if}~\lambda < 0.\end{cases}
\end{equation}
Table~\ref{tab:linear-symmetric} summarizes, for the OLS loss $\ell(\theta;\xi)=(\theta^{\top}X-Y)^2$ and the least absolute deviation (LAD) loss $\ell(\theta;\xi) = |\theta^{\top} X - Y|$, whether the best $\chi^2$-EO+ method achieves a first-order improvement, and the optimal order of $\lambda$.
\begin{table}[!htb]
    \centering
    \caption{First-order improvement of the best $\chi^2$-EO+ method across all choices of $\lambda$. Parentheses indicate the  optimal order of $\lambda$ in $n$.}
    \label{tab:linear-symmetric}
    \resizebox{0.9\textwidth}{!}{\begin{tabular}{c|ccccc}
    \toprule
    Loss / Noise Distribution & Uniform & Normal & Laplace &  Location-shifted Exponential$^*$ & $t$-distribution (degrees of freedom $> 4$)\\
    \midrule
     LAD & Yes ($\Theta(1)$)  & Yes ($\Theta(1)$) & No ($\Theta(1/n)$) & No ($\Theta(1/n)$) & Yes ($\Theta(1)$)\\
     OLS & Yes ($\Theta(1)$) & No ($\Theta(1/n)$) & Yes ($\Theta(1)$) & No ($\Theta(1/n)$) & Yes ($\Theta(1)$)\\
     \bottomrule
    \end{tabular}}
    
\vspace{0.1cm}
\footnotesize{$^*$: ``Location-shifted Exponential'' denotes the noise $\epsilon \overset{d}{=} \text{Exp}(r) - \frac{1}{r}$ for some parameter $r > 0$, such that $\E[\epsilon] = 0$.}
\end{table}
\vspace{-1em}
\end{corollary}
The $\chi^2$-EO+ method in~\eqref{eq:chi2-eo} encompasses the standard $\chi^2$-DRO when $\lambda \geq 0$ and the distributionally favorable optimization recently proposed by \citet{jiang2024distributionally} to mitigate the effect of endogenous outliers. 
To verify the conditions for attaining first-order improvement or not in each case, we consider the univariate special case $X = 1$. This reduces the problem to mean estimation with $\epsilon = Y - \theta^*$. Then, a first-order improvement of the $\chi^2$-EO+ method hinges on two conditions: $\E[M(\theta^*;\xi)] = 0$ and $\E[M(\theta^*;\xi)^{\top}\IFx(\xi)] \neq 0$ for the corresponding $M(\cdot;\cdot)$. For both OLS and LAD losses, the first condition is sufficient when the noise $\epsilon$ is symmetric, and thus fails for asymmetric distributions such as the location-shifted exponential. For OLS, one can compute $\IFx(\xi) = \epsilon$ and $M(\theta^*;\xi) \propto \epsilon^3 - 3\sigma^2\epsilon$, so the second condition reduces to $\E[\epsilon^4] - 3\sigma^4 \neq 0$, which fails under Gaussian noise but holds for other symmetric distributions. For LAD, the second condition reduces to $1 - 2\E[|\epsilon|]f(\theta^*) \neq 0$, which fails for Laplace noise but holds for other distributions. These two conditions together explain the pattern in Table~\ref{tab:linear-symmetric}.

\section{Maximizing First-Order Improvements}\label{sec:perturb-opt}
In this section, we go beyond existing data-driven optimization to create general methodologies that maximize first-order improvements.
Note that in the formulation of $\hat\theta_{H, M}$ in \Cref{defn:direction-perturb}, both the adjustment matrix $H$ and the side information $M$ are design choices, which we have shown to connect to EO+ methods that encompass many existing data-driven optimization approaches. At the same time, following the excess risk decomposition in \eqref{general expansion} and \eqref{eq:regret-optimal-h} for general $H$ when setting $\mu_M(\theta^*) = 0$, and therefore $\widetilde M(\theta^*;\xi) = M(\theta^*;\xi)+\nabla_\theta\mu_M(\theta^*)\IFx(\xi)$, the first-order improvement is determined entirely by the leading term in the excess risk expansion of $\hat\theta_{H,M}$, given by
\begin{equation}\label{eq:optimal-h}
\E_{\Dscr_n, \P}\Paran{\bigg\|\frac{1}{n}\sum_{i = 1}^n {\IFx(\xi_i) + \frac{H}{n}\sum_{i = 1}^n \widetilde M(\theta^*;\xi_i)}\bigg\|_{I_{\ell}(\theta^*)}^2} = \frac{1}{n}\E_{\P}\Paran{\bigg\|{\IFx(\xi) + {H} \widetilde M(\theta^*;\xi)}\bigg\|_{I_{\ell}(\theta^*)}^2}.
\end{equation}
Consequently, constructing the optimal directionally perturbed EO solution reduces to minimizing this term over all possible choices of $H$ and admissible $M$. Given a fixed function form $M$, when $\Var_{\P}[\widetilde M(\theta^*;\xi)] \succ {0}$, the minimizer $H^*$ can be written in closed form:
\begin{equation}\label{eq:optimal-close-h}
\begin{aligned}
    H^* &:= -\E_{\P}[\IFx(\xi) \widetilde M(\theta^*;\xi)^{\top}](\E_{\P}[\widetilde M(\theta^*;\xi) \widetilde M(\theta^*;\xi)^{\top}])^{\dagger}
\end{aligned}
\end{equation}
via standard quadratic optimization and noting that the $I_\ell(\theta^*)$ in the norm does not affect the optimal solution. However, the optimal $H^*$ in~\eqref{eq:optimal-close-h} is not directly accessible in practice because it depends on unknown population quantities, including the influence function $\IFx(\xi)$ and the underlying distribution $\P$. Our methodology aims to bypass this challenge by accurately estimating the required population quantities either from analytical calculations or bootstrap resampling.



We begin in \Cref{subsec:comp-h} with a given moment function $M$ that represents correct side information and satisfies the conditions in \Cref{thm:improvement-principle}. In this setting, we construct a data-driven adjustment matrix $\widehat H$ that achieves the optimal first-order improvement (among all methods parametrized by $H$). In \Cref{subsec:connect-cv}, we provide its connection with the perspective of control variates in Monte Carlo. In Appendix~\ref{subsec:comp-h-phi}, we consider the more general setting where we jointly select $\widehat H$ and $M$. There, multiple correct choices of side information $M$ may be available, and we select the one that leads to the largest improvement.


\subsection{Optimizing Adjustment Matrix Given Correct Side Information}\label{subsec:comp-h}
We first consider the setting where correct side
information is available. 
\begin{assumption}[Correct Side Information]\label{asp:single-side-information}
    There exists a function $M: \Theta \times \Xi \mapsto \R^{D_M}$ such that $\mu_M(\theta^*) = \mathbf{0}$, $\Sigma_0 \nabla_{\theta}\mu_M(\theta^*)^{\top} + \Gamma \neq \mathbf{0}$, $\Var_{\P}[\widetilde M(\theta^*;\xi)] \succ 0$, where $\mu_M(\theta) = \E_{\P}[M(\theta;\xi)]$, $\Gamma=\text{Cov}_{\P}\Paran{\IFx(\xi),M(\theta^*;\xi)}$ and $\widetilde M(\theta;\xi) = M(\theta;\xi) + \nabla_{\theta}\mu_M(\theta)\IFx(\xi)$.
    
    Furthermore, there exists a neighborhood $\Nscr$ of $\theta^*$ such that $\theta\mapsto M(\theta;\xi)$ is continuously differentiable on
$\mathcal N$ and $\E_{\P}[\sup_{\theta \in \Nscr}\paran{\|M(\theta;\xi)\|_2^3, \|\nabla_{\theta}M(\theta;\xi)\|^3}] < \infty$.
\end{assumption}
Above, the third-moment condition is imposed as a standard regularity requirement for subsequent expected excess risk comparisons.  Given such an $M$, we choose an adjustment matrix $H \in \R^{D_{\theta}\times D_{M}}$ so that the directionally perturbed EO in~\eqref{eq:empirical-solution-adjustment} attains the largest first-order improvement over $\hat\theta_{EO}$ among all possible choices of $H$. 
As mentioned above, $H^*$ in~\eqref{eq:optimal-close-h} leads to the largest first-order improvement because it minimizes the leading excess-risk term of $\hat\theta_{H, M}$, i.e., ~\eqref{eq:optimal-h}. That is:
\begin{theorem}[Optimizing Adjustment Matrix]\label{thm:opt-adjust}
Suppose Assumptions~\ref{asp:theta-star0},~\ref{asp:optimal-condition} and~\ref{asp:single-side-information} hold. Fix $B \geq \|H^*\|_{\infty}$. Define $\Hscr_B = \{H\in\R^{D_\theta\times D_M}:
\|H\|_\infty\leq B\}$. For each $n$, let $H_n\in\Hscr_B$ be an approximate finite-sample
minimizer satisfying $\E_{\Dscr_n}[R(\hat\theta_{H_n,M})] \le
\inf_{H\in\Hscr_B} \E_{\Dscr_n}[R(\hat\theta_{H,M})] +o(n^{-1})$. Then $\lim_{n \to \infty}\|H_n-H^*\|=0$.
\end{theorem}
Theorem \ref{thm:opt-adjust} states that any such finite-sample minimizer must converge to the closed-form $H^*$ in~\eqref{eq:optimal-close-h}. This is based on the argument that minimizing the expected excess risk asymptotically reduces to minimizing its leading population term~\eqref{eq:optimal-h}. 

With limited data, our high-level procedure is given as follows:
\begin{enumerate}
    \item Compute the EO solution $\hat\theta_{EO}$;
    \item Compute the best data-driven adjustment matrix $\widehat H$ and define its entrywise clipped version at threshold $B$ by $\widehat H_B$, i.e., $(\widehat H_B)_{ij} = \text{sign}(\widehat H_{ij})\min\{|\widehat H_{ij}|, B\}$;
    \item Output $\hat\theta_{\widehat H_B, M} = \Pi_{\Theta}\Para{\hat\theta_{EO} + \frac{\widehat H_B}{n}\sum_{i=1}^n M\Para{\hat\theta_{EO};\xi_i}}$.
\end{enumerate}
Here, $B > 0$ is a prespecified large constant to bound the entries of the adjustment matrix for numerical stability. To understand how to estimate $\widehat H$, we first introduce the notion of consistency:
\begin{assumption}[Estimation Consistency]\label{asp:consistent-estimate}
    $\|\widehat H - H^*\| = o_p(1)$. 
\end{assumption}
Under this assumption, we have:
\begin{theorem}[Nearly Optimizing Adjustment Matrix]\label{prop:max-improvement}
    Suppose Assumptions~\ref{asp:theta-star0},~\ref{asp:optimal-condition},~\ref{asp:single-side-information} and \ref{asp:consistent-estimate} hold. Then $\hat\theta_{\widehat H_B, M}$ achieves the same amount of first-order improvement as $\hat\theta_{H^*, M}$. 
\end{theorem}
Theorem~\ref{prop:max-improvement} stipulates that, to maximize the first-order improvement for a given $M$, we can focus on constructing
a consistent estimator of the oracle-best adjustment matrix $H^*$, and ``plug in" directly to the formula of $\hat\theta_{\widehat H,M}$. In the following, we propose two procedures to construct this $\widehat H$, one using analytical calculations based on influence-function estimates; the other relying on bootstrap resampling and avoiding explicit influence-function calculations. 
\subsubsection{Analytical Approach} When a reliable estimate of the influence function is available, we can use a direct plug-in estimator to approximate~\eqref{eq:optimal-close-h}. The procedure is described in \Cref{alg:analytical-estimate-1}. 
\begin{small}
\begin{algorithm}[!htb]
\caption{Analytical Approach to Estimate $\widehat H$}
\label{alg:analytical-estimate-1}
\small
\begin{algorithmic}[1]
    \REQUIRE The dataset $\{\xi_i\}_{i \in [n]}$, the EO solution $\hat\theta_{EO}$
    \IF{$\ell$ is twice differentiable in $\theta$}
    \STATE Compute $\hat I_{\ell}(\hat\theta_{EO}) = \frac{1}{n}\sum_{i = 1}^n\nabla_{\theta}^2 \ell(\hat\theta_{EO};\xi_i)$.
    \ELSE
    \STATE Compute a problem-specific consistent estimator $\hat I_{\ell}(\hat\theta_{EO})$.
    \ENDIF
    \STATE Compute the estimated influence function $\widehat{\IFx}(\xi):= - \hat I_{\ell}(\hat\theta_{EO})^{\dagger} \nabla_{\theta} \ell(\hat\theta_{EO};\xi)$.
    \STATE Compute $\widehat H$ by:
    \begin{equation}\label{eq:close-form-h}
    \widehat H= -\E_{\hat\P_n}[\widehat \IFx(\xi) \widehat M(\hat\theta_{EO};\xi)^{\top}]\Para{\E_{\hat\P_n}[\widehat M(\hat\theta_{EO};\xi) \widehat M(\hat\theta_{EO};\xi)^{\top}]}^{\dagger},    
    \end{equation}
    where $\widehat M(\theta;\xi) = M(\theta;\xi) + \Para{\frac{1}{n}\sum_{i = 1}^n \nabla_{\theta} M(\theta;\xi_i)} \widehat\IFx(\xi)$.
    \ENSURE $\widehat H$.
\end{algorithmic}
\end{algorithm}
\end{small}

\Cref{alg:analytical-estimate-1} estimates the optimal perturbation matrix $H^*$ by replacing the population quantities in the closed-form expression in~\eqref{eq:optimal-close-h} with their empirical counterparts. The main challenge is to estimate the Hessian matrix $I_{\ell}(\theta^*)$, which is needed to estimate the influence function $\IFx(\xi)$. 

More specifically, we require a consistent Hessian matrix estimator $\hat I_{\ell}(\hat\theta_{EO})$ satisfying $\|\hat I_{\ell}(\hat\theta_{EO}) - I_{\ell}(\theta^*)\| = o_p(1)$. This condition holds for all twice-differentiable objectives under standard regularity conditions using the empirical Hessian estimator as in Step 2. For nonsmooth objectives, a problem-specific estimator may be available. Such estimators often exploit alternative characterizations of local curvature, as in Step 4. 

Once a consistent estimate for $I_\ell(\theta^*)$ is available, the estimated influence function and $\widehat M(\hat\theta_{EO};\xi)$ can be computed directly. Consequently, the empirical covariance quantities appearing in~\eqref{eq:close-form-h} consistently approximate their population counterparts in~\eqref{eq:optimal-close-h}.
The resulting estimator provides a consistent approximation to $H^*$ that maximizes the first-order improvement:
 \begin{proposition}[Consistency of Analytical Approach]\label{prop:analytical-consistency}
    Suppose Assumptions~\ref{asp:theta-star0},~\ref{asp:optimal-condition} and~\ref{asp:single-side-information} hold with $\|\hat I_{\ell}(\hat\theta_{EO}) - I_{\ell}(\theta^*)\|_2 = o_p(1)$.
      Then the output of \Cref{alg:analytical-estimate-1} satisfies Assumption~\ref{asp:consistent-estimate}.
\end{proposition}
For non-smooth objectives, e.g., the newsvendor problem in \Cref{ex:newsvendor}, recall that $I_{\ell}(\theta^*) = p f(\theta^*)$, and we estimate it by $\hat I_{\ell}(\hat\theta_{EO}) = p\hat f(\hat\theta_{EO})$, where $\hat f(\cdot)$ is a kernel density estimator. Under standard regularity conditions, $\|\hat I_{\ell}(\hat\theta_{EO}) - I_{\ell}(\theta^*)\| = o_p(1)$ (cf. Chapter 1.2 in~\citet{tsybakov2008introduction}).


\subsubsection{Bootstrap Approach} \label{sec:bootstrap}
To avoid explicit influence function calculations, we propose a bootstrap-based approach that leverages resampling to construct a data-driven estimator.  The procedure is described in \Cref{alg:bootstrap-estimate-1}.

\begin{small}
\begin{algorithm}[!htb]
\caption{Bootstrap Approach to Estimate $\widehat H$ 
}
\label{alg:bootstrap-estimate-1}
\small
\begin{algorithmic}[1]
    \REQUIRE Number of bootstrap replications $B$, the dataset $\{\xi_i\}_{i \in [n]}$
    \FOR{$b = 1,\ldots, B$}
    \STATE Obtain $\Dscr_n^{(b)} = \{\xi_i^{(b)}\}_{i \in [n]}$ by resampling the dataset with replacement.
    \STATE Compute the bootstrapped EO solution $\hat\theta_{EO}^{(b)} \in \argmin_{\theta} \sum_{i \in [n]}\ell(\theta;\xi_i^{(b)})$.
    \STATE Compute the bootstrapped side information $\widebar M^{(b)} = \frac{1}{n}\sum_{i = 1}^n M\Para{\hat\theta_{EO}^{(b)};\xi_i^{(b)}}$.
    \ENDFOR
    \STATE Compute $\widehat{H}$ by:
    \begin{equation}\label{eq:close-form-h-bootstrap}
        \widehat H = -\widehat{\Cov}_{B}(\hat\theta_{EO}^{(b)}, \widebar M^{(b)})\,\widehat{\Var}_{B}(\widebar M^{(b)})^{\dagger},
    \end{equation}
    where $\widehat{\Var}_{B}$ and $\widehat{\Cov}_B$ are computed over the bootstrap samples indexed by $b \in [B]$.
    \ENSURE $\widehat H$.
\end{algorithmic}
\end{algorithm}
\end{small}

To explain \Cref{alg:bootstrap-estimate-1}, we argue how the bootstrapped quantities in \eqref{eq:close-form-h-bootstrap} collectively approximate \eqref{eq:optimal-close-h} while, importantly, avoiding influence-function estimation. The key to the latter advantage lies in the use of $\widehat{\Cov}_{B}(\hat\theta_{EO}^{(b)}, \widebar M^{(b)})$ that involves bootstrapping $\hat\theta_{EO}^{(b)}$, which can be linearized in terms of the influence function thanks to Proposition \ref{prop:eo-conv}, i.e., $\hat\theta_{EO}\approx\theta^*+(1/n)\sum_{i=1}^n\IFx(\xi_i)$. Furthermore from \eqref{expansion decomposition}, $\widebar M(\hat\theta_{EO}):=(1/n)\sum_{i=1}^nM(\hat\theta_{EO};\xi_i)\approx(1/n)\sum_{i=1}^n\widetilde M(\theta^*;\xi_i)$. Hence
\begin{equation*}
\Cov(\hat\theta_{EO}, \widebar M(\hat\theta_{EO}))\approx\frac{1}{n}\Cov(\IFx(\xi),\widetilde M(\theta^*;\xi))=\frac{1}{n}\E[\IFx(\xi)\widetilde M(\theta^*;\xi)^\top],
\end{equation*}
where the last equality follows from $\E[\IFx(\xi)]=0$. Similarly,
$$\Var(\widebar M(\hat\theta_{EO}))\approx\Var\left(\frac{1}{n}\sum_{i=1}^n\widetilde M(\theta^*;\xi_i)\right)=\frac{1}{n}\Var(\widetilde M(\theta^*;\xi))=\frac{1}{n}\E[\widetilde M(\theta^*;\xi)\widetilde M(\theta^*;\xi)^\top],$$
where the last equality follows from $\E[\widetilde M(\theta^*;\xi)]=0$. Therefore,
$$\Cov(\hat\theta_{EO}, \widebar M(\hat\theta_{EO}))\Var(\widebar M(\hat\theta_{EO}))^{\dagger}\approx \E[\IFx(\xi)\widetilde M(\theta^*;\xi)]\E[\widetilde M(\theta^*;\xi)\widetilde M(\theta^*;\xi)^\top]^{\dagger}$$
which is precisely \eqref{eq:optimal-close-h}. In other words, we can focus on evaluating $\Cov(\hat\theta_{EO}, \widebar M(\hat\theta_{EO}))\Var(\widebar M(\hat\theta_{EO}))^{\dagger}$ to compute $\widehat H$. To this end, the bootstrap applies by replacing all empirical quantities with resampled counterparts, namely replacing $\mathbb P$ with $\hat\P_n$ and hence $\hat\P_n$ with the resampled empirical distribution. It uses the principle that the statistical fluctuation of the resampled quantities, conditional on the data, mimics that of the original empirical counterparts. This gives rise to \eqref{eq:close-form-h-bootstrap}.



Therefore, we can derive the following consistency property for the bootstrap approach.


\begin{proposition}[Consistency of Bootstrap Approach]\label{prop:bootstrap-consistency}
    Suppose Assumptions~\ref{asp:theta-star0},~\ref{asp:optimal-condition} and~\ref{asp:single-side-information} hold. If the number of bootstrap replications $B = \omega(1)$, then the output of \Cref{alg:bootstrap-estimate-1} satisfies Assumption~\ref{asp:consistent-estimate}. 
\end{proposition}
As long as the number of bootstrap replications $B$ grows sufficiently with $n$, the resulting estimator $\widehat H$ consistently approximates $H^*$ and achieves the same amount of first-order improvement.
Compared with the analytical approach, the bootstrap approach only requires repeatedly solving the EO problem on resampled datasets and evaluating the corresponding side-information statistics. As a result, it is straightforward to implement and remains applicable even when influence functions or Hessian matrices are difficult to derive, for example when the objective function is complicated or non-smooth. On the other hand, it is computationally more expensive due to the repeated EO solves.

\subsection{Connection to Control Variates}\label{subsec:connect-cv}
The previous results provide a foundation to connect with the perspective of control variates~\citep{rubinstein1985efficiency,nelson1990control, asmussen2007stochastic} in Monte Carlo. The latter utilizes auxiliary outputs, which constitute the so-called control variate, to reduce variance over naive Monte Carlo. Instead of using the sample mean for target simulation outputs alone (to estimate its population mean), we can combine these auxiliary outputs into the final estimates by using a form that resembles a regression of the target against these auxiliary outputs (centered by their known mean). The coefficient in combining the auxiliary outputs is obtained by minimizing the estimation variance. In the Monte Carlo literature, a wide range of approaches to construct control variates have been studied, including applications in classical option pricing \citep{glasserman2013monte}, regularized least squares \citep{south2022}, function approximation \citep{maire2003} and neural networks \citep{wan2020neural}.


Our approach to maximize first-order improvement has a close connection with the control variate idea, but intriguingly layered in a very different context of excess risk improvement. By the linearization property in Proposition \ref{prop:eo-conv}, $\hat\theta_{EO}$ is approximately $\theta^*+(1/n)\sum_{i=1}^n\IFx(\xi_i)$. In this regard, estimating $\theta^*$ can be viewed as approximately estimating $\E[\IFx(\xi)]$ (note that $\E[\IFx(\xi)]=0$ by construction, and thus this appears a tautology, but the perspective will be useful). That is, suppose we now want to estimate the mean $\E[\IFx(\xi)]$. If we have a control variate $\widetilde M(\theta^*;\xi)$ with a known mean 0, then \eqref{eq:optimal-h}, which is equivalent to minimizing the variance since all involved quantities there have mean 0, is exactly the mechanism to find the optimal coefficient in the control variate estimator. Correspondingly, \eqref{eq:optimal-close-h} is precisely the formula for the optimal coefficient.

In other words, the side information and the adjustment matrix in our design have a parallel relation with the auxiliary outputs and the regression coefficients in the control variate context. The root of this connection is that improving the first-order performance of EO essentially boils down to variance reduction. The decomposition in \eqref{eq:reg-exact-decomp} that underlies the proof of Theorem \ref{thm:improvement-principle} reveals that, in order to achieve first-order improvements, the only possible route is to keep the bias negligible and substantially reduce variance. That is, variance reduction is the only significant driver capable of improving first-order performance. To this end, the side information $M$ acts as an effective control variate to execute this reduction.

Despite the connection, note that since our main goal is to improve excess risk, the linearization and influence function serve only as intermediate tools. The bootstrap approach discussed in Section \ref{sec:bootstrap} makes clear that we do not necessarily need to use the influence function explicitly in executing the control variate estimation. There, we directly use the EO solution as the estimator for our target, and calibrate the optimal coefficient from it. However, since this target is now a nonlinear functional of the data distribution, we leverage the bootstrap as an effective tool to estimate the coefficient, leading to Algorithm \ref{alg:bootstrap-estimate-1}.

\section{Attaining First-Order Improvements for Contextual Stochastic Optimization}\label{sec:cso}
We extend our previous analyses in Sections~\ref{sec:perturb-stat} and~\ref{sec:perturb-opt} to contextual stochastic optimization problems~\citep{bertsimas2020predictive,elmachtoub2022smart,sadana2025survey}, where the randomness $\xi$ depends on additional covariates $u$.
More precisely, in these problems,
the distribution of $\xi$ 
depends on the covariate $u \in \R^{D_{u}}$. The ground-truth distribution $\P_{\xi|u}$ is unknown; instead, the decision-maker only has
data $\Dscr_n= \{(u_i, \xi_i)\}_{i = 1}^n$ consisting of i.i.d. samples from the joint distribution 
$\P:=\P_{(u,\xi)}$. The decision-maker observes the covariate $u = u_0$ before making the decision $\theta(u)$. The excess risk of the decision $\theta(u)$ is defined as:
\[R(\theta(u)):=Z_u(\theta(u)) - Z_u(\theta^*(u)),\] 
where $Z_u(\theta):=\E_{\xi\sim \P_{\xi|u}}[\ell(\theta;\xi)]$ and $\theta^*(u) = \argmin_{\theta} Z_u(\theta)$.

Depending on the size of the decision class, methods in the contextual optimization literature can generally be divided into parametric and nonparametric approaches. For parametric approaches (see \Cref{defn:cso-param-method} in Appendix~\ref{app:cso-parametric}), 
%
the analyses in Sections~\ref{sec:perturb-stat} and~\ref{sec:perturb-opt} apply directly via reparameterization. 
In this section, we focus on nonparametric approaches, i.e., weighted empirical optimization procedures within the class of predict-then-optimize methods introduced in \cite{bertsimas2020predictive}. For simplicity, we assume $\Theta = \R^{D_{\theta}}$ here.
\begin{definition}[Weighted EO Solution]\label{defn:np-opt0}
  The weighted EO solution is computed by:
  \begin{equation}\label{eq:np-optimizer0}
    \hat\theta_{EO}(u) \in \argmin_{\theta \in \Theta} \sum_{i \in [n]} w_{n,i}(u)\ell(\theta;\xi_i),    
  \end{equation}
  where $\{w_{n,i}(u)\}_{i \in [n]}$ are weights determined by $\Dscr_n$ and $u$. 
\end{definition}
Many weight choices in the weighted EO solutions, including kernel estimators and regular random forest estimators (described in Appendix~\ref{app:perturb-opt-cso}),
satisfy the following condition:
\begin{assumption}[Influence Function Decomposition]\label{asp:if-decomposition}
At each covariate point $u_0$, there exists some $\theta^*(u_0)$ such that, for some convergence rate $\gamma \leq \frac{1}{2}$,  the weighted EO solution $\hat\theta_{EO}(u_0)$ satisfies:
\[
\hat\theta_{EO}(u_0)-\theta^*(u_0)
= n^{-\gamma}\Para{\frac{1}{\sqrt{n}}\sum_{i=1}^n \IFx_{u_0}(u_i,\xi_i) + b_{u_0} + o_p(1)},
\]
where $\E_{(u,\xi)\sim\P}[\IFx_{u_0}(u, \xi)]=0$, $\text{Var}_{(u,\xi)\sim \P}[\IFx_{u_0}(u, \xi)] = \Theta(1)$, $b_{u_0}=O(1)$ and $\E_{\Dscr_n}[n^{2\gamma}\|\hat\theta_{EO}(u_0) - \theta^*(u_0)\|_2^2]<\infty$.
\end{assumption}

\begin{example}[Kernel Estimator]\label{ex:kernel-learner0}
Let $K:\R^{D_u} \to \R$ be a kernel with $\int K(u) du  < \infty$. Then the weights $w_{n,i}(u) = K((u_i - u)/h_n)$, with bandwidth $h_n$, satisfy Assumption~\ref{asp:if-decomposition}. Here, $\gamma = \min\paran{\frac{1}{2} + \frac{D_u}{2}\frac{\log h_n}{\log n}, -\frac{\log h_n}{\log n}}$ from \cite{iyengar2024cross}. Standard kernels include the naive kernel $K(x) = \mathbf{1}_{\|x\|_2 \leq 1}$ and Gaussian kernel $K(x) = \exp(-\|x\|^2)$. 
\end{example}

Similar to \Cref{sec:perturb-stat}, we consider the following directionally perturbed weighted EO solution:
\begin{definition}[Directionally Perturbed Weighted EO Solution]\label{defn:perturb-sol-cso}
    For a given covariate $u_0$, the directionally perturbed weighted EO solution is defined as:
    \begin{equation}\label{eq:perturb-weight-eo}
        \hat\theta_{H_n,M_{n,u_0}}(u_0):=\hat\theta_{EO}(u_0)+ \frac{H_n}{n} \sum_{i=1}^n M_{n,u_0}(u_i,\xi_i),   
    \end{equation}
    where $M_{n,u_0}(u,\xi) \in \R^{D_M}$ is a perturbation function and $H_n \in \R^{D_{\theta} \times D_M}$ is an adjustment matrix, both of which may depend on $n$.
\end{definition}
We now establish the main result characterizing the improvement order achieved by $\hat\theta_{H_n,M_{n,u_0}}(u_0)$. Here, the key distinction from the non-contextual setting is that the weighted EO solution typically converges at a slower nonparametric rate. This changes the order of performance improvements, although the main idea in \Cref{sec:perturb-stat} still applies. 
The following proposition formalizes the correctness and non-orthogonality conditions that govern first-order improvements in the contextual setting.

\begin{theorem}[Characterization of Improvement Order for Weighted EO Solution]\label{prop:contextual-improvement}
For a given covariate $u_0$, suppose Assumptions~\ref{asp:theta-star0},~\ref{asp:optimal-condition} (when we replace $Z(\theta)$ with $Z_{u_0}(\theta)$) and Assumption~\ref{asp:if-decomposition} hold. Consider a sequence of perturbation functions $\{M_{n,u_0}(u,\xi)\}_{n =1}^{\infty}$ with $\Omega_n(u_0):=\Var_{\P}[M_{n,u_0}(u,\xi)]$ and $\Omega(u_0):= \lim_{n \to \infty} n^{1-2\gamma}\Omega_n(u_0) \succ 0$. For the directionally perturbed weighted EO solution $\hat\theta_{H_n,M_{n,u_0}}(u_0)$ in~\Cref{defn:perturb-sol-cso}, if:
\begin{enumerate}
\item the side information is ``almost'' correct: $\mu_{n}(u_0):=\E_{\P}[M_{n,u_0}(u,\xi)] = o(n^{\gamma - 1})$,
\item the non-orthogonality condition holds: $\Gamma_{u_0}:=\lim_{n \to \infty}
n^{1/2-\gamma}\Cov_{\P}\!\left[
\IFx_{u_0}(u,\xi),M_{n,u_0}(u,\xi)
\right] \neq \mathbf 0$,
\end{enumerate}
then there exists some $H_n$ with $\|H_n\| = \Theta(n^{1-2\gamma})$ such that $\hat\theta_{H_n, M_{n,u_0}}(u_0)$ achieves a first-order improvement over $\hat\theta_{EO}(u_0)$, whereas any $H_n$ with $\|H_n\| = o(n^{1-2\gamma})$ cannot.
\end{theorem}
Our characterization above holds for every covariate $u_0$. 
To understand whether using a given source of side information can achieve a first-order improvement at $u_0$, we need to verify two conditions in \Cref{prop:contextual-improvement}.  First, the ``almost'' correct side information condition is automatically satisfied if $M_{n,u_0}(u, \xi)$ approximates any centered moment of the joint distribution $\P_{(u, \xi)}$ well. On the other hand, the non-orthogonality condition requires careful construction of $M_{n,u_0}$. Intuitively, $M_{n,u_0}(u,\xi)$ can improve the estimator only if it captures some of the same local fluctuations as $\IFx_{u_0}(u,\xi)$. If it is orthogonal to these fluctuations, then $\Gamma_{u_0}=0$, and $M_{n,u_0}$ provides no first-order variance reduction.

Therefore, an effective perturbation function must depend on $u_0$, and a natural choice is to incorporate the conditional mean 
$\E[\xi|u_0]$. For example, under weighted EO solutions specified in \Cref{ex:kernel-learner0} with $K(x) = \mathbf{1}_{\|x\|_2 \leq 1}$, we have:
\begin{equation}\label{eq:side-info-cso}
M_{n,u_0}(u_i, \xi_i) = (\xi_i - \E_{\P}[\xi_i|u_0])\mathbf{1}_{\{\|u_i - u_0\|_2 \leq h_n\}}   
\end{equation}
for the corresponding bandwidth sequence $\{h_n\}_{n = 1}^{\infty}$. Here in~\eqref{eq:side-info-cso}, the indicator $\mathbf{1}_{\{\|u - u_0\|_2 \leq h_n\}}$ localizes the information to a neighborhood of $u_0$. In practice, $\E[\xi|u_0]$ can be estimated from pretrained models, which are well-suited to conditional mean prediction. Furthermore, the same computational procedures in \Cref{sec:perturb-opt} can be applied to maximize the first-order improvement. In Appendix~\ref{app:perturb-opt-cso}, we describe (i) mild conditions under which $\mu_n(u_0) = o(n^{\gamma-1})$ in~\eqref{eq:side-info-cso}; (ii) constructions for side information in other weight specifications; and (iii) computational procedures to maximize the first-order improvement.

Our diagnostics of performance improvements over weighted EO solutions contribute to the broader contextual optimization literature. These weighted EO solutions may converge to the optimal solutions slowly, so improvements to them are especially valuable. Similar to \Cref{subsec:perturb-stat-connection}, most existing weighted EO+ methods do not improve in the first-order sense unless the two conditions from~\Cref{prop:contextual-improvement} hold.
In contrast, existing comparisons in contextual optimization focus only on decisions under parametric distributions~\citep{elmachtoub2023estimatethenoptimize,elmachtoub2025dissecting} or linear objectives~\citep{hu2022fast}, and no general framework exists for improving these solutions in a principled way.
\section{Further Discussions}\label{sec:discuss}
\subsection{Identifying Side Information in Practical Problems}\label{app:identify}
Our analysis in previous sections reduces the question of constructing optimally perturbed (weighted) EO solutions to a concrete verification task: identifying a candidate perturbation function $M$ that is nearly correct and non-orthogonal to the fluctuation of the EO solution. In practice, such correct side information commonly arises from three sources. These categories are not mutually exclusive. Throughout the paper, the form of $M$, including any pretrained component, is externally supplied or learned from auxiliary data independent of $\Dscr_n$, except for the joint optimization of the adjustment matrix and side information in Appendix~\ref{subsec:comp-h-phi}.

\textbf{Shape Information.} The first source is structural knowledge about the distribution of the uncertainty variable, e.g., symmetry or tail-probability information. We illustrate this with several examples. First, continuing \Cref{ex:lr}, if the decision-maker knows that the noise $\epsilon = Y - (\theta^*)^{\top}X$ is symmetric around zero and exogenous, then all odd residual moments vanish at $\theta^*$, so that $M_{\phi}(\theta;\xi) = X\,\big(Y - \theta^{\top}X\big)^{2\phi + 1}, \forall \phi \in \mathbb{N}^+$ constitutes a collection of correct side information, i.e., $\mu_{M_{\phi}}(\theta^*) = \mathbf{0}$ for every $\phi$. We characterize the corresponding side information in \Cref{coro:linear-symmetric} and will further validate these constructions empirically in \Cref{subsec:lr}, where the side information induced by $\chi^2$-divergence and CVaR-DRO is shape-based. A second example is tail-probability information. In inventory and reliability applications, historical operations data can reveal a stable stockout frequency, yielding $\P(\xi > \tau) = p_0$ for accurately known threshold $\tau$ and probability level $p_0$, giving rise to $M(\theta;\xi) = \mathbf{1}\{\xi > \tau\} - p_0$.

\textbf{Invariant Information.} The second source comprises relationships that remain stable across multiple data sources or environments~\citep{arjovsky2019invariant}, a structure also commonly exploited in transfer learning. Suppose auxiliary datasets drawn from multiple distributions share an invariant component with $\P$, such as common moment restrictions. Our newsvendor example in \Cref{subsec:context-newsvendor} operationalizes exactly this mechanism by using an auxiliary distribution $\Q \neq \P$ that satisfies only the invariance condition $\E_{\Q}[\xi|u] = \E_{\P}[\xi|u]$. This condition is used both in the real-world experiment (\Cref{subsec:context-newsvendor}), and the simulation study (Appendix~\ref{app:newsvendor-simulate}), where the expected demand-feature relationship is assumed to remain invariant across products and shops.

\textbf{A Priori Moment Information.} The third source comprises moment restrictions imposed from external knowledge. Examples include a priori distributional information such as known means or dispersion bounds~\citep{delage2010distributionally} and moment conditions arising from partial observations in semi-supervised learning~\citep{song2024general}. In contextual optimization, a canonical example is the conditional mean $\E[\xi|u_0]$, which yields the construction in~\eqref{eq:side-info-cso}, i.e., centering $\xi$ at the (estimated) conditional mean.

Finally, it is important to distinguish correctness from first-order improvement. The three sources above only guarantee $\mu_M(\theta^*) = \mathbf{0}$. Whether the information leads to first-order improvement is also determined by non-orthogonality, which does not need to be verified analytically in advance. In fact, for any such prespecified $M$, if it is nearly orthogonal to the influence function, the estimated adjustment matrix $\widehat H$ in \Cref{sec:perturb-opt} will be close to zero and the perturbed solution nearly reduces to the EO solution.
\subsection{Generalizations}\label{subsec:generalization}
Our framework can be generalized in several directions. Below, the first generalization follows directly from our analysis; the others are directions for future work.

\paragraph{Nearly Correct Side Information.} Our framework extends to nearly correct side information where the bias of the side information is small. For example, $\mu_M(\theta^*) = o(n^{-1/2})$ in the standard setting (cf.~\Cref{thm:improvement-principle}), or $\mu_n(u_0) = o(n^{\gamma-1})$ in the contextual setting, as illustrated in~\Cref{prop:contextual-improvement}. In Appendix~\ref{app:newsvendor-simulate}, our numerical results show that significant improvement can be achieved under such nearly correct side information. Furthermore, our asymptotic comparison based on expected excess risk can be extended to general performance metrics. We refer to Appendix~\ref{app:generalization-improvement} for details. 

\paragraph{Finite-sample Improvements.} Our asymptotic comparison of $R(\hat\theta_{H,M})$ versus $R(\hat\theta_{EO})$ can be translated into finite-sample guarantees. When $\hat\theta_{H,M}$ achieves a first-order improvement over $\hat\theta_{EO}$, it attains a smaller asymptotic variance under the normal approximation. Consequently, by leveraging tools such as the Berry--Esseen bound~\citep{shao2022berry}, together with uniform control of the stochastic-expansion remainders, this improvement can be extended to the finite-sample regime; see, e.g., \citet{elmachtoub2025dissecting}. Developing such guarantees is an interesting direction for future work.

\paragraph{Improvements under Other Optimization Structures.}
Several extensions beyond twice-differentiable, strictly convex objectives are worth exploring. First, when the cost function is linear in $\theta$, the excess risk $R(\hat\theta_{EO})$ may not be of order $O(n^{-1})$ but may instead depend on problem degeneracy. One possible approach is to consider the entropic-regularized solution~\citep{weed2018explicit}. Second, for regular constrained problems in which $\theta^*$ lies on the boundary of $\Theta$, the influence function admits an explicit characterization~\citep{duchi2021asymptotic}. In this case, $\Pi_{\Theta}\Para{\hat\theta_{EO} + \frac{H}{n}\sum_{i=1}^n M(\hat\theta_{EO};\xi_i)}$ may have complex asymptotic behavior and its exact performance is difficult to characterize. One possible approach is to incorporate the side information as a regularization term in the objective: $\hat\theta_{H,M} \in \argmin_{\theta \in \Theta} \sum_{i=1}^n \big(\ell(\theta;\xi_i) + \|H\,M(\theta;\xi_i)\|_2^2\big)$. We leave these directions for future work.
\section{Numerical Experiments}\label{sec:numeric}
In this section, we present numerical results to validate our previous theoretical findings and demonstrate the superiority of the optimally perturbed (weighted) EO solutions. In \Cref{subsec:lr}, we mainly focus on validating the improvement pattern of existing EO+ solutions derived in \Cref{sec:perturb-stat} in a linear regression problem. In \Cref{subsec:context-newsvendor}, we compare the optimally perturbed EO solution developed in Sections~\ref{sec:perturb-opt} and~\ref{sec:cso} against weighted EO/EO+ solutions in a newsvendor problem. Our optimally perturbed (weighted) EO solution does not need to tune the adjustment size once the solution $\hat\theta_{EO}$ is obtained. In contrast, for existing (weighted) EO+ methods, we report their performance under the adjustment sizes that perform best over the true distribution $\P$ for each problem instance to give them the strongest possible performance. That is, we give these benchmark methods an advantage beyond what the data alone would allow, which makes the strength of our approach even more apparent. For computational tractability, in regularized/DRO problems, when $\lambda < 0$, we use the extrapolated solution $\hat\theta_{\lambda}^{\mathrm{ext}}:=2\hat\theta_{EO}-\hat\theta_{-\lambda}$ rather than the solution of the corresponding formal negative-parameter problem whose outer optimization can be nonconvex, e.g.,~\eqref{eq:reg-method} and~\eqref{eq:chi2-eo}; this extrapolation matches the local perturbation to first order for small $|\lambda|$.
\subsection{Linear Regression}\label{subsec:lr}
We consider the OLS loss $\ell(\theta;\xi) = (\theta^{\top}X - Y)^2$, where $\xi = (X, Y)^{\top}$, $X$ denotes the feature and $Y$ denotes the label. The true data-generating process is $Y = (\theta^*)^{\top}X + \epsilon$ with noise $\epsilon$ following different parametric specifications (Normal, Laplace, $t$-distribution) with details in Appendix~\ref{app:numerical-lr}.
We consider two existing EO+ solutions: $\chi^2$-EO+ and CVaR-EO+ (in Appendix~\ref{app:cvardro}). Their best adjustment sizes $\lambda$ in~\eqref{eq:dro-method} are chosen via the grid search to minimize the expected excess risk:
\begin{enumerate}[leftmargin = *]
\item \textbf{\(\chi^2\)-EO+.}  We search \(\lambda\in (-0.2,0.2]\) with a step size $0.0025$. When $\lambda$ is positive, we solve the optimization in \eqref{eq:chi2-eo}. When $\lambda$ is negative, we compute \(\hat\theta_{\lambda}=2\hat\theta_{EO}-\hat\theta_{-\lambda}\).
\item \textbf{CVaR-EO+.} We search \(\lambda\in(-1,1)\) with a step size \(0.025\). When $\lambda$ is positive, we solve the optimization~\eqref{eq:cvar-objective} in Appendix~\ref{app:cvardro}. When $\lambda$ is negative, we compute \(\hat\theta_{\lambda}=2\hat\theta_{EO}-\hat\theta_{-\lambda}\).
\end{enumerate}

We refer to these two EO+ solutions with the best $\lambda$ as \textbf{Chi2} and \textbf{CVaR}, respectively. In the following, we mainly compare their performance improvements over the EO solution under different noise specifications. Furthermore, to demonstrate the equivalent performance between these EO+ solutions and the corresponding directionally perturbed EO solutions in Definition~\ref{defn:direction-perturb} and \Cref{thm:unification}, we compute the optimally perturbed EO+ solutions given the following two perturbation functions motivated from the $\chi^2$-EO+ and CVaR-EO+, in Corollary~\ref{ex:auxiliary-robust} and Theorem~\ref{thm:cvar-unification} respectively: (i) $M_1(\theta;\xi) = X (Y -\theta^{\top}X)^3$; (ii) $M_{\phi}(\theta;\xi) = (1 - \frac{1}{1-\phi}\cdot \mathbf{1}_{\{\epsilon^2 > q_{\phi}\}}) X \epsilon$, where $q_{\phi}$ is the $\phi$-quantile of $\ell(\theta^*;\xi)$ with $\phi$ selected by minimizing the expected excess risk over $[0, 1)$. We refer to our optimally perturbed EO+ solutions $\hat\theta_{\widehat H, M_1}$ and $\hat\theta_{\widehat H, M_{\widehat\phi}}$ as \textbf{EO-Chi2} and \textbf{EO-CVaR}, respectively, with design and implementation details deferred to Appendix~\ref{app:numerical-lr}. 

\begin{figure*}[!htb]
\centering
\includegraphics[width=0.8\textwidth]{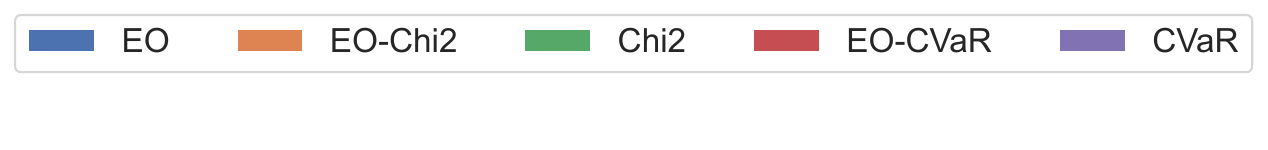}\\
\subfloat[Laplace\label{fig:lr_laplace}]{
    \includegraphics[width=0.32\textwidth]{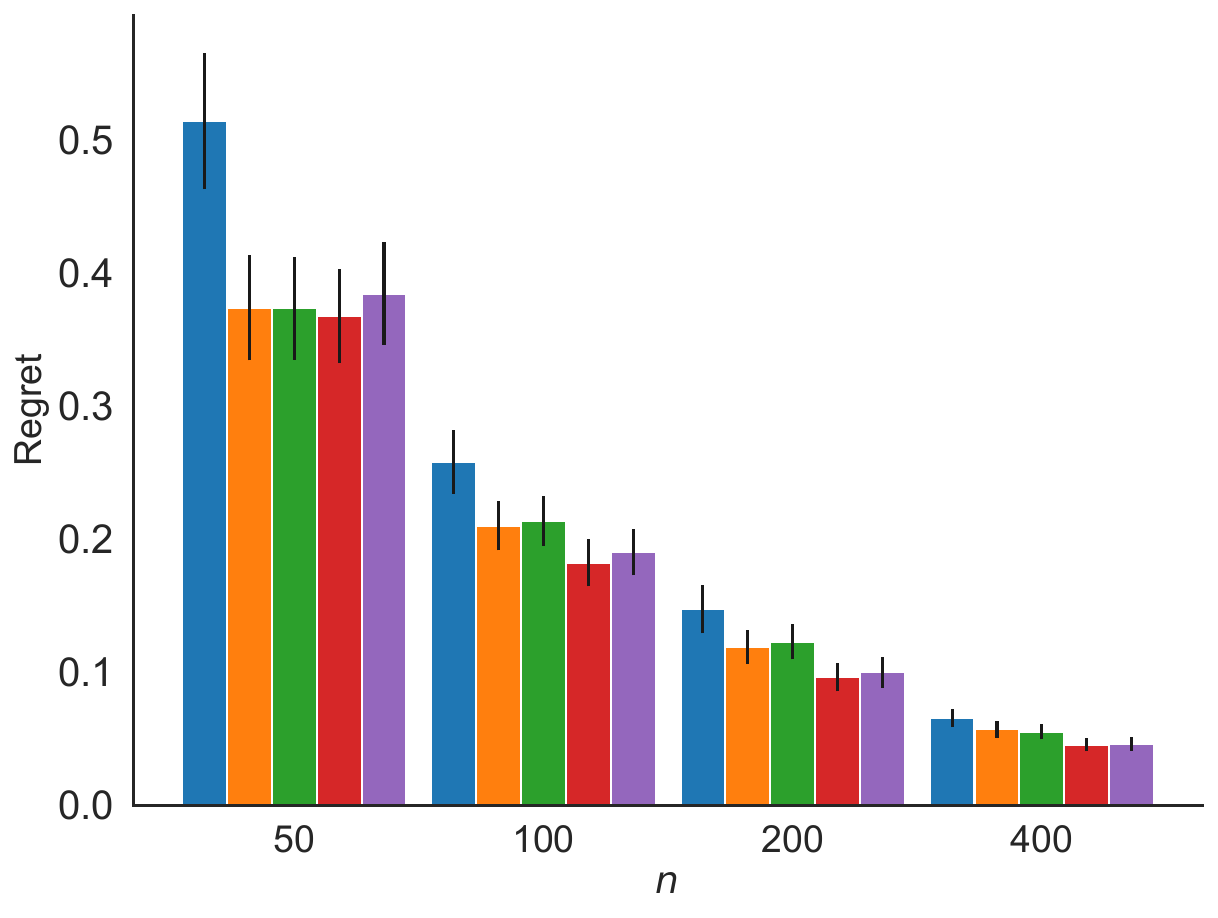}
}
\subfloat[Normal\label{fig:lr_normal}]{
    \includegraphics[width=0.32\textwidth]{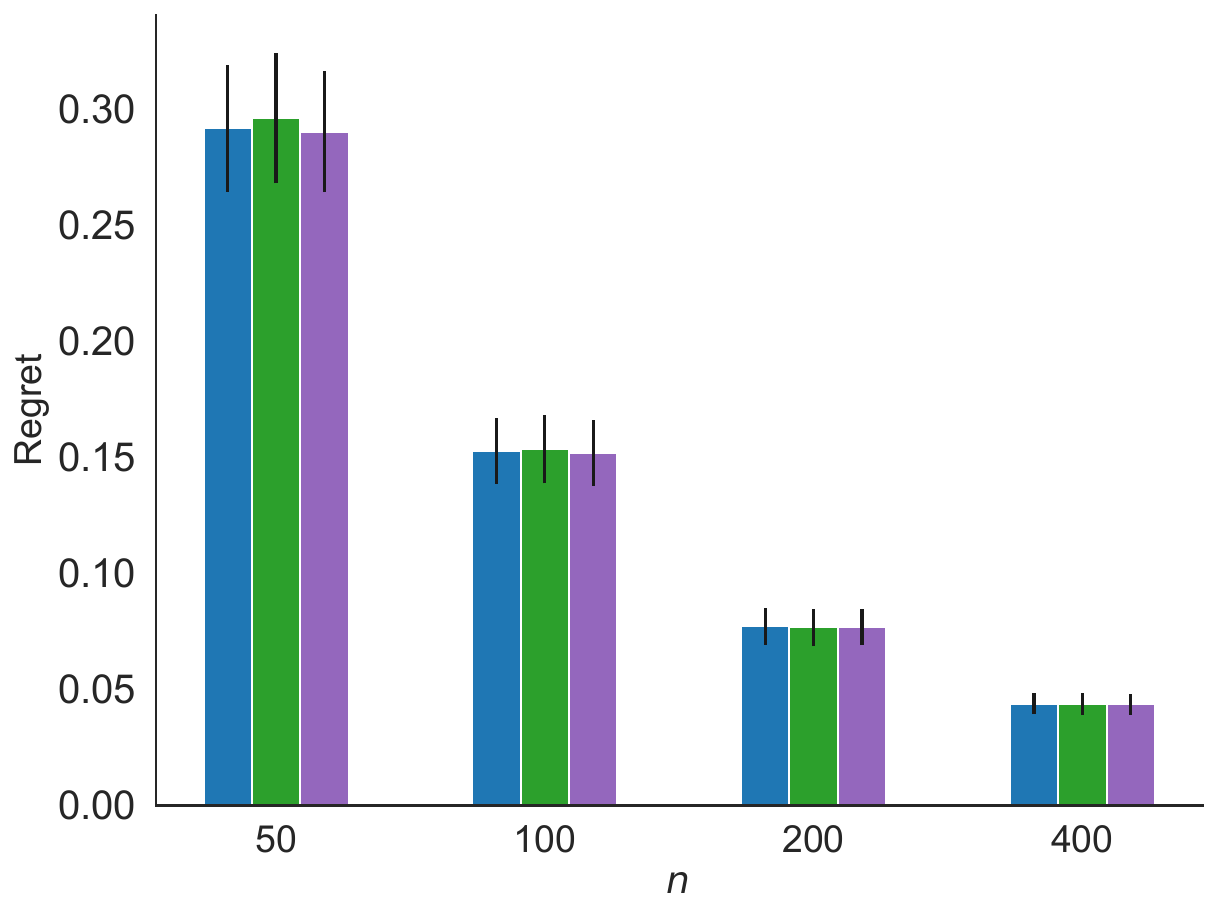}
}
\subfloat[$t$\label{fig:lr_t}]{
    \includegraphics[width=0.32\textwidth]{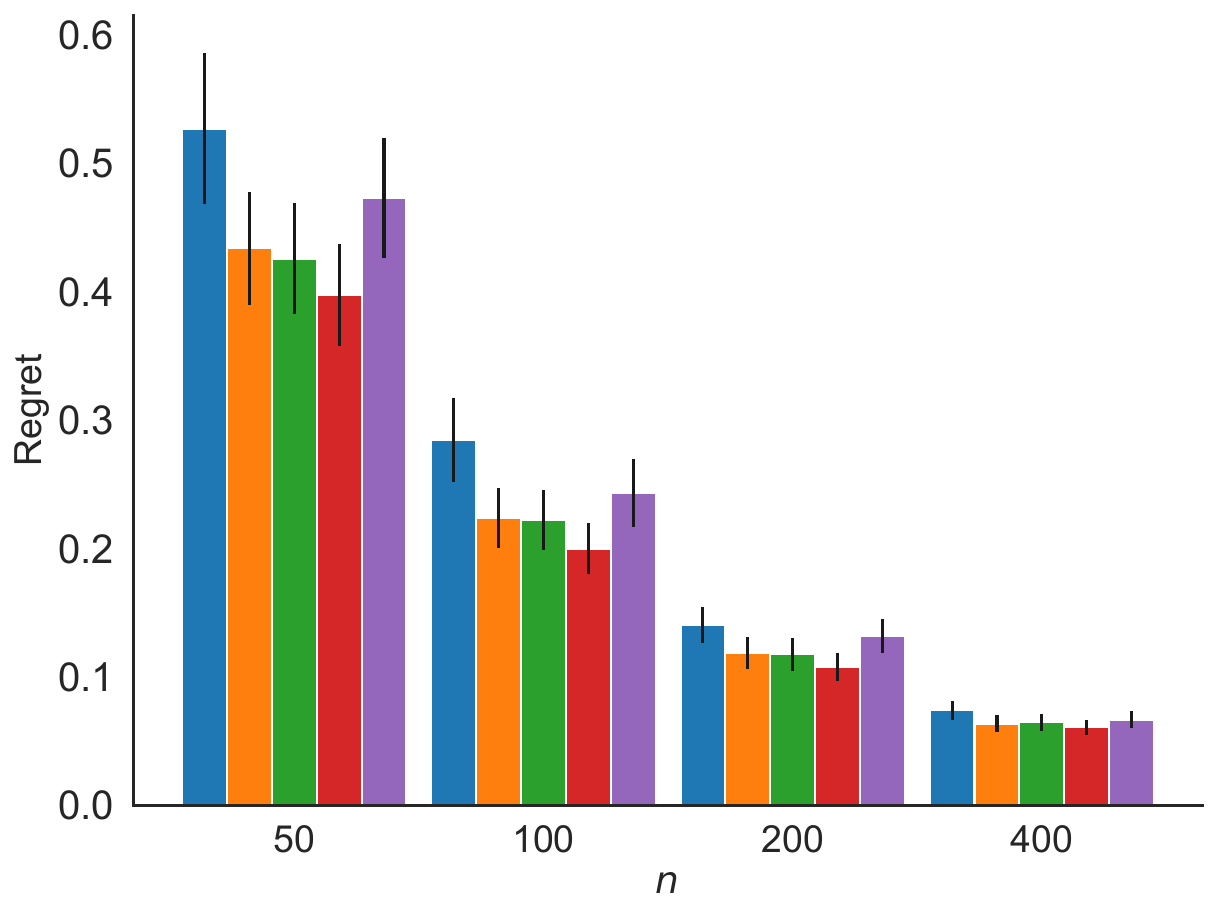}
}
\caption{Expected excess risk comparison of various methods in linear regression under different noise specifications, where the error bars denote the standard errors of the excess risk of each method.}
\label{fig:synthetic_regret_lr}
\end{figure*}

In Figure~\ref{fig:synthetic_regret_lr}, we report the expected excess risk of different solutions across noise specifications. Under the Laplace and $t$-distributions, both the $\chi^2$-EO+ and CVaR-EO+ solutions yield substantial risk reduction compared with the EO solution. Under the normal distribution, neither EO+ solution yields significant improvement. The pattern is consistent with \Cref{coro:linear-symmetric}: the improvement of $\chi^2$-EO+ over EO occurs precisely in the cases for which \Cref{coro:linear-symmetric} predicts a first-order gain under the OLS loss.
In fact, the first-order improvement result for $\chi^2$-EO+ under OLS in \Cref{coro:linear-symmetric} holds for CVaR-EO+ as well. In \Cref{fig:synthetic_regret_size}, we further report the best adjustment sizes $\lambda$ for these two EO+ methods under different noise specifications and compare them with the theoretically optimal $\lambda$ from \Cref{coro:linear-symmetric}.  In \Cref{fig:synthetic_regret_size}$(b)(e)$, under the normal distribution, the best $\lambda$ converges to 0 as $n$ increases, which is consistent with the optimal order $\Theta(1/n)$ in this case. Under other noise specifications, the best $\lambda$ remains bounded away from 0 despite fluctuations, which is consistent with the optimal order $\Theta(1)$ in these cases. These results offer practical guidance for hyperparameter tuning in an EO+ method when one wants to find the $\lambda$ with the smallest expected excess risk. Specifically, one can check whether the condition in \Cref{coro:eo-plus-improvement} holds for the particular cost function and information about the distribution: if it does, a large $\lambda$ can be retained and tuned; otherwise, $\lambda$ should be shrunk with the sample size $n$.

We also comment that \textbf{EO-Chi2} and \textbf{EO-CVaR} perform similarly to the corresponding EO+ solutions under the best $\lambda$, as observed in \Cref{fig:synthetic_regret_lr}. This indicates that our optimally perturbed EO solution exploits structural knowledge of the noise distribution and closely matches their performance. However, computing the optimally perturbed EO solution only requires solving EO, which is computationally more efficient than solving the DRO counterpart.


\begin{figure*}
\centering
\subfloat[Laplace-Chi2\label{fig:lr_laplace_radius_chi2}]{
    \includegraphics[width=0.32\textwidth]{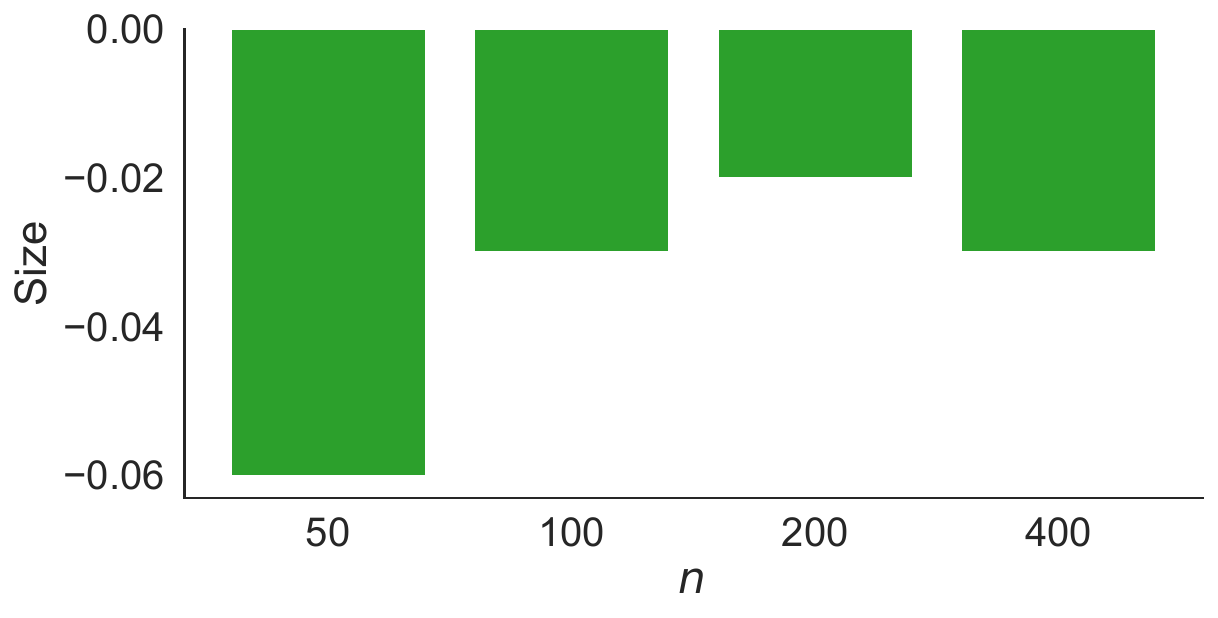}
}
\subfloat[Normal-Chi2\label{fig:lr_normal_radius_chi2}]{
    \includegraphics[width=0.32\textwidth]{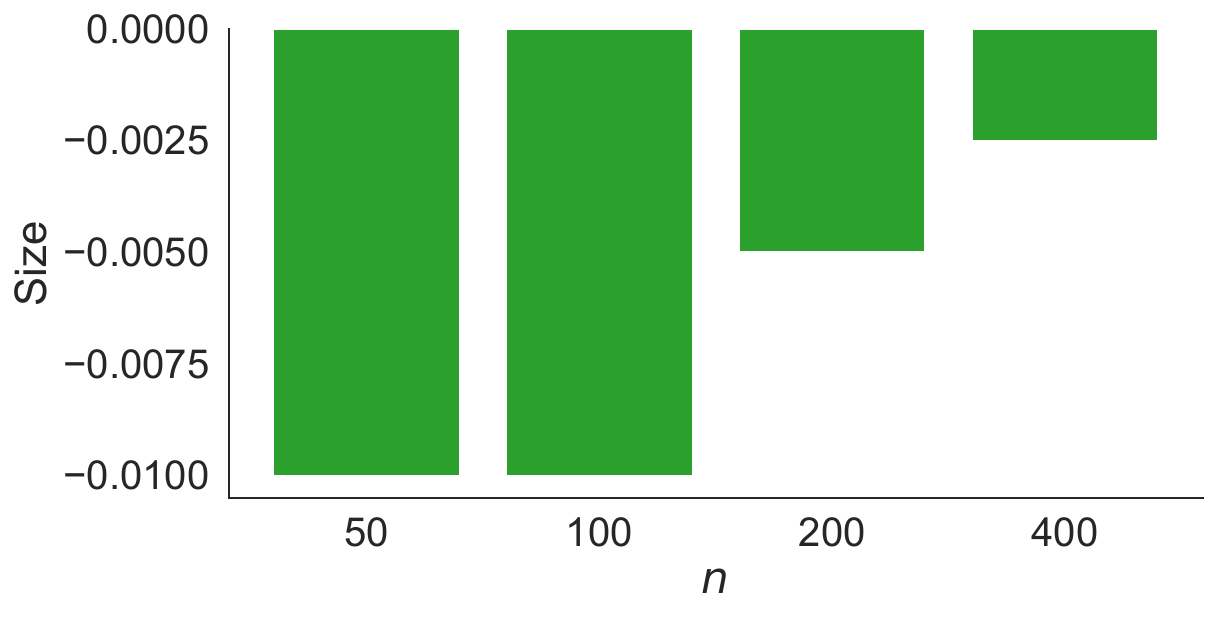}
}
\subfloat[$t$-Chi2\label{fig:lr_t_radius_chi2}]{
    \includegraphics[width=0.32\textwidth]{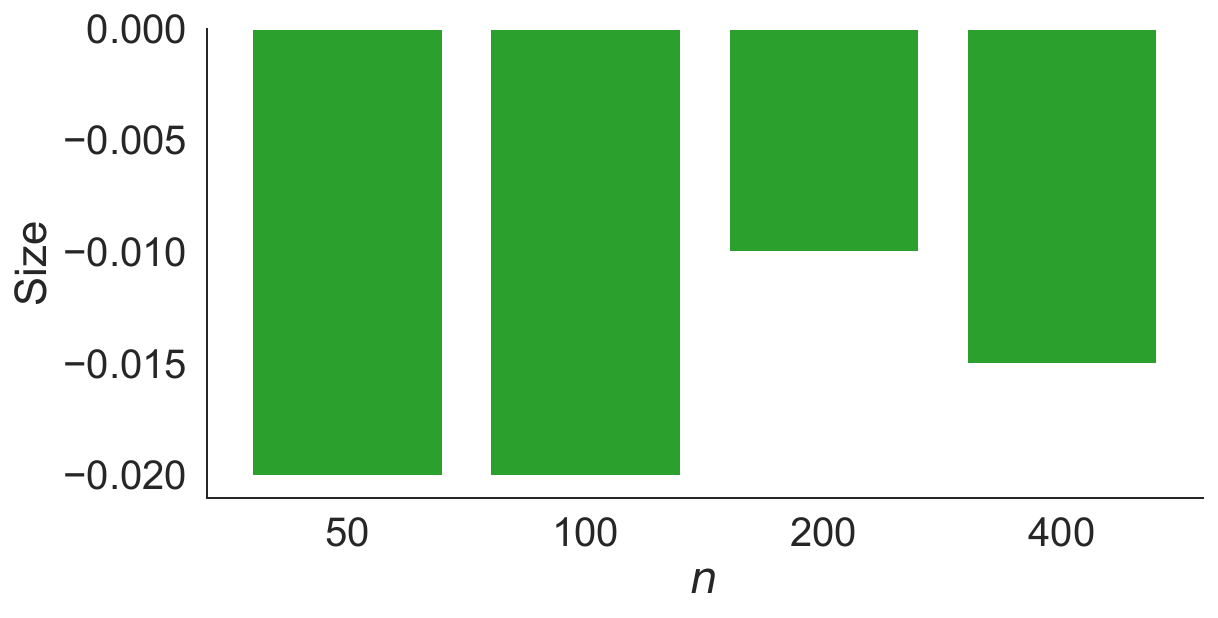}
}\\
\subfloat[Laplace-CVaR\label{fig:lr_laplace_radius_cvar}]{
    \includegraphics[width=0.32\textwidth]{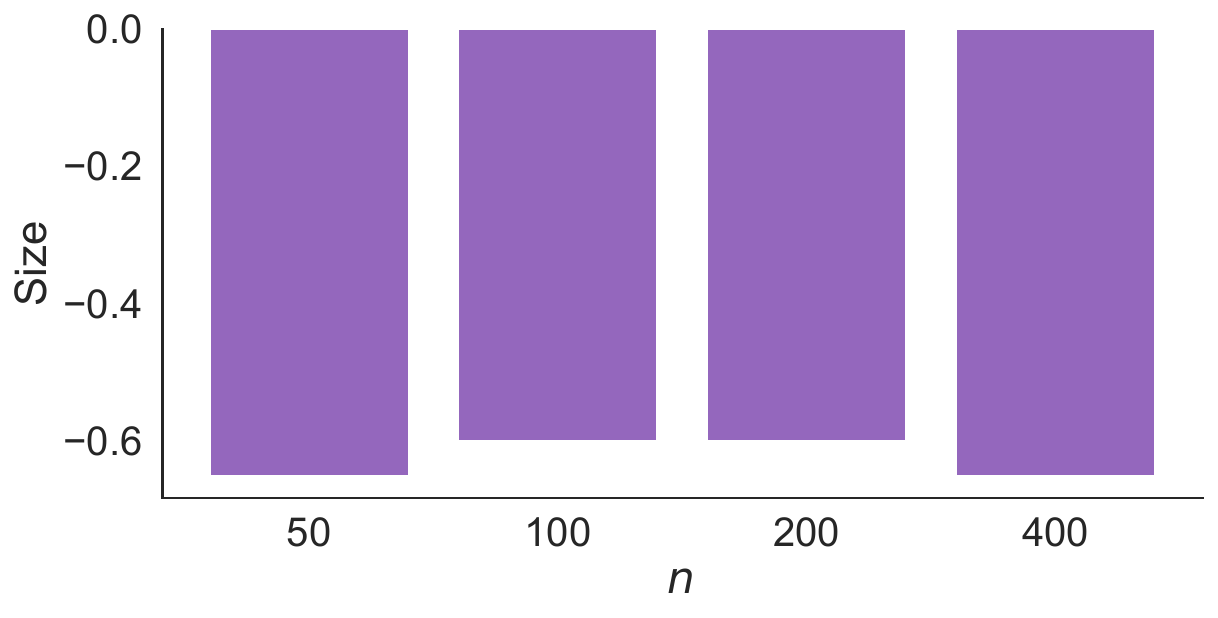}
}
\subfloat[Normal-CVaR\label{fig:lr_normal_radius_cvar}]{
    \includegraphics[width=0.32\textwidth]{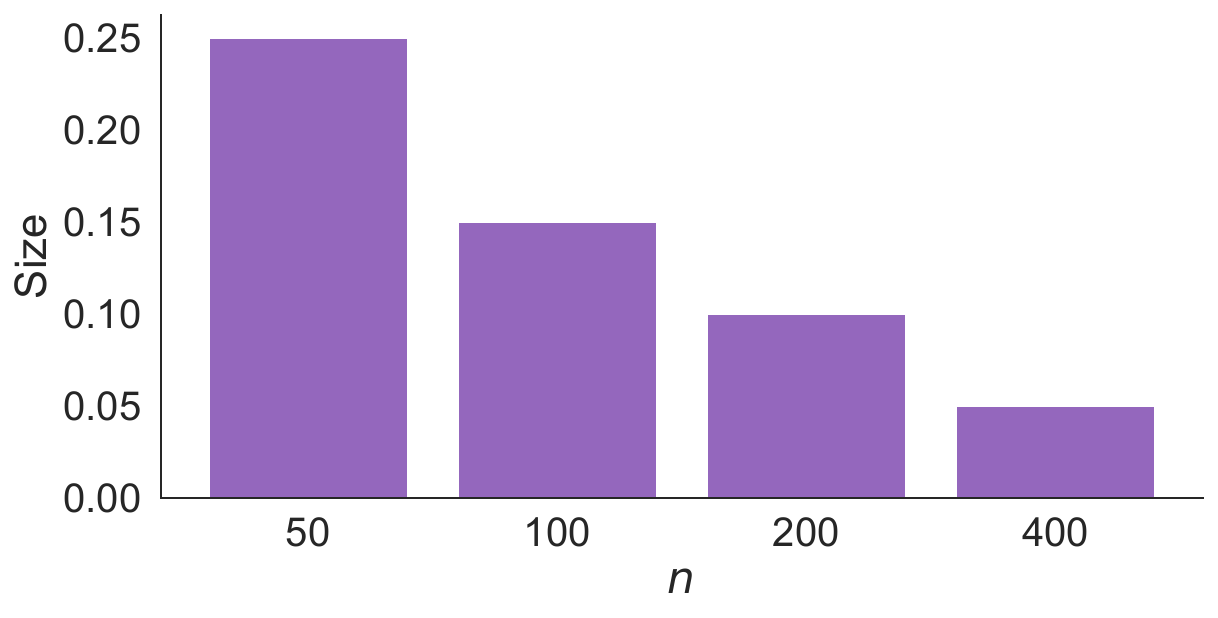}
}
\subfloat[$t$-CVaR\label{fig:lr_t_radius_cvar}]{
    \includegraphics[width=0.32\textwidth]{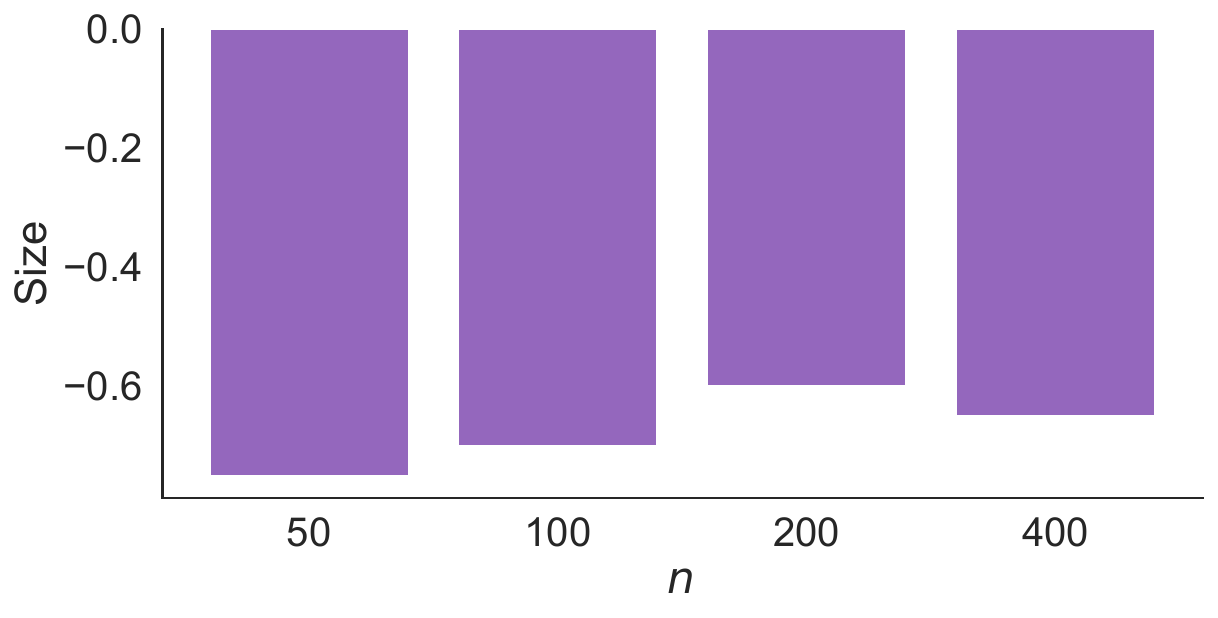}
}
\caption{The size of the selected $\lambda$ with the smallest out-of-sample cost for EO+ methods under different noise specifications. Each subplot corresponds to a pairing of a residual distribution (Laplace, Normal, $t$) with an EO+ formulation (Chi2 and CVaR).}
\label{fig:synthetic_regret_size}
\end{figure*}

\subsection{Contextual Newsvendor Problem}\label{subsec:context-newsvendor}
Consider the newsvendor loss in \Cref{ex:newsvendor} under contextual features. We evaluate the performance of different methods using three real-world datasets (\texttt{Bakery}, \texttt{Restaurant}, \texttt{SID}) following the same preprocessing pipeline as in \cite{buttler2022meta}. We set $c = 1$ and $p = 5$. To obtain $\hat\theta_{EO}(u)$, we use three nonparametric weighting methods in~\eqref{eq:np-optimizer0}, namely k-Nearest-Neighbor (kNNW), Decision Tree (DTW), and Random Forest (RFW). We provide details of the nonparametric estimates in \Cref{app:perturb-opt-cso}. 

In our directionally perturbed weighted EO solution~\eqref{eq:perturb-weight-eo}, we set the perturbation function according to~\eqref{eq:side-info-cso} by incorporating the conditional mean $\hat\E[\xi|u]$. We describe the data summary, covariates and approaches used to construct the conditional mean $\hat\E[\xi|u]$ from the three datasets in \Cref{tab:dataset}. Here, we refer to the optimally perturbed weighted EO solution as \textbf{EO-M} and defer its detailed construction procedure to Appendix~\ref{app:procedure}. Appendix~\ref{app:newsvendor-simulate} reports a simulation study whose results are consistent with the main findings.  
\begin{table}[!htb]
    \centering
    \caption{Overview of newsvendor datasets}
    \label{tab:dataset}
    \resizebox{0.95\textwidth}{!}{\begin{tabular}{c|cccc}
    \toprule
       Dataset  & Products $\times$ Shops & Sample Size per Product/Shop $(n)$ & Feature & Side Information (Predicted $\hat\E[\xi|u]$)\\
       \midrule
       \texttt{Restaurant}  & 7 $\times 1$ & 765 & \makecell{calendric, lag, weather\\holiday, promotions} & \makecell{a trained demand predictor using standard features\\and concurrent demand for other products}\\
        \texttt{Bakery} & 3 $\times$ 35 &1215 & \makecell{calendric, lag,\\ weather, holidays} & \makecell{a trained demand predictor using features from \\all products and shops via label encoding}\\
        \texttt{SID} &  50 $\times 10$ & 1826 & calendric, lag & \makecell{a trained demand predictor using features from \\all products and shops via label encoding}\\
        \bottomrule
    \end{tabular}}
\end{table}

We compare the \textbf{EO-M} solution against standard weighted EO+ solutions including (i) $\chi^2$-EO+ in~\eqref{eq:chi2-eo}; (ii) objective regularization (``Obj Reg'') in~\eqref{eq:reg-method}; (iii) \textbf{EO-M (fixed)}: the directionally perturbed weighted EO solution while fixing the estimated value $\widehat H = -1$. Note that this is the optimal $H$ when the demand noise $\xi - \E[\xi|u]$ is normally distributed under newsvendor loss, as shown in Appendix~\ref{app:newsvendor-simulate}; (iv) \textbf{EO-feature}: the weighted EO solution while augmenting $u$ with concurrent demands of other products in the \texttt{Restaurant} dataset. We report the performance of different methods in \Cref{tab:cost-reduce-real} across datasets and weight configurations (kNNW, DTW, RFW). When comparing costs, we report the maximum seed-specific p-value. Thus, a small reported value indicates statistical significance for every seed, consistent with our goal of uniform improvement. We have the following findings: 

\begin{table}[!htb]
\centering
\caption{Baseline out-of-sample objective values of the weighted EO solutions and cost reductions achieved by the corresponding weighted EO+ solutions. For each EO+ method, cost reduction is defined as
$\Delta:=\operatorname{Cost}(\mathrm{EO})-\operatorname{Cost}(\mathrm{EO+})$,
so that positive values indicate better performance relative to EO. The ``-feature'' method is evaluated only on the \texttt{Restaurant} dataset. Values are averaged over random seeds (40 for \texttt{Restaurant}, 10 for \texttt{Bakery}, and 5 for \texttt{SID}), with standard deviations reported in parentheses. }
\label{tab:cost-reduce-real}
\resizebox{\textwidth}{!}{
\begin{tabular}{llc|ccccc}
\toprule
& & \multicolumn{1}{c|}{Baseline} & \multicolumn{5}{c}{Cost Reduction Relative to EO ($\uparrow$)} \\
\cmidrule(lr){3-3} \cmidrule(lr){4-8}
Dataset & Estimator & EO Objective & EO-$M$ & EO-$M$ (fixed) & EO-feature & $\chi^2$-EO+ & Obj Reg \\
\midrule

\multirow{3}{*}{\texttt{Restaurant}}
& KNNW & $-61.030{\scriptsize(5.980)}$ & $1.486{\scriptsize(0.177)}^{***}$ & $1.528{\scriptsize(0.193)}^{**}$ & $-0.026{\scriptsize(0.072)}$ & $0.171{\scriptsize(0.030)}^{*}$ & $0.098{\scriptsize(0.021)}$ \\
& DTW & $-59.698{\scriptsize(5.961)}$ & $2.691{\scriptsize(0.228)}^{***}$ & $2.722{\scriptsize(0.239)}^{***}$ & $0.216{\scriptsize(0.144)}$ & $0.509{\scriptsize(0.072)}^{**}$ & $0.219{\scriptsize(0.042)}^{*}$ \\
& RFW & $-61.959{\scriptsize(6.078)}$ & $0.328{\scriptsize(0.112)}$ & $0.001{\scriptsize(0.121)}$ & $1.112{\scriptsize(0.122)}^{***}$ & $0.075{\scriptsize(0.022)}$ & $0.091{\scriptsize(0.024)}$ \\

\midrule

\multirow{3}{*}{\texttt{Bakery}}
& KNNW & $-336.466{\scriptsize(124.730)}$ & $25.965{\scriptsize(10.235)}^{***}$ & $28.350{\scriptsize(11.307)}^{***}$ & -- & $0.511{\scriptsize(0.400)}^{***}$ & $0.194{\scriptsize(0.255)}^{***}$ \\
& DTW & $-356.857{\scriptsize(132.771)}$ & $16.204{\scriptsize(5.212)}^{***}$ & $19.949{\scriptsize(6.480)}^{***}$ & -- & $1.412{\scriptsize(0.758)}^{***}$ & $0.242{\scriptsize(0.248)}^{**}$ \\
& RFW & $-361.602{\scriptsize(133.979)}$ & $5.873{\scriptsize(2.264)}^{***}$ & $-3.024{\scriptsize(6.430)}$ & -- & $0.231{\scriptsize(0.399)}^{*}$ & $0.744{\scriptsize(0.485)}^{***}$ \\

\midrule

\multirow{3}{*}{\texttt{SID}}
& KNNW & $-193.143{\scriptsize(28.665)}$ & $1.579{\scriptsize(0.147)}^{***}$ & $1.394{\scriptsize(0.173)}^{***}$ & -- & $0.021{\scriptsize(0.034)}^{***}$ & $0.009{\scriptsize(0.015)}^{**}$ \\
& DTW & $-194.480{\scriptsize(91.264)}$ & $1.496{\scriptsize(0.586)}^{***}$ & $1.792{\scriptsize(0.699)}^{***}$ & -- & $0.346{\scriptsize(0.276)}^{***}$ & $0.061{\scriptsize(0.083)}^{***}$ \\
& RFW & $-195.057{\scriptsize(91.420)}$ & $1.427{\scriptsize(0.575)}^{***}$ & $0.880{\scriptsize(0.659)}^{**}$ & -- & $0.055{\scriptsize(0.117)}^{***}$ & $0.065{\scriptsize(0.106)}^{***}$ \\

\bottomrule
\end{tabular}
}

\vspace{0.1cm}
\footnotesize{\emph{Note.} For each seed, a one-sided test is conducted across all product--shop pairs with $H_0: \Delta = 0, H_1: \Delta > 0$, yielding one p-value per seed. Superscripts $^{***}$, $^{**}$, and $^{*}$ indicate that the maximal seed-specific p-value is below 1\%, 5\%, and 10\%, respectively.}
\end{table}
First, across different weight estimators and datasets, incorporating side information leads to substantial performance improvements. Across nine dataset-estimator pair combinations, the \textbf{EO-M} solution achieves significant cost reductions compared to the corresponding weighted \textbf{EO} solution. In contrast, conventional weighted EO+ solutions ($\chi^2$-EO+, objective regularization) do not exhibit comparable gains.

More importantly, the performance gains from incorporating correct side information can exceed the inherent performance differences across weight estimators (e.g., between weighted EO solutions under KNNW and RFW). For instance, in the \texttt{Bakery} dataset, the baseline cost gap of weighted EO methods between the KNNW (or DTW) and RFW estimators is approximately 21 (or 6, respectively) units. In contrast, incorporating side information within KNNW (or DTW) leads to improvements of about 26 (or 16, respectively) units, thereby surpassing the gap attributable to the choice of weight estimator.

Finally, the \textbf{EO-M (fixed)} solution achieves strong performance as well. This implies that even a simple fixed adjustment magnitude ($\widehat H = -1$) is effective in practice when the adjustment matrix is difficult to compute.

\section*{Acknowledgement}
We gratefully acknowledge support from the InnoHK initiative, the Government of the HKSAR, Laboratory for AI-Powered Financial Technologies, and the Columbia Dream Sports AI Innovation Award.

\bibliography{references}
\bibliographystyle{plainnat}

\ECSwitch

\ECRUNTITLE{First-order Improvement in DDO}
\ECHead{Appendices}
\section{Proofs and Additional Details in Sections~\ref{sec:preliminary}--\ref{sec:perturb-stat}}\label{app:unification}
\subsection{Helper Lemmas}\label{app:tech-lemma}
We introduce the following standard asymptotic lemmas from the literature, which will be used in the proof of \Cref{prop:eo-conv} and \Cref{thm:unification}.
\begin{lemma}[Theorem 5.23 in \cite{van2000asymptotic}; Theorem 1 in \cite{nishiyama2010moment}]\label{lemma:m-estimator0}
Let $\xi_1,\ldots,\xi_n$ be i.i.d.\ from $\P$, let $\Theta\subseteq\R^d$ be open, and define
$M(\theta):=\E_{\P}[m_{\theta}(\xi)]$ and
$M_n(\theta):=n^{-1}\sum_{i=1}^n m_{\theta}(\xi_i)$.
Suppose that: 
\begin{enumerate}[(i)]
\item $\theta_0\in\Theta$ is the unique minimizer of $M$ and
$\hat\theta_n\overset{p}{\to}\theta_0$; 
\item $\theta\mapsto m_{\theta}(\xi)$
is differentiable at $\theta_0$ for $\P$-almost every $\xi$, with
$\E_{\P}[\nabla_{\theta}m_{\theta_0}(\xi)]=0$ and
$\E_{\P}[\|\nabla_{\theta}m_{\theta_0}(\xi)\|_2^2]<\infty$; 
\item On a
neighborhood of $\theta_0$, $|m_{\theta_1}(\xi)-m_{\theta_2}(\xi)| \leq K(\xi)\|\theta_1-\theta_2\|_2$ such that $\E_{\P}[K^3(\xi)]<\infty$;
\item $M(\theta)=M(\theta_0)
+\frac{1}{2}(\theta-\theta_0)^{\top}V_{\theta_0}(\theta-\theta_0)
+o(\|\theta-\theta_0\|_2^2)$, where $V_{\theta_0}:=\nabla_{\theta}^2M(\theta_0)$ is positive definite;
\item $\hat\theta_n$ is an approximate empirical minimizer satisfying $M_n(\hat\theta_n)\leq\inf_{\theta\in\Theta}M_n(\theta)+o_p(n^{-1})$.
\end{enumerate}
Then
\[
\sqrt n(\hat\theta_n-\theta_0)
=-V_{\theta_0}^{-1}\frac{1}{\sqrt n}
\sum_{i=1}^n\nabla_{\theta}m_{\theta_0}(\xi_i)+o_p(1),
\]
and hence $\sqrt n(\hat\theta_n-\theta_0)\Rightarrow
\mathcal N(0,V_{\theta_0}^{-1}\Omega_{\theta_0}V_{\theta_0}^{-1})$, where $\Omega_{\theta_0}:=
\E_{\P}\!\left[
\nabla_{\theta}m_{\theta_0}(\xi)
\nabla_{\theta}m_{\theta_0}(\xi)^{\top}
\right]$ and $\lim_{n \to \infty}\E_{\Dscr_n}\!\left[
n(\hat\theta_n-\theta_0)(\hat\theta_n-\theta_0)^{\top}
\right] = V_{\theta_0}^{-1}\Omega_{\theta_0}V_{\theta_0}^{-1}$.
\end{lemma}
Here, the moment convergence holds because $K(\xi) \in L_3(\P)$ and Theorem 1 in \cite{nishiyama2010moment} applies.


\begin{lemma}[Theorem 5.31 in \cite{van2000asymptotic}]\label{lemma:m-estimator0-aug}
Let $\xi_1,\ldots,\xi_n$ be i.i.d.\ from $\P$, let
$\Theta\subseteq\R^d$ and $\Lambda\subseteq\R^m$ be open, and define $\Psi(\theta,\lambda):=\E_{\P}[\psi_{\theta,\lambda}(\xi)]$, $\Psi_n(\theta,\lambda):=\frac{1}{n}\sum_{i=1}^n\psi_{\theta,\lambda}(\xi_i)$, where $\psi_{\theta,\lambda}:\Xi\to\R^d$. Fix
$(\theta^*,\lambda_0)\in\Theta\times\Lambda$, and suppose that:
\begin{enumerate}[(i)]
    \item for some $\delta>0$, the class $\big\{\psi_{\theta,\lambda}:
    \|\theta-\theta^*\|_2\leq\delta,\ 
    \|\lambda-\lambda_0\|_2\leq\delta\big\}$ is $\P$-Donsker, and $\E_{\P}[\|\psi_{\theta,\lambda}(\xi)
    -\psi_{\theta^*,\lambda_0}(\xi)\|_2^2]\to0$ as
    $(\theta,\lambda)\to(\theta^*,\lambda_0)$;
    \item $\Psi(\theta^*,\lambda_0)=0$;
    \item for $\lambda$ in a neighborhood of $\lambda_0$,
    $\Psi(\theta,\lambda)
    =\Psi(\theta^*,\lambda)
    +V_{\theta^*,\lambda}(\theta-\theta^*)
    +o(\|\theta-\theta^*\|_2)
    \quad\text{as }\theta\to\theta^*$, where $V_{\theta^*,\lambda}\to V_{\theta^*,\lambda_0}$ as
    $\lambda\to\lambda_0$ and the matrix
    $V_{\theta^*,\lambda_0}$ is nonsingular;
    \item the possibly random sequences
    $(\hat\theta_n,\lambda_n)\in\Theta\times\Lambda$ satisfy
    $(\hat\theta_n,\lambda_n)\overset{p}{\to}
    (\theta^*,\lambda_0)$ and $\sqrt n\,\Psi_n(\hat\theta_n,\lambda_n)=o_p(1)$.
\end{enumerate}
Then
\[
\hat\theta_n-\theta^*
=-V_{\theta^*,\lambda_0}^{-1}\Psi(\theta^*,\lambda_n)
-\frac{1}{n}V_{\theta^*,\lambda_0}^{-1}
\sum_{i=1}^n\psi_{\theta^*,\lambda_0}(\xi_i)
+o_p\!\left(
\|\Psi(\theta^*,\lambda_n)\|_2\vee n^{-1/2}
\right).
\]
\end{lemma}



\subsection{\textit{Proof of \Cref{prop:eo-conv}}}\label{app:eo-conv}
To establish the order $\gamma_{EO}$, we apply Theorem 5.9 of \cite{van2000asymptotic}. Because $\sup_{\theta \in \Theta}\|1/n\sum_{i \in [n]}\nabla_{\theta}\ell(\theta;\xi_i) - \nabla_{\theta}Z(\theta)\|\overset{p}{\to} 0$ and $\theta^*$ satisfies the first-order optimality condition in \Cref{asp:theta-star0}, the theorem gives $\hat\theta_{EO}\overset{p}{\to}\theta^*$.

Together with \Cref{asp:optimal-condition}, \Cref{lemma:m-estimator0} yields
\[\sqrt{n}\Para{\hat\theta_{EO} - \theta^*} = -\frac{1}{\sqrt{n}}\sum_{i = 1}^n[\nabla_{\theta}^2 Z(\theta^*)]^{-1} \nabla_{\theta}\ell(\theta^*;\xi_i) + o_p(1).\]
Therefore, $\hat\theta_{EO} - \theta^* = O_p(n^{-1/2})$.

Taking a second-order Taylor expansion of $Z(\theta)$ at the point $\theta = \theta^*$, we have:
\[Z(\hat\theta_{EO})- Z(\theta^*) = \frac{1}{2}(\hat\theta_{EO} - \theta^*)^{\top}\nabla_{\theta}^2 Z(\theta^*)(\hat\theta_{EO} - \theta^*) + o(\|\hat\theta_{EO} - \theta^*\|_2^2).\]
Taking expectations over $\Dscr_n$ and recalling $\E_{\Dscr_n}[\|\hat\theta_{EO} - \theta^*\|_2^2] = O(1/n)$ from \Cref{lemma:m-estimator0}, we have $\E[R(\hat\theta_{EO})] = \frac{C_{EO}}{n} + o(1/n)$ for some constant $C_{EO}$, and therefore, $\gamma_{EO} = 1/2$.

For the size of $\zeta_{EO}$, we require the following additional assumption:
\begin{assumption}[Regularity Condition of Second-Order Excess Risk Rate]\label{asp:regular-risk-rate}
\mbox{}
    \begin{itemize}
        \item $Z(\theta)$ is four times continuously differentiable with respect to $\theta$ around $\theta^*$;
        \item There exists some $\delta > 0$ such that $\sup_{n}\E_{\Dscr_n}[\|\sqrt{n}(\hat\theta_{EO} - \theta^*)\|_2^{4+\delta}] < \infty$, $\sup_{\theta\in\Theta}
\left\|
\nabla_\theta^k Z(\theta)
\right\|_{\mathrm{op}}
<\infty$ for $k=2,3,4$, and $\E_{\P}\left[
\|\IFx(\xi)\|^{4+\delta}
\right]
<\infty$;
        \item Denote $\Delta_n:=\hat\theta_{EO}-\theta^*$ and
        $\Sigma_0:=\E_{\P}[\IFx(\xi)\IFx(\xi)^\top]$, the estimator admits
        a second-order stochastic expansion. That is, for finite deterministic tensors
        $\mathcal B_2,\mathcal B_3$, and $\mathcal B_4$,
        \begin{equation*}
        \E_{\Dscr_n}[\Delta_n^{\otimes 2}] =\frac{\Sigma_0}{n}+\frac{\mathcal B_2}{n^2}
        +o(n^{-2}),
        \E_{\Dscr_n}[\Delta_n^{\otimes 3}] =\frac{\mathcal B_3}{n^2}+o(n^{-2}),
        \E_{\Dscr_n}[\Delta_n^{\otimes 4}] =\frac{\mathcal B_4}{n^2}+o(n^{-2}),
        \end{equation*}
        where the remainders converge in any fixed tensor norm. 
    \end{itemize}
\end{assumption}
Suppose \Cref{asp:regular-risk-rate} holds. Denote
$\Delta_n=\hat\theta_{EO}-\theta^*$. For a $k$th-order tensor
$A\in\R^{D_\theta\times\cdots\times D_\theta}$, we use the bracket
notation for tensor contraction,
\[
\nabla_{\theta}^k Z(\theta)[A]
:=\sum_{j_1,\ldots,j_k=1}^{D_\theta}
\frac{\partial^k Z(\theta)}
{\partial\theta_{j_1}\cdots\partial\theta_{j_k}}
A_{j_1,\ldots,j_k}.
\]
In particular, $\nabla_{\theta}^2 Z(\theta)[v^{\otimes2}]
=v^\top\nabla_{\theta}^2 Z(\theta)v$. A fourth-order Taylor expansion gives:
\begin{equation}\label{eq:fourth-order-eo-risk}
\begin{aligned}
Z(\hat\theta_{EO})-Z(\theta^*)
={}&\frac{1}{2}\nabla_{\theta}^2 Z(\theta^*)[\Delta_n^{\otimes 2}]
+\frac{1}{6}\nabla_{\theta}^3 Z(\theta^*)[\Delta_n^{\otimes 3}]\\
&+\frac{1}{24}\nabla_{\theta}^4 Z(\theta^*)[\Delta_n^{\otimes 4}]
+\rho_n .
\end{aligned}
\end{equation}
The smoothness and uniform-integrability conditions in
\Cref{asp:regular-risk-rate} imply
$\E_{\Dscr_n}[|\rho_n|]=o(n^{-2})$. Taking expectations in
\eqref{eq:fourth-order-eo-risk} and using the assumed tensor-moment
expansions yields
\[
\E_{\Dscr_n}[R(\hat\theta_{EO})]
=\frac{C_{EO}}{n}+\frac{E_{EO}}{n^2}+o(n^{-2}),
\]
where $C_{EO}:=\frac{1}{2}\nabla_{\theta}^2 Z(\theta^*)[\Sigma_0], E_{EO}:=\frac{1}{2}\nabla_{\theta}^2 Z(\theta^*)[\mathcal B_2]
+\frac{1}{6}\nabla_{\theta}^3 Z(\theta^*)[\mathcal B_3]
+\frac{1}{24}\nabla_{\theta}^4 Z(\theta^*)[\mathcal B_4]$.

$\hfill \square$

\subsection{\textit{Proof of \Cref{thm:improvement-principle}}}\label{app:improvement}
We use the following notation in this section:
\begin{align*}
\theta_{H,M}^* &:= \theta^* + H \E_{\P}[M(\theta^*;\xi)],\\\widebar M(\theta) &:= \frac{1}{n}\sum_{i = 1}^n M(\theta;\xi_i),\\\Omega_M&:=\text{Var}_{\P}\Paran{M(\theta^*;\xi)},\\\Gamma&:=\text{Cov}_{\P}\Paran{\IFx(\xi),M(\theta^*;\xi)},\\\Sigma_{0} &:= \E_{\P}[\IFx(\xi)\IFx(\xi)^{\top}] = (I_{\ell}(\theta^*))^{-1} \E_{\P}[\nabla_{\theta}\ell(\theta^*;\xi)\nabla_{\theta}\ell(\theta^*;\xi)^{\top}] (I_{\ell}(\theta^*))^{-1},\\\mu_M(\theta) &:= \E_{\P}[M(\theta;\xi)].
\end{align*}
We first describe the following helper lemmas used to prove \Cref{thm:improvement-principle}.
\begin{lemma}[Smoothness]\label{ass:moments}
Suppose 
there exists $K(\xi)$ such that $\E[K(\xi)] < \infty$ and $\|\nabla_{\theta} M(\theta;\xi) - \nabla_{\theta} M(\theta^*;\xi)\|\leq K(\xi)\|\theta - \theta^*\|$ for $\theta$ near $\theta^*$. Then:
\[\sup_{\|\theta-\theta^*\|\le r_n}\Big\|\widebar M(\theta)-\widebar M(\theta^*)-\nabla_\theta\widebar M(\theta^*)\Para{\theta-\theta^*}\Big\|_2 =o_p\Para{n^{-1/2}}\]
for $r_n = \Theta(n^{-\beta})$ with $\beta \in (1/4,1/2)$.
\end{lemma}

\begin{lemma}[Asymptotic Normality of Perturbed Solution]\label{lemma:asym-perturbed-solution}
    Consider $\hat\theta_{H,M}$ defined in~\eqref{eq:empirical-solution-adjustment}. Suppose Assumptions~\ref{asp:theta-star0} and~\ref{asp:optimal-condition} hold. If $M(\theta;\xi) \in \R^{D_M}$ is differentiable with respect to $\theta$ almost everywhere for each $\xi$ with $H\mu_M(\theta^*) = o(1)$ ($H$ can depend on $n$), we have:
    \[
\sqrt n(\hat\theta_{H,M}-\theta_{H,M}^*) - Y_H \Rightarrow 0
,
\]
where $Y_H \sim \mathcal N \big(0,\Sigma_H\big)$ with $\Sigma_H
=(I_{D_{\theta}\times D_{\theta}}+H \nabla_{\theta}\mu_M(\theta^*))\Sigma_{0}(I_{D_{\theta}\times D_{\theta}}+H\nabla_{\theta}\mu_M(\theta^*))^\top
+ H \Omega_M H^\top
+ 2\,\mathrm{Sym}\big((I_{D_{\theta}\times D_{\theta}}+H \nabla_{\theta}\mu_M(\theta^*))\Gamma H^\top\big)$ and $\mathrm{Sym}(A)=\frac{1}{2}(A+A^\top)$.
\end{lemma}

We provide the proofs of the two lemmas above in the following.

\textit{Proof of \Cref{ass:moments}.} We apply the mean-value argument for each $M(\theta;\xi_i)$ between $\theta^*$ and $\theta$:
\begin{equation}\label{eq:mean-val-m}
\|M(\theta;\xi_i)-M(\theta^*;\xi_i)-\nabla_\theta M(\theta^*;\xi_i)(\theta-\theta^*)\|
\leq K(\xi_i)\|\theta - \theta^*\|^2.
\end{equation}
Therefore, for any $\|\theta-\theta^*\|\le r_n$, averaging inequality~\eqref{eq:mean-val-m} over $i \in [n]$ bounds the right-hand side by $\frac{r_n^2}{2n} \sum_{i \in [n]} K(\xi_i)$. By the law of large numbers, $\frac{1}{n}\sum_{i \in [n]} K(\xi_i)\overset{p}{\to}\E[K(\xi)]<\infty$, hence the remainder is $O_p(r_n^2)= o_p(n^{-1/2})$ uniformly over $\|\theta - \theta^*\| \leq r_n$ for the corresponding choice of $r_n$. 
$\hfill \square$


\textit{Proof of \Cref{lemma:asym-perturbed-solution}.} We can decompose:
\begin{align*}
& \sqrt n\Para{\hat\theta_{H,M}-\theta_{H,M}^*}\\
= & \underbrace{\sqrt n\Para{G_n(\hat\theta_{EO})-G_n(\theta^*)}}_{(A)}
+ \underbrace{\sqrt n\Para{G_n(\theta^*)-G(\theta^*)}}_{(B)} + \underbrace{\sqrt{n}\Para{\hat\theta_{H,M} - G_n(\hat\theta_{EO})}}_{(C)},
\end{align*}
where $G_n(\theta):=\theta+H \widebar M(\theta), G(\theta):=\theta+H \mu_M(\theta)$.

For the part $(A)$, by a mean-value expansion and Lemma~\ref{ass:moments},
\[
G_n(\hat\theta_{EO})-G_n(\theta^*)
= \Para{I_{D_{\theta}\times D_{\theta}}+H \nabla_\theta\widebar M(\theta^*)}\Para{\hat\theta_{EO} -\theta^*} + r_n,
\quad \text{with}~\|r_n\|_2=o_p\Para{n^{-1/2}}.
\]
Multiplying by $\sqrt n$, we have:
\[
\sqrt n\Para{G_n(\hat\theta_{EO})-G_n(\theta^*)}
= \Para{I_{D_{\theta}\times D_{\theta}}+H \nabla_\theta\widebar M(\theta^*)}\sqrt n\Para{\hat\theta_{EO} -\theta^*} + o_p(1).
\]

For the part $(B)$, by definition:
\[
\sqrt n\Para{G_n(\theta^*)-G(\theta^*)}
= H \sqrt n\Para{\widebar M(\theta^*)-\mu_M(\theta^*)}.
\]
For the part $(C)$, for any $\varepsilon > 0$, 
\[\P\Para{\|\sqrt{n}\Para{\hat\theta_{H,M} - G_n(\hat\theta_{EO})}\| > \varepsilon} \leq \P(G_n(\hat\theta_{EO})\not\in \Theta) \to 0,\]
where the last step follows from the fact that $\theta^* \in \text{int}(\Theta)$, $\hat\theta_{EO} \in \Theta$ and $H\widebar M(\hat\theta_{EO}) = O_p(n^{-1/2})$ from the condition.

Combining $(A)$--$(C)$:
\begin{small}
    \begin{equation*}
\sqrt n\Para{\hat\theta_{H,M}-\theta_{H,M}^*}
= \Para{I_{D_{\theta}\times D_{\theta}}+H \nabla_\theta\widebar M(\theta^*)}\sqrt n\Para{\hat\theta_{EO} -\theta^*}
+ H \sqrt n\Para{\widebar M(\theta^*)-\mu_M(\theta^*)}
+ o_p(1).
    \end{equation*}
\end{small}
By Slutsky’s theorem and Lemma~\ref{ass:moments},
\[
\nabla_\theta\widebar M(\theta^*)=\frac1n\sum_{i=1}^n \nabla_\theta M(\theta^*;\xi_i)\xrightarrow{p} \nabla_\theta \mu_M(\theta^*)\in\mathbb R^{D_M\times D_{\theta}},
\]
and
\[
\frac1{\sqrt n}\sum_{i=1}^n
\begin{bmatrix}
\IFx(\xi_i)\\[2pt]
M(\theta^*;\xi_i)-\mu_M(\theta^*)
\end{bmatrix}
\ \Rightarrow\ 
\mathcal N\Para{0,
\begin{bmatrix}
\Sigma_{0} & \Gamma\\
\Gamma^\top & \Omega_M
\end{bmatrix}}.
\]
Applying the continuous mapping theorem, we have $\sqrt n(\hat\theta_{H,M}-\theta_{H,M}^*) \Rightarrow \mathcal N\big(0,\Sigma_H\big)$.

$\hfill \square$

We now return to the proof of \Cref{thm:improvement-principle} and calculate $\E[Z(\hat\theta_{H,M})] - \E[Z(\hat\theta_{EO})]$, where the expectation $\E[\cdot]$ is over the dataset $\Dscr_n$ unless specified.

When $\|H\mu_M(\theta^*)\| = \Theta(1)$, recall that $\theta_{H,M}^* = \Pi_{\Theta}\Para{\theta^* + H \mu_{M}(\theta^*)}$ and $\theta^* \in \text{int}(\Theta)$. If $\theta^* + H \mu_M(\theta^*) \in \text{int}(\Theta)$, then $Z(\theta_{H, M}^*) > Z(\theta^*)$ because $\theta^*$ is the unique minimizer of $Z(\theta)$. If instead $\theta^* + H \mu_{M}(\theta^*)$ is on or outside the boundary of $\Theta$, its projection $\theta_{H, M}^*$ lies on the boundary, and we still have $Z(\theta_{H,M}^*) - Z(\theta^*) > 0$. Because $\hat\theta_{H,M} \convp \theta_{H,M}^*$ and $\hat\theta_{EO}\convp \theta^*$, we have $\lim_{n \to \infty}[R(\hat\theta_{H,M}) - R(\hat\theta_{EO})] > 0$, which implies that $\hat\theta_{H, M}$ does not achieve any improvement over $\hat\theta_{EO}$.

When $H \mu_M(\theta^*) = o(1)$, a comparison of $Z(\theta_{H,M}^*)$ and $Z(\theta^*)$ yields:
        \begin{equation}\label{eq:objective-difference-H}
        Z(\theta_{H,M}^*) 
        = Z(\theta^*) + \frac{1}{2}(H\mu_M(\theta^*))^{\top} I_{\ell}(\theta^*) (H\mu_M(\theta^*)) + o(\|H\|^2).
        \end{equation}
Given the asymptotic normality of $\hat\theta_{H,M}$ and $\hat\theta_{EO}$ in \Cref{lemma:asym-perturbed-solution}, we apply a second-order Taylor expansion of $Z(\theta)$ at $\theta_{H,M}^*$ and $\theta^*$, respectively. Then we have:
\begin{align}
    \E[Z(\hat\theta_{H,M})] & = Z(\theta_{H,M}^*) + \frac{\text{Tr}[\Sigma_H I_{\ell}(\theta_{H,M}^*)]}{2n} + o\Para{\frac{1}{n}}, \label{eq:compare1-H}\\
    \E[Z(\hat\theta_{EO})] & =  Z(\theta^*) + \frac{\text{Tr}[\Sigma_{0} I_{\ell}(\theta^*)]}{2n} + o\Para{\frac{1}{n}}. \label{eq:compare2-H}
\end{align}

Therefore, combining~\eqref{eq:objective-difference-H},~\eqref{eq:compare1-H}, and~\eqref{eq:compare2-H} above, we have
\begin{equation}\label{eq:expectedperformance1-H}
    \E[Z(\hat\theta_{H,M})] =  \E[Z(\hat\theta_{EO})] + \frac{\rho(H)}{2n} + \frac{1}{2} \|H\mu_M(\theta^*)\|_{I_{\ell}(\theta^*)}^2 + o\Para{\frac{1}{n}},
\end{equation}
where $\rho(H) = \text{Tr}[\Sigma_H I_{\ell}(\theta_{H,M}^*) - \Sigma_{0} I_{\ell}(\theta^*)]$. 

We divide the discussion into the following cases:
\begin{enumerate}[leftmargin=*]
    \item $H\mu_M(\theta^*) = o(1)$, $H = o(1)$. We expand $\rho(H)$ around $H = \mathbf{0}$:
        \[
        \rho(H) = \rho(0) + \text{Tr}[H^{\top}I_{\ell}(\theta^*)A] + o(\|H\|) = \text{Tr}[H^{\top} I_{\ell}(\theta^*)A] + o(\|H\|),
        \]
        where $A = 2\Paran{\Sigma_0 \nabla_\theta\mu_M(\theta^*)^\top + \Gamma}$. Therefore,
        \begin{align*}
         \E[Z(\hat\theta_{H,M})] - \E[Z(\hat\theta_{EO})]
        = \frac{\text{Tr}[H^{\top}I_{\ell}(\theta^*) A]}{n}
        + \frac{1}{2} \|H\mu_M(\theta^*)\|_{I_{\ell}(\theta^*)}^2  
        + o(\|H\|^2) + o\Para{\frac{1}{n}},
        \end{align*}
        where the optimal solution to the right-hand side above is $H^* = -\frac{1}{n}\frac{A \mu_M(\theta^*)\mu_M(\theta^*)^{\top}}{\|\mu_M(\theta^*)\|^4}$, which satisfies $\|H^*\| = \Theta(1/n)$. Then $\E[Z(\hat\theta_{H^*, M}) - Z(\hat\theta_{EO})] = \Theta(1/n^2) < 0$, which is second-order.

    \item $H\mu_M(\theta^*) = 0$. Then $\theta_{H,M}^* = \theta^*$ and we expand: 
\[\rho(H) = \text{Tr}[H^{\top} I_{\ell}(\theta^*) A] + \text{Tr}[H^{\top} I_{\ell}(\theta^*) H B],\] where $B = \Cov_{\P}[\widetilde M(\theta^*;\xi)] = \nabla_\theta \mu_M(\theta^*) \Sigma_0\nabla_\theta\mu_M(\theta^*)^\top 
+ \nabla_\theta\mu_M(\theta^*) \Gamma 
+ \Gamma^\top \nabla_\theta\mu_M(\theta^*)^\top 
+ \Omega_M.$
Therefore, 
\[
\E[Z(\hat\theta_{H,M}) - Z(\hat\theta_{EO})]
=  \frac{\text{Tr}[H^{\top} I_{\ell}(\theta^*) A] + \text{Tr}[H^{\top} I_{\ell}(\theta^*) H B]}{2n} + o(n^{-1}).
\]
Here, $B \succeq 0$ as it is a covariance matrix. Therefore, the quadratic term in $H$ is positive semidefinite.
\begin{enumerate}[leftmargin = *]
    \item When $\mu_M(\theta^*) = 0$, $H \mu_M(\theta^*) = 0$ holds for any $H$. Therefore, the optimal choice of $H^* = -\frac{1}{2} A B^{\dagger}$ leads to an improvement of the magnitude $-\frac{1}{4}\tr[A B^{\dagger}A^{\top} I_{\ell}(\theta^*)]$. Hence, as long as $A \neq 0$, $\tr[A B^{\dagger}A^{\top} I_{\ell}(\theta^*)] \neq 0$ and $\hat\theta_{H^*, M}$ achieves a $\Theta(n^{-1})$ improvement.
    \item When $\mu_M(\theta^*) \neq 0$, we require $H \mu_M(\theta^*) = 0$; then finding the largest performance improvement reduces to solving the following constrained optimization problem:
    \[\min_H\{\text{Tr}[H^{\top} I_{\ell}(\theta^*) A] + \text{Tr}[H^{\top} I_{\ell}(\theta^*) H B], \text{s.t.}~H \mu_M(\theta^*) = 0.\}\]
    Here, for any $H_0$ with $\|H_0\| = \Theta(1)$ such that $H_0 \mu_M(\theta^*) = 0$, we consider $H = \lambda H_0$ and minimize the optimization problem above over $\lambda$. If $A(I_{D_M\times D_M} - \mu_M(\theta^*)\mu_M(\theta^*)^{\top}/\|\mu_M(\theta^*)\|_2^2)  \neq 0$, the optimal $\lambda^* = -\frac{\tr[H_0^{\top}I_{\ell}(\theta^*)A]}{2\tr[H_0^{\top}I_{\ell}(\theta^*)H_0 B]} \neq 0$ and the minimum of the optimization problem is positive. Otherwise, the minimum of the optimization problem is 0.  Therefore, $\hat\theta_{\lambda^* H_0, M}$ achieves a $\Theta(n^{-1})$ improvement.
\end{enumerate}
\end{enumerate}

$\hfill \square$

\subsection{\textit{Proof of \Cref{thm:unification}}}\label{app:formulation}
We summarize the additional regularity conditions under which \Cref{thm:unification} holds for each procedure.
\begin{assumption}[Other Regularity Conditions]\label{asp:additional-summary}
For the cost function $\ell(\theta;\xi)$, we have:
    \begin{itemize}
        \item $\ell(\theta;\xi)$ is twice continuously differentiable almost everywhere with respect to $\theta \in \Theta$;
        \item $\{\nabla_{\theta}\ell(\theta;\cdot): \|\theta - \theta^*\| \leq \delta\}$ is Donsker for some small $\delta > 0$, and $\{\nabla_{\theta}\ell(\theta;\cdot): \theta \in \Theta\}$ is Glivenko-Cantelli;
        \item $\nabla_{\theta}\ell(\theta;\xi)$ is Lipschitz continuous with respect to $\theta \in \Theta$, i.e., $\forall \theta_1, \theta_2 \in \Theta$, $\|\nabla_{\theta}\ell(\theta_1;\xi) - \nabla_{\theta}\ell(\theta_2;\xi)\|\leq L_2(\xi)\|\theta_1 - \theta_2\|$ with $\E_{\P}[L_2^2(\xi)] < \infty$.
    \end{itemize}
    For each EO+ method, assume that the procedure returns an empirical optimum satisfying the KKT conditions and that the following method-specific condition holds:
    \begin{enumerate}
        \item Explicit regularization methods: \Cref{asp:condition-explicitregularizer} holds.
        \item DRO methods:
        \begin{itemize}
            \item Wasserstein distance: \Cref{asp:condition-wdro} holds.
            \item $f$-divergence: $\ell(\theta;\xi)$ is uniformly bounded; $\max_{d(\Q, \hat\P_n)\leq \lambda}\E_{\Q}[\ell(\theta;\xi)] \neq \sup_{\hat\P_n}[\ell(\theta;\xi)]$ for any $\theta \in \Theta$. 
        \end{itemize}
        \item Shrinkage estimator: $\theta^* \neq \mathbf{0}$, and $G(\lambda) = 0$ if and only if $\lambda = 0$.
    \end{enumerate}
\end{assumption}

Appendices~\ref{app:explicitregularization}--\ref{app:shrinkage} provide the specific assumptions and detailed representations of EO+ methods corresponding to the form in \Cref{thm:unification}. Most of the EO+ representations come from the verification of \Cref{lemma:m-estimator0-aug}. Note that prior work~\citep{lam2021impossibility,gotoh2021calibration} focuses on a restrictive subset of DRO methods with $\lambda = o(1)$. For example, the standard $f$-divergence expansion of \cite{gotoh2021calibration} takes the form
\[
\hat\theta_{\lambda}
= \hat\theta_{EO} + \frac{\sqrt\lambda}{f''(1)} \cdot I_{\ell}(\theta^*)^{-1}
\text{Cov}_{\hat\P_n}\big(\ell(\theta^*;\xi),\nabla \ell (\theta^*;\xi)\big)
+ o(\lambda^{1/2}),
\]
which is a special case of \Cref{ex:auxiliary-robust}$(a)$ when $\lambda = o(1)$. In contrast, our results in \Cref{thm:unification} cover additional cases: (i) CVaR-DRO, which is a generalized form of the standard $f$-divergence; (ii) statistics-enhanced methods; and (iii) EO+ solutions with general $\lambda$, including constant-order $\lambda$.

\subsubsection{Explicit Regularization}\label{app:explicitregularization}
We next allow the regularization function to depend on the observed sample, i.e., $G:\Theta\times\Xi\to\R$. The resulting regularized estimator is defined as:
\begin{equation}\label{eq:reg-method2}
\hat\theta_{\lambda} \in \argmin_{\theta}\E_{\hat\P_n}[\ell(\theta;\xi) + \lambda G(\theta;\xi)].
\end{equation}
To characterize the statistical effect of this data-dependent regularization, we impose the following condition:
\begin{assumption}[Condition for Explicit Regularization Method]\label{asp:condition-explicitregularizer}
    For almost every $\xi$, $G(\theta;\xi)$ is twice differentiable in $\theta \in \Theta$, and $\nabla_{\theta}^2 G(\theta;\xi)$ is uniformly bounded in $\theta \in \Theta$.
\end{assumption}

\textit{Proof of the explicit-regularization case of \Cref{thm:unification}.}
We have two representations of the form~\eqref{eq:reg-method}, each under a different condition:

\underline{Representation (i).} Denote $\hat Z\Para{\theta}:=\E_{\hat\P_n}[\ell\Para{\theta;\xi}]$. Then the first-order condition at $\hat\theta_\lambda$ is:
\[\nabla_\theta \hat Z\Para{\hat\theta_\lambda} + \lambda \E_{\hat\P_n}{\nabla_\theta G\Para{\hat\theta_\lambda;\xi}} = 0.\] 
Let $\Delta:=\hat\theta_\lambda-\hat\theta_{EO}$. We take a Taylor expansion of both gradients of $\hat Z(\theta)$ and $\E_{\hat\P_n}[\nabla_{\theta} G(\theta;\xi)]$ at $\hat\theta_{EO}$:
\begin{align*}
\nabla_\theta \hat Z\Para{\hat\theta_\lambda}
& = \nabla_\theta \hat Z\Para{\hat\theta_{EO}} + \nabla_\theta^2 \hat Z\Para{\hat\theta_{EO}}\,\Delta + o(\Delta) = \nabla_\theta^2 \hat Z\Para{\hat\theta_{EO}}\,\Delta + o(\Delta),\\    
\E_{\hat\P_n}{\nabla_\theta G\Para{\hat\theta_\lambda;\xi}}
& = \E_{\hat\P_n}{\nabla_\theta G\Para{\hat\theta_{EO};\xi}}
+ \E_{\hat\P_n}{\nabla_\theta^2 G\Para{\hat\theta_{EO};\xi}}\,\Delta + o(\Delta).
\end{align*}
Substituting these expansions into the first-order condition and retaining only the leading terms gives:
\[
\Delta
= -\Big(\nabla_\theta^2 \hat Z\Para{\hat\theta_{EO}} + \lambda\, \E_{\hat\P_n}{\nabla_\theta^2 G \Para{\hat\theta_{EO};\xi}}\Big)^{-1}\,
\lambda\,\E_{\hat\P_n}{\nabla_\theta G\Para{\hat\theta_{EO};\xi}}
+ o_p\Para{\lambda \vee n^{-1/2}}.
\]
Replacing the sample quantities by their population counterparts at $\theta^*$ gives the following expansion, where $\nabla_\theta^2 \hat Z\Para{\hat\theta_{EO}} = I_{\ell}(\theta^*) +o_p(1)$ and $\E_{\hat\P_n}{\nabla_\theta G\Para{\hat\theta_{EO};\xi}} = \frac{1}{n}\sum_{i=1}^n \nabla_\theta G\Para{\theta^*;\xi_i} + o_p(1)$:
\[
\Delta
= \frac{1}{n}\sum_{i=1}^n \Big(-\lambda\, (I_{\ell}(\theta^*) + \lambda \E_{\P}[\nabla_{\theta}^2 G(\theta^*;\xi)])^{-1}\Big)
\nabla_\theta G\Para{\theta^*;\xi_i}+ o_p\Para{\lambda \vee n^{-1/2}}.
\]
In this case, 
\begin{align*}
    M(\theta;\xi) & = \nabla_{\theta} G(\theta^*;\xi) \in \R^{D_{\theta}},\\
    H(\lambda) & = -\lambda\, (I_{\ell}(\theta^*) + \lambda \E_{\P}[\nabla_{\theta}^2 G(\theta^*;\xi)])^{-1} \in \R^{D_{\theta}\times D_{\theta}}.
\end{align*}

\underline{Representation (ii).} When $\theta_{\lambda}^* = \theta^*$, 
comparing the asymptotic expansion of $\hat\theta_{EO}$ and $\hat\theta_{\lambda}$ in Lemmas~\ref{lemma:m-estimator0} and~\ref{lemma:m-estimator0-aug}, we obtain:
\begin{equation}\label{eq:m-estimator-diff}
\begin{aligned}
        \hat\theta_{EO} & = \theta^* - \frac{1}{n}\sum_{i = 1}^n I_{\ell}(\theta^*)^{-1} \nabla_{\theta} \ell(\theta^*;\xi_i) + o_p(n^{-1/2}),\\
    \hat\theta_{\lambda} & = \theta_{\lambda}^* - \frac{1}{n}\sum_{i = 1}^n (I_{\ell}(\theta^*) + \lambda \E_{\P}[\nabla_{\theta}^2 G(\theta_{\lambda}^*;\xi)])^{-1}(\nabla_{\theta}\ell(\theta_{\lambda}^*;\xi_i) + \lambda \nabla_{\theta} G(\theta_{\lambda}^*;\xi_i)) + o_p(n^{-1/2}).
\end{aligned}    
\end{equation}
Then we obtain:
\begin{small}
\begin{align*}
    M(\theta;\xi)
    &=\begin{bmatrix}
    \nabla_{\theta}\ell(\theta^*;\xi)\\
    \nabla_{\theta}G(\theta^*;\xi)
    \end{bmatrix}
    \in\R^{2D_{\theta}},\\
    H(\lambda)
    &=-\begin{bmatrix}
    (I_{\ell}(\theta^*)+\lambda\E_{\P}[\nabla_{\theta}^2G(\theta^*;\xi)])^{-1}
    -I_{\ell}(\theta^*)^{-1}
    &
    \lambda(I_{\ell}(\theta^*)+\lambda\E_{\P}[\nabla_{\theta}^2G(\theta^*;\xi)])^{-1}
    \end{bmatrix}
    \in\R^{D_{\theta}\times 2D_{\theta}},\\
    r_{\lambda, n} & = o_p(n^{-1/2})~\text{independent of }\lambda.
\end{align*}
\end{small}

$\hfill \square$

This explicit regularization formulation covers data-independent regularizers of the form $G(\theta;\xi)=G(\theta)$. Due to the close connections between explicit regularization and distributionally robust optimization (DRO)~\citep{gao2022wasserstein,lam2018sensitivity}, the explicit regularization formulation lays the foundation for the following DRO formulations, including Wasserstein distance and $f$-divergences.

\subsubsection{Wasserstein-Distance-Based DRO}
\begin{definition}[Wasserstein Distance]\label{def:was-distance}
For two distributions $\P,\Q\in \Pscr(\Xi)$, the $p$-Wasserstein distance $(p\in \mathbb{N})$ with respect to the $l_1$-norm (i.e., $\|\cdot\|_1$) is defined as:
\begin{equation}\label{eq:pwasserstein-dist}
\Wscr_p(\P,\Q) \coloneqq \inf_{\Fscr \in \Pi(\P, \Q)} \left( \E_{(Y_1, Y_2)\sim \Fscr}\Big[\|Y_1 - Y_2\|_1^p \Big] \right)^{\frac{1}{p}} ,
\end{equation}
where $\Pi(\P, \Q)$ denotes the set of all joint distributions with marginals $\P$ and $\Q$.
\end{definition}
\begin{assumption}[Conditions for Wasserstein DRO]\label{asp:condition-wdro}
    For a given order $p$ of Wasserstein distance, assume the following.
    \begin{itemize}
        \item There exist $L_1, \delta > 0$, such that $\forall \theta \in \Theta$, $|\ell(\theta;\xi_1) - \ell(\theta;\xi_2)|\leq L_1 \|\xi_1 - \xi_2\|^p, \forall \|\xi_1 - \xi_2\|\leq \delta$;
        \item There exists $L_2 > 0$, such that $\forall \theta \in \Theta$, $\|\nabla_{\xi}\ell(\theta;\xi_1)- \nabla_{\xi}\ell(\theta;\xi_2)\|_{\infty}\leq L_2 \|\xi_1 - \xi_2\|$ for $\xi_1, \xi_2 \in \Xi$;
        \item $\forall \theta \in \Theta$, $\E_{\P}[\|\nabla_{\xi}\ell(\theta;\xi)\|_{\infty}^{\max\{q, 2q - 2\}}]<\infty$, where $q = \frac{p}{p - 1}\mathbf{1}_{p > 1} + \infty\mathbf{1}_{p = 1}$.
    \end{itemize}
\end{assumption}
Under this assumption, Theorem 8.7 in \cite{kuhn2024distributionally} implies that Wasserstein DRO methods reduce to a special case of explicit regularization methods, with:
\[\sup_{\Q: W_p(\Q, \hat\P_n)\leq \lambda}\E_{\Q}[\ell(\theta;\xi)] = \E_{\hat\P_n}[\ell(\theta;\xi)] + \lambda\E_{\hat\P_n}[\|\nabla_{\xi} \ell(\theta;\xi)\|_{*}^q]^{\frac{1}{q}} + o(\lambda).\]
Consequently, the $p$-Wasserstein-DRO solution is a special case of the explicit regularization method in Appendix~\ref{app:explicitregularization} with:
\begin{align*}
    H(\lambda) &= -\frac{\lambda}{q} I_{\ell}(\theta^*)^{-1}(\E_{\P}[\|\nabla_{\xi}\ell(\theta^*;\xi)\|_*^q])^{\frac{1}{q} - 1} \in \R^{D_{\theta}\times D_{\theta}},\\
    M(\theta;\xi) & =\nabla_{\theta}\|\nabla_{\xi}\ell(\theta^*;\xi)\|_*^q \in \R^{D_{\theta}}.
\end{align*}

\subsubsection{$\chi^2$-Divergence-Based and Standard $f$-Divergence-Based DRO}\label{app:chi2dro}
\begin{definition}[$f$-divergence]\label{def:f-divergence}
    Let $\P$ and $\Q$ be distributions such that $\P$ is absolutely continuous with respect to $\Q$. Let
  $f:[0,\infty) \to (-\infty, \infty]$ be a convex function that is finite for every $x > 0$ and satisfies $f(1) = 0$. The $f$-divergence between $\P$ and $\Q$ is
  defined by
  \begin{equation}\label{eq:f-divergence}
          d_f(\P, \Q) = \int f\left(\frac{d\P}{d\Q}\right)d\Q =
    \E_{\Q}\left[f\left(\frac{d\P}{d\Q}\right)\right].
  \end{equation}
  The $\chi^2$-divergence corresponds to $f(x) = (x - 1)^2$.  
\end{definition}

\textit{Proof of \Cref{ex:auxiliary-robust}.}
Under \Cref{asp:additional-summary}, Proposition 1 in \cite{lam2018sensitivity} implies that, whenever the nonnegativity constraints on the empirical likelihood ratios are inactive, the inner objective of the $\chi^2$-divergence can be written as
\begin{equation}\label{eq:chi2-dro-dual}
    \max_{\chi^2(\Q, \hat\P_n)\leq \lambda}\E_{\Q}[\ell(\theta;\xi)] = \E_{\hat\P_n}[\ell(\theta;\xi)] + \sqrt{\lambda\text{Var}_{\hat\P_n}[\ell(\theta;\xi)]}.
\end{equation}
Therefore, the DRO objective can be written as
\[\min_{\theta}\max_{\Q: \chi^2(\Q, \hat\P_n)\leq \lambda} \E_{\Q}[\ell(\theta;\xi)] =\min_{\theta, u} \paran{\frac{1}{n}\sum_{i = 1}^n \ell(\theta;\xi_i) + \sqrt{\frac{\lambda}{n}\sum_{i = 1}^n (\ell(\theta;\xi_i)- u)^2}},\]
with an additional decision variable $u$. The optimal value is $u = \frac{1}{n}\sum_{i = 1}^n \ell(\theta;\xi_i)$.

\underline{Representation (i).} We have:
\begin{align*}
M(\theta;\xi)&=\big(\ell(\theta^*;\xi)-\E_{\P}[\ell(\theta^*;\xi)]\big) \nabla_\theta\ell(\theta^*;\xi)\in\R^{D_\theta},\\
H(\lambda)&=- I_{\ell}(\theta^*)^{-1} \frac{\sqrt{\lambda}}{\sqrt{\text{Var}_{\P}[\ell(\theta^*;\xi)]}}\in\R^{D_\theta\times D_\theta}.
\end{align*}

\underline{Representation (ii).} Suppose $\theta_{\lambda}^* = \theta^*$, which, by the population first-order condition of the dual reformulation~\eqref{eq:chi2-dro-dual} together with $\nabla_{\theta} Z(\theta^*) = 0$, is equivalent to the moment condition
\begin{equation}\label{eq:chi2-centering}
\Cov_{\P}\Para{\ell(\theta^*;\xi),\, \nabla_{\theta}\ell(\theta^*;\xi)} = \mathbf{0}.
\end{equation}
Denote $\bar\ell^* := \E_{\P}[\ell(\theta^*;\xi)]$, $\sigma_{\ell}^2 := \Var_{\P}[\ell(\theta^*;\xi)] > 0$, and
\begin{align*}
    G_0(\xi) & := \frac{\Para{\ell(\theta^*;\xi) - \bar\ell^*}\nabla_{\theta} \ell(\theta^*;\xi)}{\sigma_{\ell}} \in \R^{D_{\theta}},\\
    G_1(\xi) & := \frac{1}{\sigma_{\ell}}\Para{\nabla_{\theta}\ell(\theta^*;\xi)\nabla_{\theta}\ell(\theta^*;\xi)^{\top} + \Para{\ell(\theta^*;\xi) - \bar\ell^*}\nabla_{\theta}^2 \ell(\theta^*;\xi)} \in \R^{D_{\theta}\times D_{\theta}}.
\end{align*}
After augmenting the decision variables $(\theta, u)$ with the variance level $v$, we obtain a system of optimality conditions expressed as exact sample-average estimating equations. Applying \Cref{lemma:m-estimator0-aug}, the Jacobian blocks coupling $\theta$ with the nuisances $(u, v)$ vanish under~\eqref{eq:chi2-centering}. Hence, for each fixed $\lambda$ such that $I_{\ell}(\theta^*) + \sqrt{\lambda}\,\E_{\P}[G_1(\xi)]$ is invertible,
\begin{small}
\begin{equation}\label{eq:ch2-div-unification}
    \hat\theta_{\lambda} - \hat\theta_{EO} \;=\; -\sqrt{\lambda}\, \Para{I_{\ell}(\theta^*) + \sqrt{\lambda}\,\E_{\P}[G_1(\xi)]}^{-1}\cdot \frac{1}{n}\sum_{i = 1}^n (G_0(\xi_i) + \E_{\P}[G_1(\xi)]\, \IFx(\xi_i)) \;+\; o_p(n^{-1/2}).
\end{equation}
\end{small}
In the form of \Cref{thm:unification}, this corresponds to $r_{\lambda, n} = o_p(n^{-1/2})$, 
\begin{align*}
M(\theta;\xi) &= G_0(\xi) + \E_{\P}[G_1(\xi)]\IFx(\xi) \in \R^{D_{\theta}},\\
H(\lambda) &= -\sqrt{\lambda}\, (I_{\ell}(\theta^*) + \sqrt{\lambda}\E_{\P}[G_1(\xi)])^{-1} \in \R^{D_{\theta}\times D_{\theta}}.
\end{align*}

We make the following two comments. First, under~\eqref{eq:chi2-centering} the side information is automatically correct: $\E_{\P}[G_0(\xi)] = \Cov_{\P}[\ell(\theta^*;\xi), \nabla_{\theta}\ell(\theta^*;\xi)]/\sigma_{\ell} = \mathbf{0}$ and $\E_{\P}[\IFx(\xi)] = \mathbf{0}$, so $\E_{\P}[M(\theta^*;\xi)] = \mathbf{0}$, and the conditions of \Cref{thm:improvement-principle} reduce to the non-orthogonality condition $\Cov_{\P}[\IFx(\xi), M(\theta;\xi)] \neq \mathbf{0}$. Second, for small $\lambda$ we have $H(\lambda) = -\sqrt{\lambda}\, I_{\ell}(\theta^*)^{-1} + O(\lambda)$, recovering Representation (i) up to $o_p(\sqrt{\lambda} \vee n^{-1/2})$; for constant-order $\lambda$, $H(\lambda)$ remains bounded rather than diverging, a regime invisible to the $\lambda = o(1)$ expansions of \citet{gotoh2021calibration,lam2021impossibility}.

Beyond the DRO formulations, \citet{jiang2024distributionally} propose a \emph{distributionally favorable} framework, which replaces the min-max formulation by a min-min formulation over the decision variable and the distribution. This can also be incorporated into Theorem~\ref{thm:unification}: in the $\chi^2$-divergence case, the second term on the right-hand side of the equivalence~\eqref{eq:chi2-dro-dual} above changes sign from
\(+\sqrt{\lambda\,\text{Var}_{\hat\P_n}[\ell(\theta;\xi)]}\) 
to 
\(-\sqrt{\lambda\,\text{Var}_{\hat\P_n}[\ell(\theta;\xi)]}\).

This result for $\chi^2$-divergence-based DRO under representation (i) can be generalized to the standard $f$-divergence-based DRO in~\eqref{eq:f-divergence}. Under \Cref{asp:additional-summary}, Theorem 2 of \cite{duchi2021statistics} gives
\[\max_{d_f(\Q, \hat\P_n)\leq \lambda}\E_{\Q}[\ell(\theta;\xi)] = \E_{\hat\P_n}[\ell(\theta;\xi)] + \sqrt{\frac{\lambda f''(1)}{2}\text{Var}_{\hat\P_n}[\ell(\theta;\xi)]} + o(\sqrt{\lambda}).\]
Therefore, we have:
\begin{align*}
M(\theta;\xi) & = \big(\ell(\theta^*;\xi)-\E_{\P}[\ell(\theta^*;\xi)]\big)\nabla_{\theta}\ell(\theta^*;\xi) \in \R^{D_{\theta}}\\   
H(\lambda)&=- I_{\ell}(\theta^*)^{-1} \frac{\sqrt{\lambda f''(1)/2}}{\sqrt{\text{Var}_{\P}[\ell(\theta^*;\xi)]}}\in\R^{D_\theta\times D_\theta}.
\end{align*}

$\hfill \square$

\subsubsection{Transfer Learning}
Recall that in the transfer learning setup, $\theta_0$ is  a fixed parameter estimated from other data sources and the estimator is: 
\[
\hat\theta_{\lambda}\in
\argmin_{\theta\in\Theta}\Para{\frac{1}{n}\sum_{i=1}^n \ell(\theta;\xi_i)
+ \lambda\|\theta-\theta_0\|_{2}^2}.
\]
This is the same as explicit regularization~\eqref{eq:reg-method} with $G(\theta;\xi)=\|\theta-\theta_0\|_2^2$ described in~\Cref{ex:auxiliary-stat}.

\subsubsection{Shrinkage Estimator}\label{app:shrinkage}
Recall that the shrinkage estimator $\hat\theta_{\lambda}
= \Big(1 + \tfrac{G(\lambda)}{\|\hat\theta_{EO}\|_2^2}\Big)\hat\theta_{EO}$, where $G(\lambda)$ is a fixed function of $\lambda$. 

\textit{Proof of \Cref{ex:auxiliary-stat}.} From the standard influence function expansion $\hat\theta_{EO} = \theta^* + \frac{1}{n}\sum_{i=1}^n \IFx(\xi_i) +  o_p\left(\frac{1}{\sqrt{n}}\right)$, squaring both sides gives: $\|\hat\theta_{EO}\|_2^2 
= \|\theta^*\|_2^2 + 2\theta^{*\top}\left(\frac{1}{n}\sum_{i=1}^n \IFx(\xi_i)\right) 
+ o_p\left(\frac{1}{\sqrt{n}}\right)$. Therefore:
\begin{align*}
\frac{1}{\|\hat\theta_{EO}\|_2^2}
&= \frac{1}{\|\theta^*\|_2^2} 
- \frac{2\theta^{*\top}\left(\frac{1}{n}\sum_{i=1}^n \IFx(\xi_i)\right)}{\|\theta^*\|_2^4} 
+ o_p\left(\frac{1}{\sqrt{n}}\right).
\end{align*}
Moreover, 
\begin{align*}
\hat\theta_\lambda - \hat\theta_{EO}
& = \frac{G(\lambda)}{\|\hat\theta_{EO}\|_2^2}\cdot \hat\theta_{EO}\\
& = \frac{G(\lambda)}{\|\theta^*\|_2^2}\Para{1 - \frac{2\theta^{*\top} \Para{\frac{1}{n}\sum_{i = 1}^n \IFx(\xi_i)}}{\|\theta^*\|_2^2}}\Para{\theta^* + \frac{1}{n}\sum_{i =1}^n \IFx(\xi_i)}
+ o_p\left(\frac{1}{\sqrt{n}}\right).\\
& = \frac{G(\lambda)}{\|\theta^*\|_2^2}\Para{\theta^* + \Para{I_{D_{\theta}\times D_{\theta}} - \frac{2\theta^*\theta^{*T}}{\|\theta^*\|_2^2}}\frac{1}{n}\sum_{i = 1}^n \IFx(\xi_i) + O_p\Para{\frac{1}{n}}} + o_p\Para{\frac{1}{\sqrt{n}}}.
\end{align*}
This implies:
\[ M(\theta;\xi_i) 
:=\theta^* + \Para{I_{D_{\theta}\times D_{\theta}} -\frac{2\theta^* \theta^{*\top}}{\|\theta^*\|_2^2}} \IFx(\xi_i),
\qquad 
H(\lambda) :=  G(\lambda)/\|\theta^*\|_2^2, \quad r_{\lambda, n} = o_p(n^{-1/2}).
\]

$\hfill \square$

In conclusion, these casewise representations establish \Cref{thm:unification}. $\hfill \square$

\subsection{\textit{Proof of \Cref{coro:linear-symmetric}}}
For the OLS loss with $X \equiv 1$, $\ell(\theta;\xi) = (\theta - \xi)^2$, we have $\nabla_{\theta}\ell(\theta^*;\xi) = -2\epsilon$, $I_{\ell}(\theta^*) = 2$, and hence $\IFx(\xi) = \epsilon$. Denote $\sigma^2 := \E[\epsilon^2]$ and $\mu_k := \E[\epsilon^k]$, and note that $\ell(\theta^*;\xi) = \epsilon^2$ and $\sigma_{\ell}^2 = \Var(\epsilon^2) = \mu_4 - \sigma^4$.

First, the centering condition~\eqref{eq:chi2-centering} reads $\Cov(\epsilon^2, -2\epsilon) = -2\mu_3 = 0$. For the location-shifted exponential distribution, $\mu_3 \neq 0$, so $\theta_{\lambda}^* \neq \theta^*$ and the $\chi^2$-based perturbation carries a nonvanishing bias; equivalently, in Representation (i), $\mu_M(\theta^*) \propto \mu_3 \neq 0$, and a first-order improvement is not attainable under that noise type by \Cref{thm:improvement-principle}.

For symmetric noise, $\mu_3 = 0$ and~\eqref{eq:chi2-centering} holds. Specializing~\eqref{eq:ch2-div-unification}, we compute $G_0(\xi) = -2(\epsilon^3 - \sigma^2\epsilon)/\sigma_{\ell}$ and $G_1(\xi) = (6\epsilon^2 - 2\sigma^2)/\sigma_{\ell}$ with $\E[G_1(\xi)] = 4\sigma^2/\sigma_{\ell}$, so that
\[
\hat\theta_{\lambda} - \hat\theta_{EO} = \frac{\sqrt{\lambda}}{\sigma_{\ell} + 2\sqrt{\lambda}\,\sigma^2}\cdot \frac{1}{n}\sum_{i = 1}^n \Para{\epsilon_i^3 - 3\sigma^2\epsilon_i} + o_p(n^{-1/2}).
\]
Thus the perturbation direction is $M(\theta;\xi) \propto \epsilon^3 - 3\sigma^2\epsilon$, which is correct, i.e., $\E[M(\theta^*;\xi)] \propto \mu_3 = 0$, for all symmetric noise distributions. For the non-orthogonality condition, we calculate
\[
\E[M(\theta^*;\xi)^{\top}\IFx(\xi)] \propto \E[\epsilon^4 - 3\sigma^2\epsilon^2] = \mu_4 - 3\sigma^4.
\]
\begin{itemize}
    \item When $\epsilon$ follows a normal distribution, $\mu_4 = 3\sigma^4$ and $\E[M(\theta^*;\xi)^{\top}\IFx(\xi)] = 0$: under the normal distribution, $\epsilon^3 - 3\sigma^2\epsilon$ is orthogonal to $\epsilon$ by Hermite orthogonality, so no first-order improvement is available.
    \item When $\epsilon$ follows a symmetric distribution with finite fourth moment and $\mu_4 \neq 3\sigma^4$ (e.g., Laplace, uniform, or $t$-distribution with more than four degrees of freedom), $\E[M(\theta^*;\xi)^{\top}\IFx(\xi)] \neq 0$ and a first-order improvement is attainable.
\end{itemize}
For the LAD loss $\ell(\theta;\xi) = |\theta - \xi|$, we obtain $\IFx(\xi) = \frac{\text{sgn}(\epsilon)}{2f(\theta^*)}$. For the $\chi^2$-divergence EO+ solution, we have
\[M(\theta^*;\xi) = \frac{\text{sgn}(\epsilon) {\sqrt{\text{Var}(|\epsilon|)}} + \sqrt{\lambda}(\epsilon - \E[|\epsilon|]\text{sgn}(\epsilon))}{2f(\theta^*) {\sqrt{\text{Var}(|\epsilon|)}} + \sqrt{\lambda}(1 - 2 \E[|\epsilon|] f(\theta^*))} - \frac{\text{sgn}(\epsilon)}{2f(\theta^*)}.\] For the location-shifted exponential distribution, $\E[M(\theta^*;\xi)] \neq 0$. This implies that first-order improvement is not attainable under that noise type.  For other noises, $\E[M(\theta^*;\xi)] = 0$. 

Furthermore, we calculate:
$\E[M(\theta^*;\xi)^{\top}\IFx(\xi)] \propto 1 - 2 \E[|\epsilon|] f(\theta^*)$. 
\begin{itemize}
    \item When $\epsilon \sim \text{Lap}(0, b)$, $\E[|\epsilon|] = b$ and $f(\theta^*) = \frac{1}{2b}$; therefore, $\E[M(\theta^*;\xi)^{\top} \IFx(\xi)] = 0$;
    \item When $\epsilon$ follows other symmetric distributions (Normal, $t$-distribution, Uniform distribution), $1 - 2 \E[|\epsilon|] f(\theta^*) \neq 0$; therefore, $\E[M(\theta^*;\xi)^{\top} \IFx(\xi)] \neq 0$. 
\end{itemize}

$\hfill \square$

\subsection{\textit{Proof of \Cref{coro:eo-plus-improvement}}}\label{app:proof-eo-plus}
We show that, for each EO+ formulation in \Cref{defn:problemclass}, the condition $H(\lambda)\mu_M(\theta^*) = 0$ implies $\mu_M(\theta^*) = 0$ when $\lambda \neq 0$. 

For all methods whose explicit representations are derived in Appendices~\ref{app:explicitregularization} to~\ref{app:shrinkage}, the adjustment matrix takes the form $H(\lambda) = G(\lambda)\tilde G_{\lambda}^{-1}$, where $G(\lambda) = \lambda^q$ for some $q > 0$ and $\tilde G_{\lambda}$ is an invertible matrix. 

Then if $H(\lambda)\mu_M(\theta^*) = \lambda^q \tilde G_{\lambda}^{-1}\mu_M(\theta^*) = 0$, multiplying both sides by $\tilde G_{\lambda}$ yields $\lambda^q \mu_M(\theta^*) = 0$. Since $\lambda \neq 0$, it follows that $\mu_M(\theta^*) = 0$. The result then follows from the same analyses of Theorem~\ref{thm:improvement-principle} after replacing~\eqref{eq:empirical-solution-adjustment} with the formulation of $\hat\theta_{\lambda}$ in~\eqref{eq:eo-plus-representation} of \Cref{thm:unification}.

$\hfill \square$


\subsection{Connections between Directionally Perturbed EO Solutions and CVaR-DRO} \label{app:cvardro}
We next consider functions $f(\cdot)$ satisfying \Cref{def:f-divergence} but not necessarily the standard smooth form. This class includes generalized divergences associated with Conditional Value-at-Risk (CVaR)~\citep{rockafellar2000optimization,duchi2021learning,sahoo2022learning}. For example, in~\eqref{eq:f-divergence}, by setting $f(x) = 0$ if $x \in [1 - \lambda, \frac{1}{1-\lambda}]$ and $\infty$ otherwise, we obtain the CVaR objective, which can be formulated as:
\begin{equation}\label{eq:cvar-objective}
    (\hat\theta_{\lambda}, \hat\eta_{\lambda}) =\argmin_{\theta,\eta}\;\; \eta + \frac{1}{1-\lambda}\,\E_{\hat\P_n}\big[(\ell(\theta;\xi)-\eta)^+\big].
\end{equation}
\begin{assumption}[Condition for CVaR-DRO]\label{asp:condition-cvar}
Let $(\theta_\lambda^*,\eta_\lambda^*)$ denote the population solution obtained by replacing $\hat\P_n$ with $\P$ in~\eqref{eq:cvar-objective}. Let $\eta_\lambda^*$ be the $\lambda$-quantile of $\ell(\theta_{\lambda}^*;\xi)$ under $\P$. Assume the following conditions.
\begin{itemize}
    \item $\theta_{\lambda}^* = \theta^*$;
    \item There exists a unique $(\theta_{\lambda}^*, \eta_{\lambda}^*)$ such that $\E[\psi(\theta_\lambda^*,\eta_\lambda^*;\xi)]=0$, where $\psi(\theta,\eta;\xi)
:=\begin{bmatrix}
\frac{1}{1-\lambda}\,\mathbf{1}_{\{\ell(\theta;\xi)>\eta\}}\;\nabla_\theta \ell(\theta;\xi)\\[4pt]
1-\frac{1}{1-\lambda}\,\mathbf{1}_{\{\ell(\theta;\xi)>\eta\}}
\end{bmatrix}$;
    \item Assume $J_\lambda:=\nabla_{(\theta,\eta)}\E[\psi(\theta_\lambda^*,\eta_\lambda^*;\xi)]$ is invertible. Partition it as $J_\lambda=\left[\begin{smallmatrix}J_{\theta\theta,\lambda}&
    J_{\theta\eta,\lambda}\\J_{\eta\theta,\lambda}&
    J_{\eta\eta,\lambda}\end{smallmatrix}\right]$. Also assume
$J_{\theta\eta,\lambda}=\mathbf 0$ and
$J_{\eta\theta,\lambda}=\mathbf 0$;
    \item  The random variable $\ell(\theta^*;\xi)$ has a continuous distribution in a neighborhood of the quantile threshold $\eta_{\lambda}^*$;
\end{itemize}
\end{assumption}
\begin{theorem}[Representation and Improvement Characterization of CVaR-EO+ Solution]\label{thm:cvar-unification}
Suppose Assumptions~\ref{asp:theta-star0},~\ref{asp:optimal-condition}, and~\ref{asp:condition-cvar} holds. Suppose the CVaR-EO+ solution $\hat\theta_{\lambda}$ is computed by: (i) the solution to~\eqref{eq:cvar-objective}, if $\lambda \geq 0$; (ii) $2\hat\theta_{EO} - \hat\theta_{-\lambda}$, if $\lambda < 0$. Then we have:
\[\hat\theta_{\lambda} = \hat\theta_{EO} + \frac{1}{n}\sum_{i = 1}^n M_{\lambda}(\xi_i) + o_p(n^{-1/2}),\]
where $M_{\lambda}(\xi)$ is a function that depends on $\xi$ and $\lambda$. Moreover, $\hat\theta_{\lambda}$ achieves a first-order improvement over $\hat\theta_{EO}$ only if $\lambda = \Theta(1)$.
\end{theorem}
Because the perturbation function is complicated, characterizing improvements for constant-order $\lambda$ requires problem-specific calculations. We provide the linear-regression calculations in \Cref{app:numerical-lr}.

\textit{Proof of \Cref{thm:cvar-unification}.} For $\lambda > 0$, 
following \Cref{lemma:m-estimator0-aug}, the influence function for
$(\hat\theta_{\lambda},\hat\eta_{\lambda})$ is
$-J_\lambda^{-1}\psi(\theta_\lambda^*,\eta_\lambda^*;\xi)$. For the
representation in \Cref{ex:auxiliary-robust}$(b)$, under \Cref{asp:condition-cvar}, writing
$J_\lambda=\left[\begin{smallmatrix}J_{\theta\theta,\lambda}&
J_{\theta\eta,\lambda}\\J_{\eta\theta,\lambda}&
J_{\eta\eta,\lambda}\end{smallmatrix}\right]$, we have
$J_{\theta\eta,\lambda}=\mathbf 0$ and
$J_{\eta\theta,\lambda}=\mathbf 0$. Hence the $\theta$ score is locally
decoupled from the $\eta$ score.
Define the unscaled $\theta$-Jacobian of the tail score by
\[
I_{\ell,\lambda}(\theta^*)
:=\left.\nabla_\theta\E_{\P}\!\left[
\mathbf{1}_{\{\ell(\theta;\xi)>\eta_\lambda^*\}}
\nabla_\theta\ell(\theta;\xi)
\right]\right|_{\theta=\theta_\lambda^*}.
\]
Then the $\theta\theta$ block of $J_\lambda$ is
$I_{\ell,\lambda}(\theta^*)/(1-\lambda)$. Consequently, the factors $1-\lambda$ cancel when
this block is inverted, and the influence function of $\hat\theta_\lambda$ is
\[
\IFx_\lambda(\xi)
=-I_{\ell,\lambda}(\theta^*)^{-1}
\mathbf{1}_{\{\ell(\theta_\lambda^*;\xi)>\eta_\lambda^*\}}
\nabla_\theta\ell(\theta_\lambda^*;\xi).
\]

Since $\theta_\lambda^*=\theta^*$, the asymptotic linear expansions of the
CVaR and EO estimators give
\begin{equation}\label{eq:cvar-unify}
\begin{aligned}
\hat\theta_\lambda-\hat\theta_{EO}
&\;=\; \frac{1}{n}\sum_{i=1}^n
 {\Big[I_{\ell}(\theta^*)^{-1}-I_{\ell,\lambda}(\theta^*)^{-1}
 \mathbf{1}_{\{\ell(\theta^*;\xi_i)>\eta_\lambda^*\}}\Big]}
 {\nabla_\theta \ell(\theta^*;\xi_i)}
 \;+\; o_p(n^{-1/2}),
\end{aligned}
\end{equation}
which implies $M_{\lambda}(\xi)
=\left(I_{\ell}(\theta^*)^{-1}
-I_{\ell,\lambda}(\theta^*)^{-1}\mathbf{1}_{\{\ell(\theta^*;\xi)>\eta_\lambda^*\}}\right)
\nabla_\theta\ell(\theta^*;\xi)$ when $\lambda > 0$. For $\lambda < 0$, one can easily compute $M_{\lambda}(\xi) = \left(I_{\ell,-\lambda}(\theta^*)^{-1}\mathbf{1}_{\{\ell(\theta^*;\xi)>\eta_{-\lambda}^*\}} - I_{\ell}(\theta^*)^{-1}\right)
\nabla_\theta\ell(\theta^*;\xi)$.

Because CVaR-EO+ has the same perturbation form as \Cref{defn:direction-perturb}, the analysis follows the $\mu_M(\theta^*) = 0$ case in \Cref{thm:improvement-principle} and gives
\[\E[Z(\hat\theta_{\lambda}) - Z(\hat\theta_{EO})] = \frac{\tr[I_{\ell}(\theta^*)\Gamma(\lambda)] + \tr[I_{\ell}(\theta^*)\Var_{\P}[M_{\lambda}(\xi)]]}{2n},\]
where the perturbation function is denoted as $M_{\lambda}(\xi)$ and the covariance matrix $\Gamma(\lambda)$ can be computed as follows:
\begin{align*}
\Gamma(\lambda) & = \Cov_{\P}[\IFx(\xi), M_{\lambda}(\xi)] \\
& = \E_{\P}[\IFx(\xi)M_{\lambda}(\xi)^{\top}]\\
& = \E_{\P}[\IFx(\xi)\IFx(\xi)^{\top}] - \E_{\P}[\nabla_{\theta}\ell(\theta^*;\xi)(I_{\ell}(\theta^*)I_{\ell, \lambda}(\theta^*))^{-1}\mathbf{1}_{\{\ell(\theta^*;\xi) > \eta_{\lambda}^*\}}\nabla_{\theta}\ell(\theta^*;\xi)^{\top}].
\end{align*}
Above, $\lim_{\lambda \to 0}\mathbf{1}_{\{\ell(\theta^*;\xi) > \eta_{\lambda}^*\}} = 1$ and $\lim_{\lambda \to 0}I_{\ell, \lambda}(\theta^*) = 1$ for almost every $\xi$. Then $\lim_{\lambda \to 0}\Gamma(\lambda) = 0$. If  $\lambda = o(1)$, the non-orthogonality condition does not hold asymptotically while $\Var_{\P}[M_{\lambda}(\xi)] \succeq 0$. This implies that a first-order improvement is not achievable when $\lambda = o(1)$.

$\hfill \square$

\subsection{Connections between Directionally Perturbed EO Solutions and GMM}\label{app:gmm}
Finally, we relate our directionally perturbed EO solution to the GMM estimator $\hat\theta_{{GMM}}$, which is defined as the solution to
\[
\hat\theta_{GMM} \in \arg\min_{\theta \in \Theta}\ \hat g(\theta)^\top C \hat g(\theta),
\]
where $\hat g(\theta)
=\begin{bmatrix}
\frac{1}{n}\sum_{i=1}^n \nabla_\theta \ell(\theta;\xi_i)\\[6pt]
\frac{1}{n}\sum_{i=1}^n G(\theta;\xi_i)
\end{bmatrix} \in \R^{D_\theta + D}$ and $C \in \R^{(D+D_{\theta})\times (D+D_{\theta})}$ is a fixed weighting matrix.

\begin{assumption}[Condition for GMM]\label{asp:condition-gmm}
Let
\[
g(\theta;\xi):=
\begin{bmatrix}
\nabla_\theta\ell(\theta;\xi)\\
G(\theta;\xi)
\end{bmatrix},
\qquad
\bar g(\theta):=\E_{\P}[g(\theta;\xi)].
\]
Suppose the following conditions hold:
\begin{enumerate}[(i)]
    \item $C=C^\top\succ0$ is fixed and deterministic, and
    $\bar g(\theta^*)=\mathbf 0$;
    \item for some neighborhood
    $\mathcal N$ of $\theta^*$, $\theta^*$ is the unique minimizer of
    $\bar g(\theta)^\top C\bar g(\theta)$ over $\mathcal N$;
    \item $g(\theta;\xi)$ is continuously differentiable in $\theta$ on
    $\mathcal N$ for $\P$-almost every $\xi$. There exist measurable
    envelopes $F$ and $L$ satisfying
    $\E_{\P}[F(\xi)^2]+\E_{\P}[L(\xi)]<\infty$ such that $\sup_{\theta\in\mathcal N}
    \big\{\|g(\theta;\xi)\|_2+
    \|\nabla_\theta g(\theta;\xi)\|\big\}
    \leq F(\xi)$
    and $\|\nabla_\theta g(\theta_1;\xi)
    -\nabla_\theta g(\theta_2;\xi)\|_{\mathrm{op}}
    \leq L(\xi)\|\theta_1-\theta_2\|_2$ for any $\theta_1,\theta_2\in\mathcal N$;
    \item Denote
    \[
    A:=\E_{\P}[\nabla_\theta^2\ell(\theta^*;\xi)],
    \qquad
    B:=\E_{\P}[\nabla_\theta G(\theta^*;\xi)],
    \qquad
    D_*:=\begin{bmatrix}A\\ B\end{bmatrix}.
    \]
    The matrix $D_*^\top C D_*$ is nonsingular.
\end{enumerate}
\end{assumption}
\begin{theorem}[Representation of GMM as EO+ Solution]\label{thm:gmm-unification}
Suppose Assumptions~\ref{asp:theta-star0},~\ref{asp:optimal-condition}, and~\ref{asp:condition-gmm} hold. Then, the GMM solution $\hat\theta_{GMM}$ can be expressed as an EO+ solution with the following representation:
\[
\hat\theta_{GMM}-\hat\theta_{EO}
= \frac{1}{n}\sum_{i=1}^n H^* G\Para{\hat\theta_{EO};\xi_i} + o_p\Para{n^{-1/2}},
\]
where $H^*
= -\Para{A^\top C_{11} A + A^\top C_{12} B + B^\top C_{21} A + B^\top C_{22} B}^{-1}
 \Para{A^\top C_{12} + B^\top C_{22}}$, with $A$ and $B$ evaluated at $\theta^*$ as defined in \Cref{asp:condition-gmm}.
\end{theorem}
The theorem is restricted to a smooth, correctly specified, locally identified
GMM problem with an interior solution and a fixed weighting matrix $C$. For any $C$ with finite entries, $\hat\theta_{GMM}$ achieves a first-order improvement over $\hat\theta_{EO}$. The optimal matrix $C$ can be constructed from $\hat\theta_{EO}$ and the optimal GMM estimator is a one-step estimator~\citep{imbens1997one}.  Intuitively, comparing \Cref{thm:gmm-unification} with \Cref{coro:eo-plus-improvement}, $\|H^*\| = \Theta(1)$ and $\E_{\P^*}[G(\theta^*;\xi)] = 0$, and it satisfies the conditions of first-order improvements there. For now, we only focus on the fixed $C$ but one can derive a similar perturbation representation for a data-dependent weighting matrix $C$. Expressing the solution in the unified perturbation representation under \eqref{eq:empirical-solution-adjustment} and \Cref{thm:unification} offers two main advantages over the standard GMM formulation. 
First, it facilitates analysis of constrained or non-smooth problems, where $\nabla\ell$ may arise from noncontinuous objectives. 
Second, the unified view in \Cref{thm:unification} is more flexible: it naturally handles moment conditions that do not depend on $\theta$, is robust to misspecification, and enables fast adaptation in streaming-data settings without repeated optimization given the empirical solution $\hat\theta_{EO}$.

\textit{Proof of \Cref{thm:gmm-unification}.} Denote 
\[
\hat g(\theta):=\frac1n\sum_{i=1}^n g(\theta;\xi_i),\qquad
\hat g_j(\theta):=\frac1n\sum_{i=1}^n g_j(\theta;\xi_i)\ \Para{j=1,2},
\]
where $g_1(\theta;\xi):=\nabla_\theta \ell(\theta;\xi)\in\R^{D_{\theta}}, g_2(\theta;\xi):=G(\theta;\xi)\in\R^{D}$. Partition $C$ into blocks:
\[
C=\begin{bmatrix} C_{11} & C_{12}\\ C_{21} & C_{22}\end{bmatrix},\qquad  C_{11} \in \R^{D_{\theta}\times D_{\theta}}, C_{12}\in\R^{D_{\theta}\times D}, C_{21}\in\R^{D \times D_{\theta}}, C_{22} \in \R^{D\times D}.
\]
Let $\hat D(\theta):=\nabla_\theta\hat g(\theta)$. The envelope and
Lipschitz conditions imply the required uniform law of large numbers and
the uniform Taylor expansion
\[
\hat g(\theta_1)-\hat g(\theta_2)
=D_*(\theta_1-\theta_2)
+o_p(\|\theta_1-\theta_2\|_2)
\]
for consistent $\theta_1,\theta_2\in\mathcal N$. They also imply
$\hat D(\hat\theta_{GMM})\overset{p}{\to}D_*$. Moreover,
$\hat g_2(\hat\theta_{EO})=O_p(n^{-1/2})$ by the central limit theorem,
the expansion around $\theta^*$, and
$\hat\theta_{EO}-\theta^*=O_p(n^{-1/2})$.

Because $\hat\theta_{GMM}$ is an interior solution, its first-order
condition is
\[
\hat D(\hat\theta_{GMM})^\top C
\hat g(\hat\theta_{GMM})=\mathbf 0.
\]
Linearizing $\hat g(\hat\theta_{GMM})$ at $\hat\theta_{EO}$, using
$\hat g_1(\hat\theta_{EO})=o_p(n^{-1/2})$, and using the nonsingularity of
$D_*^\top C D_*$ gives
\begin{align*}
\hat\theta_{GMM}-\hat\theta_{EO}
&=-(D_*^\top C D_*)^{-1}D_*^\top C
\begin{bmatrix}\mathbf 0\\ \hat g_2(\hat\theta_{EO})\end{bmatrix}
+o_p(n^{-1/2})\\
&=-(D_*^\top C D_*)^{-1}
\Para{A^\top C_{12}+B^\top C_{22}}
\hat g_2(\hat\theta_{EO})
+o_p(n^{-1/2})\\
&=\frac1n\sum_{i=1}^n H^*
G(\hat\theta_{EO};\xi_i)+o_p(n^{-1/2}).
\end{align*}
Expanding $D_*^\top C D_*$ by blocks gives the displayed formula for
$H^*$.

$\hfill \square$


\section{Proofs and Additional Details in Section~\ref{sec:perturb-opt}}\label{app:proof-perturb-opt}
\subsection{\textit{Proof of \Cref{thm:opt-adjust}}}
Define $L_n(H):=\E_{\Dscr_n}[R(\hat\theta_{H,M})]$ and $Q(H):=
\frac12
\E_{\P}\!\left[
\|\IFx(\xi)+H\,\widetilde M(\theta^*;\xi)\|_{I_\ell(\theta^*)}^2
\right]$. 
Because $\Hscr_B$ is compact, the proof of
\Cref{thm:improvement-principle}, together with the moment and remainder
conditions in Assumption~\ref{asp:single-side-information}, gives the
expected-regret expansion uniformly over $\Hscr_B$:
\begin{equation}\label{eq:uniform-H-risk-expansion}
\sup_{H\in\Hscr_B}
\left|
nL_n(H)-Q(H)
\right|
=o(1).
\end{equation}
Since $H_n$ is an approximate minimizer and $H^*\in\Hscr_B$, we have $nL_n(H_n)
\leq \inf_{H\in\Hscr_B}nL_n(H)+o(1)
\leq nL_n(H^*)+o(1)$.

Applying~\eqref{eq:uniform-H-risk-expansion} to both $H_n$ and $H^*$
yields
\begin{equation}\label{eq:Q-near-minimizer}
Q(H_n)\leq Q(H^*)+o(1).
\end{equation}

To identify the population minimizer, from \Cref{asp:single-side-information}, we have:
$\E_{\P}[\widetilde M(\theta^*;\xi)]=\mathbf 0, \Cov_{\P}[\widetilde M(\theta^*;\xi)]\succ0$. Completing the square in
$Q$ at the minimizer $H^*$ gives
\begin{align*}
Q(H)-Q(H^*)
& =\frac12\tr\!\left[
I_\ell(\theta^*)(H-H^*)\widetilde\Omega_M(H-H^*)^\top
\right]\\
& \gtrsim \frac12\lambda_{\min}\!\left(I_\ell(\theta^*)\right)
\lambda_{\min}\!\left(\widetilde\Omega_M\right)
\|H-H^*\|^2.
\end{align*}
Combining this inequality with~\eqref{eq:Q-near-minimizer} shows that
$\|H_n-H^*\|=o(1)$.

$\hfill \square$

\subsection{\textit{Proof of \Cref{prop:max-improvement}}}
First, $\hat\theta_{H^*, M}$ achieves the largest first-order improvement over the EO solution $\hat\theta_{EO}$ by \Cref{asp:single-side-information} and \Cref{thm:opt-adjust}. We then bound:
\begin{small}
\begin{align*}
    n\|\hat\theta_{\widehat H_B, M} - \hat\theta_{H^*, M}\|_2^2 & = n\bigg\|\Pi_{\Theta}\Para{\hat\theta_{EO} + \frac{\widehat H_B}{n}\sum_{i=1}^n M\Para{\hat\theta_{EO};\xi_i}} - \Pi_{\Theta}\Para{\hat\theta_{EO} + \frac{H^*}{{n}}\sum_{i=1}^n M\Para{\hat\theta_{EO};\xi_i}}\bigg\|_2^2\\
    & \leq \bigg\|\frac{(\widehat H_B - H^*)}{\sqrt{n}}\sum_{i = 1}^n M(\hat\theta_{EO};\xi_i) \bigg\|_2^2 \leq \|\widehat H_B - H^*\|^2\bigg\|\sum_{i = 1}^n M(\hat\theta_{EO};\xi_i)/\sqrt{n}\bigg\|_2^2
\end{align*}
\end{small}
where the first inequality above follows from the non-expansiveness property of the projection operator.  Taking expectations on both sides, we have:
\begin{small}
\begin{align*}
    \E\Paran{n\|\hat\theta_{\widehat H_B, M} - \hat\theta_{H^*, M}\|_2^2} \leq \Para{\E\Paran{\|\widehat H_B - H^*\|^6}}^{\frac{1}{3}} \Para{\E\bigg\|\frac{1}{\sqrt{n}} \sum_{i = 1}^n M(\hat\theta_{EO};\xi_i)\bigg\|_2^3}^{\frac{2}{3}}  = o(1) \Theta(1) = o(1),
\end{align*}
\end{small}
where the first inequality above follows from H\"older inequality. In the second inequality, $\E\Paran{\|\widehat H_B - H^*\|^6} = o(1)$ follows from the fact that $\|\widehat H - H^*\| = o_p(1)$ and both $\widehat H_B, H^*$ are bounded; and $\E\bigg\|\frac{1}{\sqrt{n}} \sum_{i = 1}^n M(\hat\theta_{EO};\xi_i)\bigg\|_2^3 = \Theta(1)$ follows from \Cref{asp:consistent-estimate} for a neighborhood $\Nscr$ of $\theta^*$ and $\E_{\P}\|\sum_{i =1}^n M(\theta^*;\xi_i)/\sqrt{n}\|_2^3 =  \Theta(1)$. 

Finally, a second-order expansion of $R(\theta)$ at $\hat\theta_{\widehat H_B, M}$ and $\hat\theta_{H^*, M}$ yields $\E[R(\hat\theta_{\widehat H_B, M})] - \E[R(\hat\theta_{H^*, M})] = o(1/n)$. We therefore conclude that $\hat\theta_{\widehat H_B, M}$ has the same first-order improvement as $\hat\theta_{H^*, M}$.

$\hfill \square$

\subsection{Hardness of Attaining $H\mu_M(\theta^*) = \mathbf{0}$ when $\mu_M(\theta^*) \neq \mathbf{0}$}\label{app:h-correct-impossible}
In \Cref{sec:perturb-opt}, we do not consider the data-driven estimation of the optimal adjustment matrix $H$ in the case where $\mu_M(\theta^*)\neq \mathbf{0}$ but $H\mu_M(\theta^*)=\mathbf{0}$. We refer to this case as $H$-correctness. While such an adjustment is population-correct under \Cref{thm:improvement-principle},
preserving this correctness after estimating $H$ is generally difficult:
\begin{proposition}[Hardness of Data-Driven $H$-Correctness]
\label{lemma:h-correct-nonadaptive}
Let $H^*$ satisfy $H^*\mu_M(\theta^*)=\mathbf{0}$, and suppose
$\widehat H-H^*=o_p(1)$. Then $\hat\theta_{\widehat H, M}$ and $\hat\theta_{H^*, M}$ achieve the same amount of first-order improvement only if $(\widehat H-H^*)\mu_M(\theta^*) = o_p(n^{-1/2})$.

Moreover, suppose $H=H(\eta)$ for a finite-dimensional parameter
$\eta\in\R^q$, where $H(\eta)$ is differentiable at $\eta^*$ and
$H^*=H(\eta^*)$. If the matrix $G
    :=
    \left.
    \nabla_{\eta}
    \bigl\{H(\eta)\mu_M(\theta^*)\bigr\}
    \right|_{\eta=\eta^*}$ has full column rank, then $(\widehat H-H^*)\mu_M(\theta^*)
    =
    o_p(n^{-1/2})$ implies $\|\widehat\eta-\eta^*\|
    =
    o_p(n^{-1/2}).$
\end{proposition}
Above, if $\widehat\eta$ is a regular estimator with a nondegenerate
$\sqrt n$ limiting distribution, $\hat\theta_{\widehat H, M}$ and $\hat\theta_{H^*, M}$ cannot achieve the same amount of first-order improvement. Therefore, regular estimators typically cannot attain this rate without assuming additional structural information.

\textit{Proof of \Cref{lemma:h-correct-nonadaptive}.}
For the first claim, recall the expansion
\begin{align*}
    \hat\theta_{\widehat H,M}
    -
    \hat\theta_{H^*,M}
    &=
    (\widehat H-H^*)\sum_{i = 1}^n M(\hat\theta_{EO};\xi_i)/n
    +o_p(n^{-1/2}) \\
    &=
    (\widehat H-H^*)\mu_M(\theta^*)
    +
    (\widehat H-H^*)
    \Para{
        \sum_{i = 1}^n M(\hat\theta_{EO};\xi_i)/n
        -
        \mu_M(\theta^*)}
    +o_p(n^{-1/2}).
\end{align*}
If $\hat\theta_{\widehat H, M}$ and $\hat\theta_{H^*, M}$ achieve the same amount of first-order improvement, then $\|\hat\theta_{\widehat H, M} - \hat\theta_{H^*, M}\| = o_p(n^{-1/2})$.
Since $\widehat H-H^*=o_p(1)$ and
$\sum_{i = 1}^n M(\hat\theta_{EO};\xi_i)/n
        -
        \mu_M(\theta^*)=O_p(n^{-1/2})$ by the preceding arguments in the proof of \Cref{thm:improvement-principle}, the second
term on the right-hand side above is $o_p(n^{-1/2})$. This implies that the first term $(\widehat H-H^*)\mu_M(\theta^*) $ is of the order $o_p(n^{-1/2})$.

For the second claim, we have:
\[
    \bigl(H(\widehat\eta)-H(\eta^*)\bigr)\mu_M(\theta^*)
    =
    G(\widehat\eta-\eta^*)
    +
    o_p\bigl(\|\widehat\eta-\eta^*\|\bigr).
\]
Because $G$ has full column rank, its smallest singular value is strictly
positive. Thus, in a neighborhood of $\eta^*$, $\left\|
        \bigl(H(\widehat\eta)-H(\eta^*)\bigr)\mu_M(\theta^*)
    \right\|
    \gtrsim
    \|\widehat\eta-\eta^*\|$, which yields the desired property.

$\hfill \square$

\subsection{\textit{Proof of \Cref{prop:analytical-consistency}}}
\begin{lemma}[Plug-in Consistency]\label{asp:plugin}
Suppose $\|\hat I_{\ell}(\hat\theta_{EO})-I_{\ell}(\theta^*)\|_{\mathrm{op}}=o_p(1)$. 
Then:
\[\|\widehat{\IFx}(\xi)-\IFx(\xi)\|_{2}=o_p(1), \forall \xi~\text{a.e.}\]
\end{lemma}

\textit{Proof of \Cref{asp:plugin}.} 
We have
\[
\widehat{\IFx}(\xi) - \IFx(\xi)
=
-[\hat I_{\ell}(\hat\theta_{EO})]^{\dagger}\nabla_{\theta} \ell(\hat\theta_{EO};\xi)
+
[I_{\ell}(\theta^*)]^{-1} \nabla_{\theta}\ell(\theta^*;\xi).
\]
Adding and subtracting $[I_{\ell}(\theta^*)]^{-1}\nabla_{\theta}\ell(\hat\theta_{EO};\xi)$ and applying the triangle inequality gives
\begin{small}
\[
\begin{aligned}
\|\widehat{\IFx}(\xi) - \IFx(\xi)\|_2
&\le
\|([\hat I_{\ell}(\hat\theta_{EO})]^{\dagger} - [I_{\ell}(\theta^*)]^{-1}) \nabla_{\theta} \ell(\hat\theta_{EO};\xi)\|_2 +
\|[I_{\ell}(\theta^*)]^{-1}\big(\nabla_{\theta} \ell(\hat\theta_{EO};\xi) - \nabla_{\theta}\ell(\theta^*;\xi)\big)\|_2 \\
&\le
\|[\hat I_{\ell}(\hat\theta_{EO})]^{\dagger} - [I_{\ell}(\theta^*)]^{-1}\|_{\mathrm{op}}
\cdot
\|\nabla_{\theta} \ell(\hat\theta_{EO};\xi)\|_2 \\
&\quad +
\|[I_{\ell}(\theta^*)]^{-1}\|_{\mathrm{op}}
\cdot
\|\nabla_{\theta} \ell(\hat\theta_{EO};\xi) - \nabla_{\theta}\ell(\theta^*;\xi)\|_2.
\end{aligned}
\]
\end{small}

By the assumption $\|\hat I_{\ell}(\hat\theta_{EO}) - I_{\ell}(\theta^*)\|_{\mathrm{op}} = o_p(1)$ and nonsingularity of $I_{\ell}(\theta^*)$, the matrix inverse is continuous, so
\[
\|[\hat I_{\ell}(\hat\theta_{EO})]^{\dagger} - [I_{\ell}(\theta^*)]^{-1}\|_{\mathrm{op}} = o_p(1).
\]
Moreover, if $\hat\theta_{EO} \overset{p}{\to} \theta^*$ and $\nabla_{\theta}\ell(\theta;\xi)$ is continuous at $\theta^*$ for almost every $\xi$, we have:
\[
\|\nabla_{\theta} \ell(\hat\theta_{EO};\xi) - \nabla_{\theta}\ell(\theta^*;\xi)\|_2 = o_p(1),
\qquad \forall \xi \text{ a.e.}
\]
Also, $\|\nabla_{\theta} \ell(\hat\theta_{EO};\xi)\|_2 = O_p(1)$ for almost every $\xi$. Combining the above bounds yields
\[
\|\widehat{\IFx}(\xi) - \IFx(\xi)\|_2 = o_p(1),
\qquad \forall \xi \text{ a.e.}
\]

$\hfill \square$

We now return to the proof of \Cref{prop:analytical-consistency}.

We first consider the unconstrained case $\Theta = \R^{D_{\theta}}$. Here, both $\widehat H$ and $H^*$ admit closed-form expressions as follows:
\begin{align*}
\widehat H &= -\E_{\hat\P_n}[\widehat \IFx(\xi) \widehat M(\hat\theta_{EO};\xi)^{\top}](\E_{\hat\P_n}[\widehat M(\hat\theta_{EO};\xi) \widehat M(\hat\theta_{EO};\xi)^{\top}])^{\dagger} \\
H^* &= -\E_{\P}[\IFx(\xi) \widetilde M(\theta^*;\xi)^{\top}](\E_{\P}[\widetilde M(\theta^*;\xi) \widetilde M(\theta^*;\xi)^{\top}])^{-1},
\end{align*}
where $\widehat M(\theta;\xi) = M(\theta;\xi) + \nabla_{\theta} \widebar M \, \widehat\IFx(\xi)$.

By \Cref{asp:plugin}, the consistency $\hat\theta_{EO} \overset{p}{\to} \theta^*$, and the continuity of $M(\theta;\xi)$, we have, for any $\xi$:
\begin{small}
\[\|\widetilde M(\theta^*;\xi) - \widehat M(\hat\theta_{EO};\xi)\|_{1} \lesssim \max\{\|\hat\theta_{EO} - \theta^*\|, \|\nabla_{\theta}\mu_M(\theta^*) - \nabla_{\theta}\frac{1}{n}\sum_{i = 1}^n M(\theta^*;\xi_i)\|, \|\IFx(\xi) - \widehat{\IFx}(\xi)\|\} = o_p(1).\]
\end{small}
We establish the consistency of the moment term and cross-moment term appearing in $\widehat H$ and $H^*$.

\noindent For the moment term:
\begin{small}
\[
\begin{aligned}
&\left\|\E_{\P}[\widetilde M(\theta^*;\xi)\widetilde M(\theta^*;\xi)^{\top}] - \E_{\hat\P_n}[\widehat M(\hat\theta_{EO};\xi)\widehat M(\hat\theta_{EO};\xi)^{\top}]\right\| \\
&\le 
\left\|\E_{\hat\P_n}\!\left[\widehat M(\hat\theta_{EO};\xi)\widehat M(\hat\theta_{EO};\xi)^{\top} - \widetilde M(\theta^*;\xi)\widetilde M(\theta^*;\xi)^{\top}\right]\right\| \\
&\quad +
\left\|\E_{\hat\P_n}[\widetilde M(\theta^*;\xi)\widetilde M(\theta^*;\xi)^{\top}] - \E_{\P}[\widetilde M(\theta^*;\xi)\widetilde M(\theta^*;\xi)^{\top}]\right\|
= o_p(1).
\end{aligned}
\]
\end{small}
For the cross-moment term:
\begin{small}
\[
\begin{aligned}
&\left\|\E_{\P}[\IFx(\xi)\widetilde M(\theta^*;\xi)^{\top}] - \E_{\hat\P_n}[\widehat \IFx(\xi)\widehat M(\hat\theta_{EO};\xi)^{\top}]\right\| \\
&\le 
\left\|\E_{\hat\P_n}\!\left[\widehat \IFx(\xi)\widehat M(\hat\theta_{EO};\xi)^{\top} - \IFx(\xi)\widetilde M(\theta^*;\xi)^{\top}\right]\right\| \\
&\quad +
\left\|\E_{\hat\P_n}[\IFx(\xi)\widetilde M(\theta^*;\xi)^{\top}] - \E_{\P}[\IFx(\xi)\widetilde M(\theta^*;\xi)^{\top}]\right\| \\
&\le 
\E_{\hat\P_n}\!\left[\|\widehat \IFx(\xi) - \IFx(\xi)\|_2 \, \|\widehat M(\hat\theta_{EO};\xi)\|_2\right] \\
&\quad +
\E_{\hat\P_n}\!\left[\|\IFx(\xi)\|_2 \, \|\widehat M(\hat\theta_{EO};\xi) - \widetilde M(\theta^*;\xi)\|_2\right] \\
&\quad +
\left\|\E_{\hat\P_n}[\IFx(\xi)\widetilde M(\theta^*;\xi)^{\top}] - \E_{\P}[\IFx(\xi)\widetilde M(\theta^*;\xi)^{\top}]\right\|
= o_p(1).
\end{aligned}
\]
\end{small}

Therefore, we have $\|\widehat H - H^*\| = o_p(1)$ when $\Theta = \R^{D_{\theta}}$.

For a constrained $\Theta$, the same argument applies provided that the unconstrained empirical minimizer
\[
\widetilde H
=
-\E_{\hat\P_n}[\widehat \IFx(\xi) \widehat M(\hat\theta_{EO};\xi)^{\top}]
\Big(\E_{\hat\P_n}[\widehat M(\hat\theta_{EO};\xi) \widehat M(\hat\theta_{EO};\xi)^{\top}]\Big)^{\dagger}
\]
is asymptotically feasible, in the sense that $\P\!\left(
\hat\theta_{EO} + \frac{\widetilde H}{n}\sum_{i=1}^n M(\hat\theta_{EO};\xi_i)\in\Theta
\right)\to 1.$
Indeed, on the event $\hat\theta_{EO} + \frac{\widetilde H}{n}\sum_{i=1}^n M(\hat\theta_{EO};\xi_i)\in\Theta$, $\widetilde H$ is feasible for the constrained empirical optimization problem. Since $\widetilde H$ is also the unique minimizer of the same quadratic objective without the constraint, it must coincide with the constrained minimizer $\widehat H$. Therefore, it suffices to show
\[
\P\!\left(
\hat\theta_{EO} + \frac{\widetilde H}{n}\sum_{i=1}^n M(\hat\theta_{EO};\xi_i)\notin\Theta
\right)\to 0.
\]

To verify this, note that $\frac{1}{n}\sum_{i=1}^n M(\hat\theta_{EO};\xi_i)=O_p(n^{-1/2})$, and $\widetilde H = O_p(1)$ by the same argument as in the unconstrained case, since $\widetilde H \overset{p}{\to} H^*$. Therefore, $\frac{\widetilde H}{n}\sum_{i=1}^n M(\hat\theta_{EO};\xi_i)=O_p(n^{-1/2}).$

Thus the correction term is asymptotically negligible. Assumption~\ref{asp:theta-star0} gives $\theta^*\in\operatorname{int}(\Theta)$, and $\hat\theta_{EO}\convp\theta^*$. Because the additional correction term $ \frac{\widetilde H}{n}\sum_{i=1}^n M(\hat\theta_{EO};\xi_i) = O_p(n^{-1/2})$, the corrected solution remains in $\Theta$ with probability tending to one:
\[
\P\!\left(
\hat\theta_{EO} + \frac{\widetilde H}{n}\sum_{i=1}^n M(\hat\theta_{EO};\xi_i)\in\Theta 
\right)\to 1.
\]
Consequently, it follows that $\P\Para{\widehat H = \widetilde H} \to 1$, and the same consistency conclusion follows.

$\hfill \square$

\subsection{\textit{Proof of \Cref{prop:bootstrap-consistency}}}
Recall the definitions $\widetilde M(\theta;\xi) = M(\theta;\xi) + \nabla_{\theta} \mu_M(\theta)\IFx(\xi)$, and $\widehat M(\theta;\xi) = M(\theta;\xi) + \nabla_{\theta} \widebar M\widehat\IFx(\xi)$ in the proof of \Cref{prop:analytical-consistency}.

Since each bootstrap pair $(\hat\theta_{EO}^{(b)},\widebar M^{(b)})_{b \in [B]}$ is generated from an i.i.d. bootstrap resample, these pairs are i.i.d. conditional on the data. Recall the joint asymptotic linear representation of $\hat\theta_{EO}$ and $\widebar M$:
\[
\sqrt n
\begin{pmatrix}
\hat\theta_{EO}-\theta^*\\
\widebar M-\mu_M(\theta^*)
\end{pmatrix}
\Rightarrow
\mathcal N\!\left(
0,
\begin{bmatrix}
\Sigma_0
&
\E_{\P}\!\left[
\IFx(\xi)\widetilde M(\theta^*;\xi)^\top
\right]
\\
\E_{\P}\!\left[
\IFx(\xi)\widetilde M(\theta^*;\xi)^\top
\right]^\top
&
\E_{\P}\!\left[
\widetilde M(\theta^*;\xi)
\widetilde M(\theta^*;\xi)^\top
\right]
\end{bmatrix}
\right).
\]
Furthermore, since Assumptions~\ref{asp:theta-star0},~\ref{asp:optimal-condition} and~\ref{asp:single-side-information} hold, we have the following bootstrap consistency result:
\[
\sqrt n(\hat\theta_{EO}^{(b)}-\hat\theta_{EO},\, \widebar M^{(b)}-\widebar M)
\Rightarrow
\mathcal N\Para{0,
\begin{bmatrix}
\Sigma_{0} & \E_{\P}[\IFx(\xi) \widetilde M(\theta^*;\xi)^{\top}]\\
\E_{\P}[\IFx(\xi) \widetilde M(\theta^*;\xi)^{\top}]^\top & \E_{\P}[\widetilde M(\theta^*;\xi) \widetilde M(\theta^*;\xi)^{\top}]
\end{bmatrix}},
\]
conditionally on the data $\Dscr_n$. Then the conditional covariance matrices satisfy
\[
n\Cov^*(\hat\theta_{EO}^{(b)},\widebar M^{(b)}) \overset{p}{\to} \E_{\P}[\IFx(\xi) \widetilde M(\theta^*;\xi)^{\top}], \quad 
n\Var^*(\widebar M^{(b)}) \overset{p}{\to} \E_{\P}[\widetilde M(\theta^*;\xi) \widetilde M(\theta^*;\xi)^{\top}],
\]
where $\Cov^*(\cdot,\cdot)$ and $\Var^*(\cdot)$ are the covariance and variance under the bootstrap distribution conditional on the original data $\Dscr_n$. Therefore,
\[
-\Cov^*(\hat\theta_{EO}^{(b)},\widebar M^{(b)})\Var^*(\widebar M^{(b)})^{\dagger}
\overset{p}{\to} H^*.
\]
Then we control the Monte Carlo error from using only $B$ bootstrap replicates. Conditional on $\Dscr_n$, the estimators $\widehat{\Cov}_B(\hat\theta_{EO}^{(b)},\widebar M^{(b)})$ and 
$\widehat{\Var}_B(\widebar M^{(b)})$ are empirical covariance matrices based on $B$ i.i.d.\ samples. Therefore, the conditional law of large numbers yields
\begin{small}
\[
\widehat{\Cov}_B(\hat\theta_{EO}^{(b)},\widebar M^{(b)})
-
\Cov^*(\hat\theta_{EO}^{(b)},\widebar M^{(b)})
=O_p\Para{\frac{1}{n\sqrt{B}}},\quad \widehat{\Var}_B(\widebar M^{(b)})
-
\Var^*(\widebar M^{(b)})
=O_p\Para{\frac{1}{n\sqrt{B}}}.
\]
\end{small}
For $B=\omega(1)$, the two differences above become $o_p(1/n)$. Then by continuity of matrix inversion,
\[
\widehat H
=
-\widehat{\Cov}_B(\hat\theta_{EO}^{(b)},\widebar M^{(b)})
\widehat{\Var}_B(\widebar M^{(b)})^{\dagger}
\overset{p}{\to}
H^*.
\]
Thus $\widehat H$ is a consistent estimator of $H^*$, and Assumption~\ref{asp:consistent-estimate} is satisfied.

$\hfill \square$

\section{Proofs and Additional Details in Section~\ref{sec:cso}}\label{app:cso-same}
\subsection{Parametric Contextual Stochastic Optimization Methods}\label{app:cso-parametric}
We first describe common parametric contextual stochastic optimization
methods and then state conditions under which a given method reduces to the non-contextual analysis.
\begin{definition}[Parametric Contextual EO Solutions]\label{defn:cso-param-method}
    Consider the following parametric contextual optimization solutions, where the decision rule $\pi_z(u)$ is parametrized by $z$ and comes from the decision class $\{\pi_z(\cdot): \Uscr \mapsto \Theta \mid z \in \Zscr\}$. The resulting contextual EO solution is $\hat\theta_{EO}(u) = \pi_{\hat z_{EO}}(u)$.
    \begin{itemize}
        \item \underline{Decision-rule Optimization}: $\pi_{z}(u)$ is any parametrized family of decision functions, and $\hat z_{EO} \in \argmin_{z} \sum_{i = 1}^n \ell(\pi_{z}(u_i);\xi_i)$.
        \item The conditional distribution of $\xi$ given the covariate $u$ is parametrized within the family $\mathcal{\P}_{\Zscr}:=\{\P_{\xi|z,u}| z \in \Zscr\}$,  and $\pi_{z}(u) \in\argmin_{\theta} \E_{\P_{\xi|z,u}}[\ell(\theta;\xi)]$. 
        \begin{enumerate}
            \item \underline{Estimate-then-Optimize}: $\hat z$ is estimated using a statistical approach that only uses the data $\Dscr_n$, e.g., in the maximum likelihood estimation, we have $\hat z_{EO} \in \argmax_{z \in \Zscr}\sum_{i = 1}^n \ln p_z(\xi_i|u_i)$.
            \item \underline{Integrated-Estimation-and-Optimization}: $\hat z_{EO} \in \argmin_{z \in \Zscr}\sum_{i =1}^n \ell(\pi_z(u_i);\xi_i)$.
        \end{enumerate}
    \end{itemize}
\end{definition}

\begin{proposition}[Conditional Reduction to the Non-contextual Setting]
\label{prop:cso-param-reduction}
Let $W:=(u,\xi)$, let $\Zscr\subseteq\R^{D_z}$ be open, and define $L(z;W):=\ell(\pi_z(u);\xi), Q(z):=\E_{\P}[L(z;W)]$.

Consider one of the parametric procedures in
\Cref{defn:cso-param-method}, with estimator $\hat z_{EO}$ and target
$z^*\in\Zscr$. Suppose that:
\begin{enumerate}[(i)]
    \item $z\mapsto\pi_z(u)$ is continuously differentiable near $z^*$
    for almost every $u$, the composed loss $z\mapsto L(z;W)$ is twice continuously differentiable near $z^*$ almost everywhere;
    \item $z^*$ is the unique local minimizer of $Q$, lies in the interior
    of $\Zscr$, and the composed population Hessian satisfies
    $\nabla_z^2Q(z^*)\succ0$;
    \item $\hat z_{EO}\overset{p}{\to}z^*$ and, for some measurable
    $\varphi_z(W)$ satisfying
    $\E_{\P}[\varphi_z(W)]=\mathbf 0$ and
    $\E_{\P}[\|\varphi_z(W)\|_2^2]<\infty$,
    \[
    \hat z_{EO}-z^*
    =\frac1n\sum_{i=1}^n\varphi_z(W_i)+o_p(n^{-1/2});
    \]
    \item the empirical-process, moment, Taylor-remainder, and
    side-information assumptions required by the particular result being
    invoked hold for $L(z;W)$, $\varphi_z(W)$, and $M(z;W)$.
\end{enumerate}
Then each result in Sections~\ref{sec:perturb-stat}
and~\ref{sec:perturb-opt} whose remaining assumptions are satisfied applies
to the reduced finite-dimensional problem under the substitutions
\[
\theta\leftarrow z,\qquad
\xi\leftarrow W,\qquad
Z(\theta)\leftarrow Q(z),\qquad
\IFx(\xi)\leftarrow\varphi_z(W).
\]
In particular, the associated excess risk is
$Q(z)-Q(z^*)$, and correct side information satisfies
$\E_{\P}[M(z^*;W)]=\mathbf 0$.
\end{proposition}
For decision-rule optimization and integrated estimation-and-optimization,
$L(z;W)$ is also the empirical optimization criterion, so standard smooth
$M$-estimation conditions may yield the influence expansion in part~(iii).
For estimate-then-optimize methods, $\hat z_{EO}$ is generally obtained from
a different estimating criterion, such as a likelihood. In that case,
\Cref{prop:cso-param-reduction} requires separately verifying that its own likelihood score yields the stated influence function $\varphi_z$. The likelihood score generally differs from the gradient of the decision loss.

\subsection{\textit{Proof of \Cref{prop:contextual-improvement}}} 

In addition to the notation introduced in \Cref{prop:contextual-improvement}, fix $u_0$ and write $\widebar M_n(u_0):=\frac{1}{n}\sum_{i=1}^n M_{n,u_0}(u_i,\xi_i), I_{\ell,u_0}:=\nabla_\theta^2Z_{u_0}(\theta^*(u_0)), \Gamma_n(u_0):=n^{1/2-\gamma}\Cov_{\P}\!\left[
\IFx_{u_0}(u,\xi),M_{n,u_0}(u,\xi)
\right]$. Thus $\lim_{n \to \infty}\Gamma_n(u_0)= \Gamma_{u_0}$. In the calculations below, $H_n$
denotes an adjustment sequence with $\|H_n\| = O(n^{1-2\gamma})$.

For any random $\tilde\theta(u_0)$ close to $\theta^*(u_0)$, we take a second-order Taylor expansion of $Z_{u_0}$ at $\theta^*(u_0)$:
\begin{align*}
R(\tilde\theta(u_0))& =Z_{u_0}(\tilde\theta(u_0))-Z_{u_0}(\theta^*(u_0)) \\   
& = \frac{1}{2}(\tilde\theta(u_0)-\theta^*(u_0))^\top I_{\ell,u_0}(\tilde\theta(u_0)-\theta^*(u_0)) + o_p\big(\|\tilde\theta(u_0)-\theta^*(u_0)\|^2\big).
\end{align*}
Taking expectations over the dataset $\Dscr_n$ and applying the bias-variance decomposition,
\begin{align*}
\E[R(\tilde\theta(u_0))] = & \frac{1}{2}\text{Tr}(I_{\ell,u_0}\Var(\tilde\theta(u_0))) \\
+ &\frac{1}{2}(\E[\tilde\theta(u_0)]-\theta^*(u_0))^\top I_{\ell,u_0} (\E[\tilde\theta(u_0)]-\theta^*(u_0)) + o(\E\|\tilde\theta(u_0)-\theta^*(u_0)\|^2).
\end{align*}

We compute the difference of the expected excess risk between $\tilde\theta(u_0)=\hat\theta_{EO}(u_0)$ and $\hat\theta_{H_n,M_{n,u_0}}(u_0)$.

For $\hat\theta_{EO}(u_0)$, from the influence function decomposition in \Cref{asp:if-decomposition} with $\gamma \leq \frac{1}{2}$,
\[
\Var(\hat\theta_{EO}(u_0))=\frac{1}{n^{2\gamma}}\Sigma_{u_0} + o\Big(n^{-2\gamma}\Big), 
\qquad 
\E[\hat\theta_{EO}(u_0)]-\theta^*(u_0)= n^{-\gamma}b_{u_0} + o(n^{-\gamma}),
\]
where $\Sigma_{u_0} = \E_{\P}[\IFx_{u_0}(u,\xi) \IFx_{u_0}(u, \xi)^{\top}] = \Theta(1)$. Therefore:
\[
\E[R(\hat\theta_{EO}(u_0))]
= \frac{1}{2n^{2\gamma}}\text{Tr}(I_{\ell,u_0} \Sigma_{u_0}) + \frac{1}{2n^{2\gamma}}b_{u_0}^\top I_{\ell,u_0} b_{u_0} + o(n^{-2\gamma}).
\]

For $\hat\theta_{H_n,M_{n,u_0}}(u_0)$, consider the moments of \(\hat\theta_{H_n,M_{n,u_0}}(u_0)=\hat\theta_{EO}(u_0)+H_n\widebar M_n(u_0)\). From \Cref{asp:if-decomposition} and the definition of $\Gamma_n(u_0)$, we have:
\begin{align*}
    \text{Var}(\hat\theta_{H_n,M_{n,u_0}}(u_0))
& = \text{Var}(\hat\theta_{EO}(u_0))
+ \text{Cov}(\hat\theta_{EO}(u_0),\widebar M_n(u_0))H_n^\top\\
&\quad + H_n\text{Cov}(\widebar M_n(u_0),\hat\theta_{EO}(u_0))
+ H_n \text{Var}(\widebar M_n(u_0)) H_n^\top\\
& = \frac{1}{n^{2\gamma}}\Sigma_{u_0}
+\frac{1}{n}\Big(H_n\Gamma_n(u_0)^{\top} + \Gamma_n(u_0) H_n^{\top}
+ H_n \Omega_n(u_0) H_n^\top\Big)
+ o\Big(n^{-2\gamma}\Big).
\end{align*}
Moreover, we have $\E[\hat\theta_{H_n,M_{n,u_0}}(u_0)]-\theta^*(u_0)
= (\E[\hat\theta_{EO}(u_0)]-\theta^*(u_0)) + H_n\E[\widebar M_n(u_0)] = n^{-\gamma}b_{u_0} + H_n\mu_n(u_0) + o(n^{-\gamma})$, so:
\begin{small}
\begin{align*}
  \E[R(\hat\theta_{H_n,M_{n,u_0}}(u_0))] 
  & = \frac{1}{2n^{2\gamma}}\text{Tr}(I_{\ell,u_0} \Sigma_{u_0})
  + \frac{1}{2n^{2\gamma}}b_{u_0}^\top I_{\ell,u_0} b_{u_0}\\
  &\quad + \frac{1}{2n}\text{Tr}\!\Big(I_{\ell,u_0}(H_n\Gamma_n(u_0)^{\top}
  + \Gamma_n(u_0) H_n^{\top} + H_n \Omega_n(u_0) H_n^\top)\Big)\\
  &\quad + \frac{1}{2}(H_n\mu_n(u_0))^\top I_{\ell,u_0} (H_n\mu_n(u_0))
  + n^{-\gamma}(H_n\mu_n(u_0))^\top I_{\ell,u_0} b_{u_0}
  +o(n^{-2\gamma}).
\end{align*}
\end{small}

Subtracting the expansions of $\E[R(\hat\theta_{H_n,M_{n,u_0}}(u_0))]$ and $\E[R(\hat\theta_{EO}(u_0))]$, we have 
\begin{equation}\label{eq:context-diff-expand}
\begin{aligned}
&~~~~\E[R(\hat\theta_{H_n,M_{n,u_0}}(u_0))]-\E[R(\hat\theta_{EO}(u_0))]\\
&= \frac{1}{2n}\text{Tr}\!\Big(I_{\ell,u_0}(H_n\Gamma_n(u_0)^{\top}
+ \Gamma_n(u_0) H_n^{\top} + H_n \Omega_n(u_0) H_n^\top)\Big)\\
&\qquad + \frac{1}{2}(H_n\mu_n(u_0))^\top I_{\ell,u_0} (H_n\mu_n(u_0))\\
&\qquad + n^{-\gamma}(H_n\mu_n(u_0))^\top I_{\ell,u_0} b_{u_0}+ o(n^{-2\gamma}).
\end{aligned}
\end{equation}

Recall that $\E[R(\hat\theta_{EO}(u_0))] = \Theta(n^{-2\gamma})$. We consider the following three cases:

\medskip
\noindent\textbf{Case (i):} $\mu_n(u_0)=o(n^{\gamma-1})$ and $\Gamma_{u_0}\neq\mathbf0$. Since $\|H_n\|=O(n^{1-2\gamma})$, we have $H_n\mu_n(u_0)=o(n^{-\gamma})$. Therefore,
\begin{align*}
&\E[R(\hat\theta_{H_n,M_{n,u_0}}(u_0))]-\E[R(\hat\theta_{EO}(u_0))]\\
&\qquad= \frac{1}{2n}\text{Tr}\!\Big(I_{\ell,u_0}(H_n\Gamma_n(u_0)^{\top}
+ \Gamma_n(u_0) H_n^{\top} + H_n \Omega_n(u_0) H_n^\top)\Big)
+ o\Big(n^{-2\gamma}\Big).
\end{align*}
For any sufficiently small fixed $t\in(0,1)$, set
$H_n(t)=-t\Gamma_n(u_0)\Omega_n(u_0)^{\dagger}$. The convergence of
$\Gamma_n(u_0)$ and $n^{1-2\gamma}\Omega_n(u_0)$ implies
$\|H_n(t)\|=\Theta(n^{1-2\gamma})$ and
\begin{align*}
&\E[R(\hat\theta_{H_n(t),M_{n,u_0}}(u_0))]
-\E[R(\hat\theta_{EO}(u_0))]\\
&\qquad=-\frac{2t-t^2}{2n}\operatorname{Tr}\!\left(
I_{\ell,u_0}\Gamma_n(u_0)\Omega_n(u_0)^\dagger
\Gamma_n(u_0)^\top\right)+o(n^{-2\gamma})
=-\Theta(n^{-2\gamma})<0.
\end{align*}
Taking $t$ sufficiently small also keeps the leading risk coefficient positive,
so the adjusted solution has the same rate $n^{-2\gamma}$ and a strictly
smaller leading coefficient than weighted EO.

\medskip
\noindent\textbf{Case (ii):} $\mu_n(u_0)=o(n^{\gamma-1})$ and $\Gamma_{u_0}=\mathbf0$.
Then 
\[
\E[R(\hat\theta_{H_n,M_{n,u_0}}(u_0))]-\E[R(\hat\theta_{EO}(u_0))]
= \frac{1}{2n}\text{Tr}\!\Big(I_{\ell,u_0} H_n \Omega_n(u_0) H_n^\top\Big) + o(n^{-2\gamma}),
\]
where the displayed leading term is nonnegative, so the best leading-order choice is $H_n=0$.

\medskip
\noindent\textbf{Case (iii):} $\mu_n(u_0)=\Theta(n^{\gamma-1})$. In this boundary regime, $H_n\mu_n(u_0)$ can be of order $n^{-\gamma}$, so all three terms in~\eqref{eq:context-diff-expand} can be relevant. Define
\[
Q_n(u_0):=\frac{1}{n}\Omega_n(u_0)+\mu_n(u_0)\mu_n(u_0)^\top,
\qquad
C_n(u_0):=\frac{1}{n}\Gamma_n(u_0)+n^{-\gamma}b_{u_0}\mu_n(u_0)^\top.
\]
Then the leading excess-risk difference can be written as
\[
\frac{1}{2}\operatorname{Tr}\!\left(I_{\ell,u_0}H_n Q_n(u_0)H_n^\top\right)
+\operatorname{Tr}\!\left(I_{\ell,u_0}H_n C_n(u_0)^\top\right)
+o(n^{-2\gamma}).
\]
Suppose \(I_{\ell,u_0}\) is positive definite. The
minimum-norm optimizer of this quadratic expression is $H_n^*=-C_n(u_0)Q_n(u_0)^\dagger$, and the corresponding leading excess-risk difference is
\[
-\frac{1}{2}\operatorname{Tr}\!\left(
I_{\ell,u_0}C_n(u_0)Q_n(u_0)^\dagger C_n(u_0)^\top\right)+o(n^{-2\gamma}).
\]
Above, the order of the performance improvement depends on
\(\Gamma_n(u_0)\), \(b_{u_0}\), and their scaling with \(n\), so it is not determined without further assumptions. We do not discuss the details here.

$\hfill \square$

\subsection{Maximizing First-Order Improvements in Contextual Stochastic Optimization}\label{app:perturb-opt-cso}
\paragraph{Mild Conditions Leading to Correct Side Information.} Suppose the conditional mean $\E_{\P}[\xi|u]$ is Lipschitz continuous and the marginal density of $u$ is bounded in a neighborhood of $u_0$, so that, for some $L > 0$,
\[\|\E_{\P}[\xi|u_1] - \E_{\P}[\xi|u_2]\|\leq L\|u_1 - u_2\|_2.\]
If $h_n = \Theta(n^{-\beta})$ with $\beta \in (\frac{1}{D_u + 2}, \frac{1}{D_u})$, then $\mu_n(u_0)=o(n^{\gamma-1})$ for the perturbation function in~\eqref{eq:side-info-cso}. More specifically,
\begin{align*}
\|\mu_n(u_0)\|_2
&=
\left\|
\E_{\P}\!\left[
\left(\E_{\P}[\xi\mid u]-\E_{\P}[\xi\mid u_0]\right)
\mathbf{1}_{\{\|u-u_0\|_2\leq h_n\}}
\right]
\right\|_2\\
&\leq
L h_n\,\P\!\left(\|u-u_0\|_2\leq h_n\right)
 =O(h_n^{D_u+1})
=o(n^{\gamma-1}),
\end{align*}
where the first equality follows from iterated expectations, while the bound $O(h_n^{D_u+1})$ uses $\P(\|u-u_0\|_2\leq h_n)=O(h_n^{D_u})$. Moreover, under the stated choice of $\beta$, $\gamma=(1-\beta D_u)/2>0$ and $\beta(D_u+1)>1-\gamma$, which gives the last equality.
\paragraph{Side Information.} For the following weight specifications, at each new covariate $u_0$, we define the perturbation function as follows:
\[M_{n,u_0}(u_i, \xi_i) = w_{n,i}(u_0)(\xi_i - \E[\xi_i|u_0]), \qquad \text{for every } i \in [n].\] 
The nonparametric methods differ only in the choice of $w_{n,i}(u)$. In addition to Example~\ref{ex:kernel-learner0}, the following nonparametric methods satisfy \Cref{asp:if-decomposition} under regularity conditions analogous to those in \cite{tsybakov2008introduction,athey2019generalized,kallus2023stochastic}:
\begin{example}[k-Nearest-Neighbor (kNN) Estimator]\label{ex:knn-learner0}
    Denote the nearest index set $\Nscr_{\Dscr_n,u}(k_n) = \{i \in [n] | u_i~\text{is a kNN of}~u\}$. Then 
    $w_{n,i}(u) = \mathbf{1}_{\{u_i~\text{is a }k_n\text{-NN of}~u\}}$ satisfies Assumption~\ref{asp:if-decomposition} with $k_n$, the number of neighbors, as the hyperparameter. 
\end{example}

\begin{example}[Decision Tree Estimator]\label{ex:decision-tree-learner0}
Consider a decision-tree-induced partition $\tau: \R^{D_u} \to \{1,\ldots, L\}$ of the feature space $\R^{D_u}$ into disjoint regions such that $\tau^{-1}(1)\cup\ldots\cup \tau^{-1}(L) = \R^{D_u}$. Then $w_{n,i}(u) = \mathbf{1}_{\{\tau(u) = \tau(u_i)\}}$ satisfies Assumption~\ref{asp:if-decomposition}, with $L$, the number of leaves, as the hyperparameter.
\end{example}

\begin{example}[Random Forest Estimator]\label{ex:rf-learner0}
    Consider a forest consisting of a set of trees $\{\tau_1,\ldots, \tau_T\}$, where each $\tau_i: \R^{D_u} \to \{1,\ldots, L_i\}$ is a decision-tree-induced partition of $\R^{D_u}$ into $L_i$ regions. Then $w_{n,i}(u) = \sum_{j = 1}^T \mathbf{1}_{\{\tau_j(u_i) = \tau_j(u)\}}/T$ satisfies Assumption~\ref{asp:if-decomposition}, with $\beta$, the subsampling ratio, as the hyperparameter.
\end{example}

\paragraph{Computational Procedure.} Given limited data and a new covariate $u_0$, the high-level procedure for computing the optimally perturbed weighted EO solution is specified as follows:
\begin{enumerate}
    \item Compute the weighted EO solution $\hat\theta_{EO}(u_0)$;
    \item Compute the best data-driven adjustment matrix $\widehat H_n$;
    \item Output $\hat\theta_{EO}(u_0) + \frac{\widehat H_n}{n} \sum_{i \in [n]} M_{n,u_0}(u_i, \xi_i).$
\end{enumerate}
As in \Cref{sec:perturb-opt}, $\widehat H_n$ can be obtained by either of the following approaches, analogous to those in that section:
\begin{itemize}
    \item \textbf{Analytical Approach}. Given the estimated influence function $\{\widehat{\IFx}_{n,u_0}(u_i, \xi_i)\}_{i \in [n]}$ (which takes account of the nonparametric scaling $n^{\gamma - \frac{1}{2}}$), we compute $\widehat H_n$ by 
    \begin{equation}\label{eq:comp-h-cso}
    \widehat H_n
    =-\E_{\hat\P_n}
    [\widehat \IFx_{n,u_0}(u,\xi),M_{n,u_0}(u,\xi)]\big({\Var}_{\hat\P_n}
    [M_{n,u_0}(u,\xi)]\big)^{\dagger}.
    \end{equation}
    \item \textbf{Bootstrap Approach}. We follow the same routine as \Cref{alg:bootstrap-estimate-1}. For each bootstrap replication $b \in [B]$, 
    \begin{enumerate}
        \item Obtain the bootstrapped dataset $\Dscr_n^{(b)} = \{(u_i^{(b)},\xi_i^{(b)})\}_{i \in [n]}$ by resampling the dataset with replacement;
        \item Compute the bootstrapped EO solution $\hat\theta_{EO}^{(b)}(u_0) \in \argmin_{\theta}\sum_{i \in [n]}w_{n,i}^{(b)}(u_0) \ell(\theta;\xi_i^{(b)})$;
        \item Compute the bootstrapped side information $\widebar M_{u_0}^{(b)}= \frac{1}{n}\sum_{i \in [n]}M_{n,u_0}^{(b)}(u_i^{(b)},\xi_i^{(b)})$.
    \end{enumerate}
    Finally, we compute
    \[
    \widehat H_n = -\widehat{\Cov}_{B}(\hat\theta_{EO}^{(b)}(u_0), \widebar M_{u_0}^{(b)})\,\widehat{\Var}_{B}(\widebar M_{u_0}^{(b)})^{\dagger}.
    \]
\end{itemize}
\section{Extension of \Cref{sec:perturb-opt}: Joint Optimization of Adjustment Matrix and Side Information}\label{subsec:comp-h-phi}
In this section, we generalize our methodology to the joint optimization of $H$ and $M$. Specifically, we consider a setting in which multiple correct choices of side information $M$ are available. In such settings, the key challenge is no longer merely estimating $H$, but also selecting the side information that yields the
largest first-order improvement. We formalize the setting below:

\begin{assumption}[Regular Collection of Side Information]
\label{asp:collection-sideinfo}
Let $(\Phi,\|\cdot\|_2)$ denote a compact metric space, and let $\Phi^*\subseteq\Phi$ be a nonempty closed subset. Consider the mapping $M_{\phi}(\cdot;\cdot): \Theta \times \Xi \mapsto \R^{D_M}$. There exists a neighborhood $\mathcal N$ of $\theta^*$ such that the following conditions hold.

\begin{enumerate}[(i)]
    \item For every $(\theta,\phi)\in\mathcal N\times\Phi$, the mappings
    $\xi\mapsto M_{\phi}(\theta;\xi)$ and
    $\xi\mapsto \nabla_{\theta}M_{\phi}(\theta;\xi)$ are measurable.
    Moreover, for almost every $\xi$, the mappings $(\theta,\phi)\mapsto M_{\phi}(\theta;\xi)$ and $(\theta,\phi)\mapsto \nabla_{\theta}M_{\phi}(\theta;\xi)$ are continuous on $\mathcal N\times\Phi$.

    \item There exist measurable envelopes $F, L:\Xi\to[0,\infty)$
    such that $\E_{\P}\!\left[
        F(\xi)^3+L(\xi)^3
        \right]<\infty$ and
    \[
    \begin{small}
    \begin{aligned}
        \sup_{(\theta,\phi)\in\mathcal N\times\Phi}\paran{
        \|M_{\phi}(\theta;\xi)\|_2, \|\nabla_{\theta}M_{\phi}(\theta;\xi)\|}
        &\leq F(\xi),\\
        \sup\paran{\|M_{\phi}(\theta;\xi)-M_{\phi'}(\theta;\xi)\|_2, \|\nabla_{\theta}M_{\phi}(\theta;\xi)
        -\nabla_{\theta}M_{\phi'}(\theta;\xi)} & \leq L(\xi)\|\phi - \phi'\|_2, \forall \theta \in \Nscr, \phi, \phi' \in \Phi.
    \end{aligned}
    \end{small}
    \]
    Furthermore, $\int_0^1
        \sqrt{\log N(\varepsilon,\Phi,\|\cdot\|_2)}\,d\varepsilon<\infty$, where $N(\varepsilon,\Phi,\|\cdot\|_2)$ is the covering number of $(\Phi,\|\cdot\|_2)$~\citep{vapnik1999overview,bartlett2005local}.

    \item For every $\phi\in\Phi^*$, $\mu_{M_{\phi}}(\theta^*)
        :=\E_{\P}[M_{\phi}(\theta^*;\xi)]
        =\mathbf 0$, $\Sigma_0 \nabla_{\theta}\mu_{M_{\phi}}(\theta^*)^\top + \Gamma_{\phi} \neq \mathbf{0}$ and $\Cov_{\P}[\widetilde M_{\phi}(\theta^*;\xi)] \succ 0$, where $\mu_{M_{\phi}}(\theta) = \E_{\P}[M_{\phi}(\theta;\xi)]$, $\Gamma_{\phi} = \text{Cov}_{\P}\Paran{\IFx(\xi),M_{\phi}(\theta^*;\xi)}$ and $\widetilde M_{\phi}(\theta;\xi) = M_{\phi}(\theta;\xi) + \nabla_{\theta} \mu_{M_{\phi}}(\theta)\IFx(\xi)$. 
\end{enumerate}
\end{assumption}

Given a collection of side information and a prespecified constant $B>0$, define
$\Hscr_B=\{H\in\R^{D_\theta\times D_M}:\|H\|_\infty\leq B\}$.
The set of optimal pairs when $M$ is parametrized by $\phi \in \Phi^*$ in~\eqref{eq:optimal-h} is
\begin{equation}\label{eq:optimal-close-h-phi}
\Hscr_B^* :=
\argmin_{(H,\phi)\in\Hscr_B \times \Phi^*} \E_{\P}\Paran{\bigg\|\IFx(\xi)+H\,\widetilde M_{\phi}(\theta^*;\xi)\bigg\|_{I_{\ell}(\theta^*)}^2}.
\end{equation}
\begin{theorem}[Jointly Optimizing Adjustment Matrix and Side Information]\label{thm:opt-adjust-side-info}
Suppose Assumptions~\ref{asp:theta-star0},~\ref{asp:optimal-condition} and~\ref{asp:collection-sideinfo} hold. Fix $B>0$ and define $\Hscr_B = \{H\in\R^{D_\theta\times D_M}:
\|H\|_\infty\leq B\}$. For each $n$, let $(H_n,\phi_n)\in\Hscr_B\times\Phi^*$ be an approximate finite-sample minimizer satisfying
\[
\E_{\Dscr_n}[R(\hat\theta_{H_n,M_{\phi_n}})]
\leq
\inf_{(H,\phi)\in\Hscr_B\times\Phi^*}
\E_{\Dscr_n}[R(\hat\theta_{H,M_\phi})]+o(n^{-1}).
\]
Recall that $\Hscr_B^*$ is defined in~\eqref{eq:optimal-close-h-phi}. Then $\lim_{n\to\infty} \operatorname{dist}\big((H_n,\phi_n),\Hscr_B^*\big)=0$, where $\operatorname{dist}\big((H,\phi),\Hscr_B^*\big)
:=
\inf_{(\bar H,\bar\phi)\in\Hscr_B^*}
\big\{\|H-\bar H\|+\|\phi-\bar\phi\|_2\big\}$.
\end{theorem}
This is the joint analogue of \Cref{thm:opt-adjust}: approximate finite-sample minimizers of the expected excess risk over $\Hscr_B \times \Phi^*$ converge to the population-optimal set $\Hscr_B^*$ defined in~\eqref{eq:optimal-close-h-phi}, which need not be a singleton since different $\phi$ may attain the same first-order improvement. As before, it therefore suffices to estimate a population-optimal pair $(\widehat H, \widehat\phi)$ and plug it in.

Given limited data, we extend the high-level procedure in \Cref{subsec:comp-h} to jointly estimate $(\widehat H, \widehat \phi)$:
\begin{enumerate}
    \item Compute the EO solution $\hat\theta_{EO}$;
    \item Compute $(\widehat H, \widehat \phi)$ from the data (including clipping);
    \item Output $\hat\theta_{\widehat H, M_{\widehat\phi}} = \Pi_{\Theta}\Para{\hat\theta_{EO} + \frac{\widehat H}{n}\sum_{i=1}^n M_{\widehat\phi}\Para{\hat\theta_{EO};\xi_i}}$.
\end{enumerate}
Similarly, we can compute $(\widehat H,\widehat\phi)$ using analytical and bootstrap approaches analogous to those in \Cref{subsec:comp-h}, as described below.

\subsection{Analytical Approach} Compared with \Cref{subsec:comp-h}, one additional practical difficulty is that the set $\Phi^*$ in \Cref{asp:collection-sideinfo} is often unknown. Instead, we need to replace $\Phi^*$ with some known $\Phi$ and jointly optimize the adjustment matrix and side information over $\phi \in \Phi$. We still follow \Cref{alg:analytical-estimate-1} but replace Step 7, the closed-form estimate of $\widehat H$, with the following optimization problem over $(\widehat H, \widehat \phi)$:
\begin{equation}
\begin{aligned}\label{eq:cost-improvement2}
(\widehat H,\widehat \phi) \in & 
\argmin_{(H,\phi) \in \Hscr \times \Phi}
\Paran{\sum_{i=1}^n \big\|\widehat{\IFx}(\xi_i) + H\,\widehat M_{\phi}(\hat\theta_{EO};\xi_i)\big\|_{\hat I_{\ell}(\hat\theta_{EO})}^2}\\
& ~~~~~~\text{s.t.} \quad
\hat\theta_{EO} + \frac{H}{n}\sum_{i=1}^n M_{\phi}\Para{\hat\theta_{EO};\xi_i}\in\Theta,
\end{aligned}
\end{equation}
where $\widehat M_{\phi}(\theta;\xi) = M_{\phi}(\theta;\xi) + \Para{\frac{1}{n}\sum_{j = 1}^n \nabla_{\theta }M_{\phi}(\hat\theta_{EO};\xi_j)} \, \widehat\IFx(\xi)$. 

We specify the requirements for $\Phi$ as follows. 
\begin{assumption}[Construction Guarantee]\label{asp:construct-guarantee}
    For any $\delta > 0$, there exists a measurable closed set $\Phi(\delta)$ such that with probability at least $1-\delta$: (i) $\Phi^* \subseteq \Phi(\delta)$; and (ii) $\sup_{\phi \in \Phi(\delta)}\|\mu_{M_{\phi}}(\theta^*)\| = o(n^{-1/2})$.
\end{assumption}
That is, for a given $\delta$, $\Phi(\delta)$ is a subset that contains $\Phi^*$ with probability at least $1-\delta$ while maintaining controlled complexity. We consider a concrete example $\Phi(\delta)$ as follows:
\begin{example}[Moment Information Set]\label{ex:moment-info-set}
Suppose \Cref{asp:collection-sideinfo} holds and there are $K$ auxiliary distributions of $\xi$. For each $\phi \in \Phi^*\subseteq \R^{D_{\phi}}$, let $M_{\phi}$ take one of the following forms:
\begin{enumerate}[label=(\alph*)]
    \item $M_{\phi}(\theta;\xi)
    = \phi \cdot \Big(
    \xi^{(1)} - \E[\xi^{(1)}], \ldots,
    \xi^{(s)} - \E[\xi^{(s)}]
    \Big)^{\top}$,  where $D_{\phi} = D_{\theta}\times s$ and $\xi^{(i)}$ denote the $i$-th entry of $\xi$. This encodes the shared first moments of selected components in the uncertainty variable across data sources.
    \item $M_{\phi}(\theta;\xi) = \phi \xi + \theta + c$, where $D_{\phi} = {D_{\theta} \times D_{\xi}}$ and $c$ is a known constant. This encodes a known linear relationship between $\xi$ and $\theta$.
\end{enumerate}
For each $i \in [K]$, assume that a bounded $\xi \sim \Q_i$ satisfies  $\E_{\xi \sim \Q_i}[M_{\phi}(\theta_i^*;\xi)] = 0$ with a known $\theta_i^*$. We observe i.i.d. samples $\{\xi_{i,j}\}_{j \in [N]}$ drawn from $\Q_i$. The identification condition holds: $\sup_{\phi \in \Phi}\Para{\|\mu_{M_{\phi}}(\theta^*)\| - C{\max_{i \in [K]}\|\E_{\Q_i}[M_{\phi}(\theta_i^*;\xi)]}\|} \leq 0$ for some constant $C > 0$.

Then when the auxiliary sample size $N$ per domain satisfies $\frac{N}{D_{\phi}\log K} = \omega(n)$,  $\Phi(\delta)
:=\Big\{\phi:\ \max_{i\in[K]}\Big\|\sum_{j=1}^N M_{\phi}\Para{\theta_i^*;\xi_{i,j}}/N\Big\|_2^2
\lesssim D_{\phi}\log (KN/\delta)/{N} \Big\}$ satisfies \Cref{asp:construct-guarantee}.  
\end{example}
We formalize the consistency guarantee of $(\widehat H, \widehat\phi)$ in~\eqref{eq:cost-improvement2} under such  $\Phi(\delta)$:
    \begin{theorem}[Consistency of Analytical Approach over Multiple Sources of Side Information]\label{thm:perf}
Suppose Assumptions~\ref{asp:theta-star0},~\ref{asp:optimal-condition},~\ref{asp:collection-sideinfo}, and~\ref{asp:construct-guarantee} hold and $\hat\theta_{\widehat H, M_{\widehat\phi}}$ is solved from~\eqref{eq:cost-improvement2} with $\Phi := \Phi(\delta)$ and $\delta = o(1/n)$. 
Then $\hat\theta_{\widehat H,M_{\widehat \phi}}$ achieves the same amount of first-order improvement as $\hat\theta_{H^*, M_{\phi^*}}$.
\end{theorem}
This result shows that the data-driven estimator achieves the same amount of first-order improvement as the oracle choice $(H^*, \phi^*)$ under an appropriate $\delta$. However, solving~\eqref{eq:cost-improvement2} is generally non-convex, and in practice it is implemented via alternating minimization over $H$ and $\phi$.

\subsubsection{Special Cases} When $\Theta$ and $\Phi(\delta)$ have structures that make $\widehat H$ a closed-form function of $\widehat\phi$, the optimization problem~\eqref{eq:cost-improvement2} reduces to minimization over $\phi$. Consider the following case:

\begin{example}[Computation under Unconstrained $\Theta$]\label{prop:perf-no-constraint}
Suppose $\Theta = \R^{D_{\theta}}$. Denote $\widehat A_\phi := \frac{1}{n}\sum_{i = 1}^n \widehat M_\phi(\hat\theta_{EO};\xi_i)\widehat M_\phi(\hat\theta_{EO};\xi_i)^\top, \widehat B_\phi := \frac{1}{n}\sum_{i =1 }^n \widehat\IFx(\xi_i)\widehat M_\phi(\hat\theta_{EO};\xi_i)^\top$. Then the pair defined by
\begin{equation}\label{eq:if-max-unconstrained}
\widehat\phi \in \argmax_{\phi\in\Phi(\delta)}
\tr\!\Big(\hat I_{\ell}(\hat\theta_{EO})\,\widehat B_\phi \widehat A_\phi^{\dagger}\widehat B_\phi^\top\Big), \quad \widehat H = -\widehat B_{\widehat\phi} \widehat A_{\widehat\phi}^{\dagger},
\end{equation}
is an optimal solution to~\eqref{eq:cost-improvement2}.
\end{example}

\subsection{Bootstrap Approach}\label{app:bootstrap}
Since the best $\phi^*$ in~\eqref{eq:optimal-close-h-phi} depends on the Hessian of the cost function $I_{\ell}(\theta^*)$, we propose a bootstrap-based approach that avoids estimating $I_{\ell}(\theta^*)$. The procedure is described in \Cref{alg:bootstrap-estimate-2}. \Cref{alg:bootstrap-estimate-2} estimates the adjustment matrix and the parametrized side information through a bootstrap procedure with sample splitting. More specifically, at each replication, we draw the bootstrap samples and split them into two disjoint subsets. On the first subset, we compute the EO solution and evaluate the empirical side information for each candidate perturbation function parametrized by $\phi$. After repeating this across all replications, we select the adjustment matrix and side information that minimize the average loss evaluated at the directionally perturbed EO solution on the held-out second subset.

\begin{small}
    \begin{algorithm}[!htb]
\caption{Bootstrap Approach to Estimate $\widehat H, \widehat\phi$}
\label{alg:bootstrap-estimate-2}
\small
\begin{algorithmic}[1]
    \REQUIRE Number of bootstrap replications $B$, the dataset $\{\xi_i\}_{i \in [n]}$
    \FOR{$b = 1,\ldots, B$}
    \STATE Obtain $\Dscr_n^{(b)} = \{\xi_i^{(b)}\}_{i \in [n]}$ by resampling the dataset with replacement;
    \STATE Split $\Dscr_n^{(b)}$ into two disjoint subsets $\Dscr_{n,1}^{(b)}=\{\xi_i^{(b,1)}\}_{i \in [n_1]}$ and $\Dscr_{n,2}^{(b)}=\{\xi_i^{(b,2)}\}_{i \in [n_2]}$ with $n_1+n_2=n$;
    \STATE Compute the bootstrapped EO solution on the first subset
    \[
    \hat\theta_{EO}^{(b)}
    \in
    \argmin_{\theta \in \Theta}
    \widehat Z_1^{(b)}(\theta),
    \qquad
    \widehat Z_1^{(b)}(\theta)
    :=
    \frac{1}{n_1}\sum_{i = 1}^{n_1}\ell(\theta;\xi_i^{(b,1)});
    \]
    \STATE Compute the empirical bootstrapped side information on the first subset for each $\phi$
    \[
    \widebar M_{\phi}^{(b)}
    =
    \frac{1}{n_1}\sum_{i = 1}^{n_1}
    M_{\phi}\Para{\hat\theta_{EO}^{(b)};\xi_i^{(b,1)}},
    \]
    and the empirical side information on the original data $\widebar M_{\phi} = \frac{1}{n}\sum_{i = 1}^n M_{\phi} \Para{\hat\theta_{EO};\xi_i}$. 
    \ENDFOR
    \STATE Compute $(\widehat H, \widehat \phi)$ by solving
    \begin{equation}\label{eq:bootstrap-opt}
    (\widehat H, \widehat \phi)
    \in
    \argmin_{(H, \phi) \in \Hscr \times \Phi(\delta)}
    \frac{1}{B}\sum_{b=1}^B
    \widehat Z_2^{(b)}
    \Para{\hat\theta_{EO}^{(b)} + H\Para{\widebar M_{\phi}^{(b)}-\widebar M_{\phi}}},
    \end{equation}
    where
    \[
    \widehat Z_2^{(b)}(\theta)
    :=
    \frac{1}{n_2}\sum_{i = 1}^{n_2}\ell(\theta;\xi_i^{(b,2)}).
    \]
    \ENSURE $\widehat H, \widehat \phi$.
\end{algorithmic}
\end{algorithm}
\end{small}

\begin{theorem}[Consistency of Bootstrap Approach over Multiple Sources of Side Information]\label{thm:perf-bootstrap}
Suppose Assumptions~\ref{asp:theta-star0},~\ref{asp:optimal-condition},~\ref{asp:collection-sideinfo}, and~\ref{asp:construct-guarantee} hold, $\Hscr$ is compact, $\frac{n_1}{n} \to c \in (0, 1)$ and $B = \omega(1)$, and  $(\widehat H, \widehat\phi)$ is the output of \Cref{alg:bootstrap-estimate-2} with $\delta= o(1/n)$ in~\eqref{eq:bootstrap-opt}. Then $\hat\theta_{\widehat H,M_{\widehat \phi}}$ achieves the same amount of first-order improvement as $\hat\theta_{H^*, M_{\phi^*}}$.
\end{theorem}

\subsection{Proofs in \Cref{subsec:comp-h-phi}}
\subsubsection{\textit{Proof of \Cref{thm:opt-adjust-side-info}}} 
Define $L_n(H,\phi):=\E_{\Dscr_n}[R(\hat\theta_{H,M_\phi})]$ and $Q(H,\phi):=
\frac12
\E_{\P}\!\left[
\left\|
\IFx(\xi)
+
H\,
\widetilde M_\phi(\theta^*;\xi)
\right\|_{I_\ell(\theta^*)}^2
\right]$. Because $\Hscr_B\times\Phi^*$ is compact, the regularity conditions in
\Cref{asp:collection-sideinfo} make the expected-regret expansion in the proof of
\Cref{thm:improvement-principle} hold uniformly: $\sup_{(H,\phi)\in\Hscr_B\times\Phi^*}
\left|nL_n(H,\phi)-Q(H,\phi)\right|=o(1)$.

Since $(H_n,\phi_n)$ is an approximate minimizer, we have $nL_n(H_n,\phi_n)
\leq
\inf_{(H,\phi)\in\Hscr_B\times\Phi^*}nL_n(H,\phi)+o(1)$.
The uniform expansion therefore gives $Q(H_n,\phi_n)
\leq
\inf_{(H,\phi)\in\Hscr_B\times\Phi^*}Q(H,\phi)+o(1)$.

It remains to show convergence to the optimal set. Suppose, to the contrary, that
$\operatorname{dist}((H_n,\phi_n),\Hscr_B^*)$ does not converge to zero. Then there
exist $\varepsilon>0$ and a subsequence $\{n_k\}$ such that
$\operatorname{dist}((H_{n_k},\phi_{n_k}),\Hscr_B^*)\geq\varepsilon$ for every $k$.
By compactness of $\Hscr_B\times\Phi^*$, a further subsequence converges to some
$(\bar H,\bar\phi)\in\Hscr_B\times\Phi^*$. The function $Q$ is continuous under
\Cref{asp:collection-sideinfo}; hence the preceding approximate-optimality inequality
implies $Q(\bar H,\bar\phi)
=
\inf_{(H,\phi)\in\Hscr_B\times\Phi^*}Q(H,\phi)$. Thus $(\bar H,\bar\phi)\in\Hscr_B^*$, contradicting
$\operatorname{dist}((H_{n_k},\phi_{n_k}),\Hscr_B^*)\geq\varepsilon$.
Therefore,
$\lim_{n \to \infty}\operatorname{dist}((H_n,\phi_n),\Hscr_B^*) = 0$.

$\hfill \square$

\subsubsection{\textit{Proof of \Cref{ex:moment-info-set}}}
We verify the two conditions in \Cref{asp:construct-guarantee}.

By the uniform Bernstein inequality in Chapters 4--5 of \cite{wainwright2019high}, the condition in \Cref{ex:moment-info-set} implies that there exists a universal constant $C>0$ such that, for any $\delta\in(0,1)$, with probability at least $1-\delta$,
\begin{small}
\begin{equation}\label{eq:bernstein}
\begin{aligned}
& \sup_{\phi\in\Phi^*}\ \max_{i\in[K]}
\left\|
\frac{1}{N}\sum_{j=1}^N\Big(M_\phi(\theta_i^*;\xi_{i,j})-\E_{\xi\sim\Q_i}[M_\phi(\theta_i^*;\xi)]\Big)
\right\|_2\\
\le
& C\left(
\sqrt{\frac{\log N(1/N, \Phi^*, \|\cdot\|_2)+\log(K/\delta)}{N}}
+
\frac{\log N(1/N, \Phi^*, \|\cdot\|_2) + \log(K/\delta)}{N}
\right)\\
\lesssim  & C \sqrt{\frac{D_{\phi} \log(K N /\delta)}{N}}, 
\end{aligned}
\end{equation}
\end{small}
where the last inequality follows by $\log N(\varepsilon, \Phi^*, \|\cdot\|_2) \lesssim D_{\phi}\log(1/\varepsilon)$ for any $\varepsilon$ when $\Phi^* \subseteq \R^{D_{\phi}}$ is a finite-dimensional compact set.
We now verify that the moment-based construction satisfies \Cref{asp:construct-guarantee} by showing that the true invariant parameter lies in the empirical set with high probability. 

For the first condition in \Cref{asp:construct-guarantee}, consider $\tilde\phi \in \Phi^*$. It satisfies the invariance condition $\E_{\Q_i}[M_{\tilde\phi}(\theta;\xi)] = 0$ for every $i \in [K]$, as assumed in \Cref{ex:moment-info-set}. Then the above bound implies
\[
\max_{i \in [K]}
\left\|
\frac{1}{N}\sum_{j=1}^N M_{\tilde\phi}(\theta_i^*;\xi_{i,j})
\right\|_2
\lesssim
\sqrt{\frac{D_{\phi}\log(KN /\delta)}{N}}.
\]
Therefore, $\tilde\phi \in 
\Phi(\delta)$ with probability at least $1-\delta$.

For the second condition in \Cref{asp:construct-guarantee}, since $\Phi(\delta)\subseteq \Phi^*$, from the identification condition, we have:
\begin{small}
\begin{align*}
    \sup_{\phi \in \Phi(\delta)}\|\mu_{M_{\phi}}(\theta^*)\|_2 & \leq C \sup_{\phi \in \Phi(\delta)}\max_{i \in [K]}\|\E_{\Q_i}[M_{\phi}(\theta_i^*;\xi)]\|_2\\
    & \leq \sup_{\phi \in \Phi(\delta)}\max_{i \in [K]}\Para{\bigg\|\frac1N\sum_{j = 1}^N M_{\phi}(\theta_i^*;\xi_{i,j})\bigg\|_2 + \bigg\|\E_{\Q_i}[M_{\phi}(\theta_i^*;\xi)] - \frac1N\sum_{j = 1}^N M_{\phi}(\theta_i^*;\xi_{i,j})\bigg\|_2}\\
    & \lesssim 2C\sqrt{D_{\phi}\log(KN/\delta)/N} = o(n^{-1/2}),
\end{align*}
\end{small}
where the third inequality follows from~\eqref{eq:bernstein} and the last equality follows from the growth condition on $N$.

$\hfill \square$


\subsubsection{\textit{Proof of \Cref{thm:perf}}}
Define $\widetilde M_{\phi}(\theta;\xi) = M_{\phi}(\theta;\xi)+ \nabla_{\theta} \mu_{M_{\phi}}(\theta)\IFx(\xi)$ with $\mu_{M_{\phi}}(\theta) = \E_{\P}[M_{\phi}(\theta;\xi)]$. Applying the second-order Taylor expansion and influence-function representation used in Theorem~\ref{thm:improvement-principle} yields:
\begin{equation}\label{eq:risk-expansion}
\E[R\big(\hat\theta_{H,M_{\phi}}\big)]
\ = \frac{1}{2n} \E_{\P}\big[\|\IFx(\xi)+H \widetilde M_{\phi}(\theta^*;\xi)\|_{I_{\ell}(\theta^*)}^2\big]
\ +\ o\Big(\frac{1}{n}\Big), \forall H \in \Hscr, \phi \in \Phi(\delta).   
\end{equation}
We next establish the uniform convergence of the empirical objective in~\eqref{eq:cost-improvement2}. Let
\begin{align*}
    Q(H,\phi)
&:=\frac{1}{2}\E_{\P}\big[\|\IFx(\xi)+H\,\widetilde M_{\phi}(\theta^*;\xi)\|_{I_{\ell}(\theta^*)}^2\big],\\
\widehat Q_n(H,\phi)
& :=\frac{1}{2n}\sum_{i=1}^n\|\widehat{\IFx}(\xi_i)+H\,\widehat M_{\phi}(\hat\theta_{EO};\xi_i)\|_{\hat I_{\ell}(\hat\theta_{EO})}^2.
\end{align*}
By standard triangle inequality and uniform convergence arguments, since $H$ and $\phi$ are both bounded under Assumptions~\ref{asp:collection-sideinfo} and~\ref{asp:construct-guarantee}, 
we have $\E\bigg[\sup_{(H, \phi) \in\Hscr \times\Phi(\delta)}
\big|\widehat Q_n(H,\phi)-Q(H,\phi)\big|\bigg] = o(1)$.
Taking $\delta=o(1/n)$ in \Cref{asp:construct-guarantee}, the additional error from using $\widehat \phi \in \Phi(\delta)$ is $o(1)$ uniformly over $M_{\widehat \phi}$.

Then we transfer the optimality gap from the empirical objective to the population one. Let $(\widehat H,\widehat \phi)$ minimize \eqref{eq:cost-improvement2} 
over $\Hscr\times \Phi(\delta)$. Then, for any population oracle $(H^*, \phi^*) \in\argmin_{H,\phi\in\Phi^*} \E_{\P}\big[\|\IFx(\xi)+H\,\widetilde M_{\phi}(\theta^*;\xi)\|_{I_{\ell}(\theta^*)}^2\big]$, conditional on the event $\Phi^* \subseteq \Phi(\delta)$, we have
\begin{small}
\begin{align*}
& \quad \E[Q(\widehat H,\widehat \phi)]-Q(H^*, \phi^*) \\
& \leq 
\big[\E[Q(\widehat H,\widehat \phi)]-\E[\widehat Q_n(\widehat H,\widehat \phi)]\big] + \big[\E[\widehat Q_n(\widehat H,\widehat \phi)] -\E[\widehat Q_n(H^*, \phi^*)]\big]
+\big[\E[\widehat Q_n(H^*, \phi^*)]-Q(H^*, \phi^*)\big]\\
& \leq 2\E\bigg[\sup_{(H, \phi) \in\Hscr \times\Phi(\delta)}
\big|\widehat Q_n(H,\phi)-Q(H,\phi)\big|\bigg] = o(1)   
\end{align*}
\end{small}

On the event $\Phi^* \subseteq \Phi(\delta)$, applying expansion~\eqref{eq:risk-expansion} to both $(\widehat H,\widehat \phi)$ and $(H^*, \phi^*)$ gives the following performance gap:
\begin{align*}
\E\Paran{R(\hat\theta_{\widehat H, M_{\widehat\phi}})-R(\hat\theta_{H^*, M_{\phi^*}})}
&=\frac{1}{n}\Big[\E[Q(\widehat H,\widehat \phi)]-Q(H^*, \phi^*)\Big] + o\Big(\frac{1}{n}\Big) = o\Para{\frac{1}{n}}.
\end{align*}
On the other hand, we know $\P(\Phi^* \not\subseteq \Phi(\delta)) \leq \delta$. Setting $\delta = o(1/n)$ and applying \Cref{asp:construct-guarantee} gives: 
\begin{small}
\begin{equation}\label{eq:regret-prob-1n}
\begin{aligned}
&~~~\E[R(\hat\theta_{\widehat H, M_{\widehat\phi}})-R(\hat\theta_{H^*, M_{\phi^*}})]\\ &\leq \E[R(\hat\theta_{\widehat H, M_{\widehat\phi}})-R(\hat\theta_{H^*, M_{\phi^*}})|\Phi^* \subseteq \Phi(1/n)] + \E[|R(\hat\theta_{\widehat H, M_{\widehat\phi}}) - R(\hat\theta_{H^*, M_{\phi^*}})|]\P(\Phi^* \not\subseteq \Phi(\delta)) \\
& \leq o(1/n) + C o(1/n) = o(1/n), 
\end{aligned}
\end{equation}
\end{small}

$\hfill \square$

\subsubsection{\textit{Proof of \Cref{prop:perf-no-constraint}}}
In the unconstrained case, the optimal $\phi^*$ is the maximizer of the following problem:
\begin{equation}\label{eq:if-max-unconstrained-true}
\max_{\phi\in\Phi}
\tr\!\Big(I_{\ell}(\theta^*)\,B_\phi A_\phi^{-1}B_\phi^\top\Big), 
\end{equation}
where $A_\phi := \E\!\Big[\widetilde M_\phi(\theta^*;\xi)\widetilde M_\phi(\theta^*;\xi)^\top\Big], B_\phi := \E\!\Big[\IFx(\xi)\widetilde M_\phi(\theta^*;\xi)^\top\Big]$.
The objective value~\eqref{eq:if-max-unconstrained-true} corresponds to $Q(H^*, \phi^*)$, and the remainder of the proof is identical to the proof of \Cref{thm:perf}.

$\hfill \square$

\subsubsection{\textit{Proof of \Cref{thm:perf-bootstrap}}}
Define
\begin{align*}
\widehat{\mathcal J}^{(B)}(H,\phi)
& :=
\frac{1}{B}\sum_{b=1}^B
\widehat Z_2^{(b)}
\Para{
\hat\theta_{EO}^{(b)} + H\Para{\widebar M_{\phi}^{(b)} - \widebar M_{\phi}}}\\
Q(H,\phi)&:=\frac{1}{2}\E_{\P}\big[\|\IFx(\xi)+H\,\widetilde M_{\phi}(\theta^*;\xi)\|_{I_{\ell}(\theta^*)}^2\big].
\end{align*}
Because the second subset of the dataset is independent of the first subset,
\[
\E\!\left[
\widehat Z_2^{(b)}
\Para{
\hat\theta_{EO}^{(b)} + H\Para{\widebar M_{\phi}^{(b)} - \widebar M_{\phi}}
}
\;\middle|\;
\Dscr_{n,1}^{(b)}
\right]
=
\E_{\hat\P_n}\Paran{\ell\Para{\hat\theta_{EO}^{(b)} + H\Para{\widebar M_{\phi}^{(b)} - \widebar M_{\phi}};\xi}},
\]
where the expectation on the left-hand side is taken over bootstrap samples conditional on the original data $\Dscr_n$.
The term $\widehat Z_2^{(b)}(\hat\theta_{EO}^{(b)})$ does not depend on $(H,\phi)$, so we may subtract it from the criterion without changing the minimizer. For each $b \in [B]$, expanding the objective $\widehat Z_2^{(b)}(\cdot)$ to second order at $\hat\theta_{EO}$ gives
\begin{equation*}
\begin{aligned}
& \quad \widehat{\mathcal J}^{(B)}(H,\phi) - \widehat{\mathcal J}^{(B)}(\mathbf{0},\phi)\\
& =
\frac{1}{B}\sum_{b=1}^B \Big[\big(\hat\theta_{EO}^{(b)} - \hat\theta_{EO}\big)^{\top} I_{\ell}(\theta^*)\, H\big(\widebar M_{\phi}^{(b)} - \widebar M_{\phi}\big)
+ \tfrac{1}{2}\big\|H\big(\widebar M_{\phi}^{(b)} - \widebar M_{\phi}\big)\big\|_{I_{\ell}(\theta^*)}^2\Big]
+ o_p\Para{\frac{1}{n}},
\end{aligned}
\end{equation*}
which follows from the fact that $\hat\theta_{EO}$ optimizes the empirical objective approximately by \Cref{asp:optimal-condition} and its empirical Hessian converges to $I_{\ell}(\theta^*)$ as $B = \omega(1)$. 
Conditional on the data, $\widebar M_{\phi}^{(b)} - \widebar M_{\phi}$ is mean-zero up to $O_p(1/n_1)$. Following the bootstrap consistency established in the proof of \Cref{prop:bootstrap-consistency}, applied to the $n_1$ resamples, the conditional covariance matrices of $\sqrt{n_1}(\hat\theta_{EO}^{(b)} - \hat\theta_{EO})$ and $\sqrt{n_1}(\widebar M_{\phi}^{(b)} - \widebar M_{\phi})$ converge to those of $\IFx(\xi)$ and $\widetilde M_{\phi}(\theta^*;\xi)$. Therefore, averaging over $b$,
\begin{equation}\label{eq:boot-limit}
n\,\Big[\widehat{\mathcal J}^{(B)}(H,\phi) - \widehat{\mathcal J}^{(B)}(\mathbf{0},\phi)\Big]
=
\frac{n}{n_1}\Big[Q(H,\phi) - \tfrac{1}{2}\tr\big(I_{\ell}(\theta^*)\Sigma_0\big)\Big] + o_p(1).
\end{equation}
Since $n/n_1 \to 1/c > 0$ and the subtracted trace is free of $(H,\phi)$, the criterion~\eqref{eq:bootstrap-opt} and $Q$ share the same asymptotic minimizers.

Furthermore, conditional on the event $\Phi^* \subseteq \Phi(\delta)$, by the optimality of $(\widehat H,\widehat\phi)$ for~\eqref{eq:bootstrap-opt}, we have $\widehat{\mathcal J}^{(B)}(\widehat H,\widehat\phi) \le \widehat{\mathcal J}^{(B)}(H^*,\phi^*)$. Since Assumptions~\ref{asp:collection-sideinfo} and~\ref{asp:construct-guarantee} hold, uniform convergence~\eqref{eq:boot-limit} applied to both sides implies $Q(\widehat H,\widehat\phi) \le Q(H^*,\phi^*)+o_p(1)$. As in the proof of \Cref{thm:perf}, with $\delta = o(1/n)$ the additional error from optimizing over $\Phi(\delta)$ rather than $\Phi^*$ is $o_p(1)$ uniformly, so we also have $Q(\widehat H,\widehat\phi) \ge Q(H^*,\phi^*) - o_p(1)$. Combining both, and since $Q$ is bounded on $\Hscr\times\Phi(\delta)$,
\[
\E\big[Q(\widehat H,\widehat\phi)\big]-Q(H^*,\phi^*)=o(1).
\]
Next, by the excess-risk expansion,
\[
n\E_{\Dscr_n}\!\left[
R(\hat\theta_{H,M_\phi})
\right]
=
Q(H,\phi)+o(1)
\]
uniformly over $(H,\phi)\in\Hscr\times\Phi^*$. Applying this expansion at
$(\widehat H,\widehat\phi)$ and at $(H^*,\phi^*)$ gives
\[
n\Para{
\E_{\Dscr_n}
\left[
R(\hat\theta_{\widehat H,M_{\widehat\phi}})
\right]
-
\E_{\Dscr_n}
\left[
R(\hat\theta_{H^*,M_{\phi^*}})
\right]}
=
\E\big[Q(\widehat H,\widehat\phi)\big]-Q(H^*,\phi^*)+o(1)
=
o(1).
\]
On the other hand, when $\Phi^* \not\subseteq \Phi(\delta)$, which happens with probability at most $o(1/n)$, similar to~\eqref{eq:regret-prob-1n}, we can bound the difference of the excess risk contribution by $o(1/n)$. 

Therefore, $\E_{\Dscr_n}
\left[
R(\hat\theta_{\widehat H,M_{\widehat\phi}})
\right]
-
\E_{\Dscr_n}
\left[
R(\hat\theta_{H^*,M_{\phi^*}})
\right]
=
o(1/n)$.

$\hfill \square$

\section{Extension: Improvement Characterization for General Risk Functions}\label{app:generalization-improvement}

In this section, we extend the comparison of $Z(\hat\theta_{H,M})$ and $Z(\hat\theta_{EO})$ beyond
their expectations in the main body by considering $\E_{\Dscr_n}[g(Z(\hat\theta_{H,M}))]$ and
$\E_{\Dscr_n}[g(Z(\hat\theta_{EO}))]$ for a general risk function $g(\cdot)$.
{Throughout this section, we focus on side information that does not depend on the decision, i.e., $M(\theta;\xi) = M(\xi)$. In this case, $\nabla_{\theta}\mu_M(\theta^*) = \mathbf{0}$ and $\widetilde M(\theta^*;\xi) = M(\xi)$, so the
non-orthogonality condition in \Cref{thm:improvement-principle} reduces to
$\Gamma = \E_{\P}[\IFx(\xi) M(\xi)^{\top}] \neq \mathbf{0}$.}

\begin{assumption}[Conditions on $g(\cdot)$]\label{asp:riskfunction}
    The function $g:\mathbb{R}\to\mathbb{R}$ is nondecreasing and twice continuously
    differentiable with $\sup_{z}|g''(z)| < \infty$, and $g'(Z(\theta^*)) > 0$.
\end{assumption}

\begin{proposition}[First-Order Improvement for General Risk Functions]\label{prop:performance-risk-func}
Suppose Assumptions~\ref{asp:theta-star0},~\ref{asp:optimal-condition},~\ref{asp:regular-risk-rate}
and~\ref{asp:riskfunction} hold, and $M(\theta;\xi) = M(\xi)$ satisfies
$\E_{\P}[M(\xi)] = \mathbf{0}$ and $\E_{\P}[\|M(\xi)\|_2^4] < \infty$.
Then for any fixed $H$ with $\|H\| < \infty$,
\begin{equation}\label{eq:g-reduction}
\E_{\Dscr_n}[g(Z(\hat\theta_{H,M}))] - \E_{\Dscr_n}[g(Z(\hat\theta_{EO}))]
= g'(Z(\theta^*))\,\E_{\Dscr_n}\big[R(\hat\theta_{H,M}) - R(\hat\theta_{EO})\big]
+ o\Para{\frac{1}{n}}.
\end{equation}
Consequently, if $\Gamma \neq \mathbf{0}$, then $H^*$ in~\eqref{eq:optimal-close-h} satisfies
$\E_{\Dscr_n}[g(Z(\hat\theta_{H^*,M}))] - \E_{\Dscr_n}[g(Z(\hat\theta_{EO}))] = \Theta(1/n) < 0$,
i.e., $\hat\theta_{H^*,M}$ achieves a first-order improvement over $\hat\theta_{EO}$ under the
risk function $g$, and $H^*$ maximizes this improvement over all bounded $H$.
\end{proposition}
The identity~\eqref{eq:g-reduction} shows that, given correct decision-independent side information, applying a smooth monotone risk function rescales the first-order comparison by the positive constant $g'(Z(\theta^*))$ but leaves its sign, its $\Theta(1/n)$ order, and the optimal adjustment matrix unchanged. Furthermore, the optimizer $H^*$ from \Cref{thm:opt-adjust} is invariant to the choice of $g$. Intuitively, both $\hat\theta_{H,M}$ and $\hat\theta_{EO}$ concentrate around the same point $\theta^*$, so the excess risks of both solutions are of order $O_p(1/n)$ and the curvature of $g$ enters only at order $n^{-2}$.

\textit{Proof of \Cref{prop:performance-risk-func}.}
Since $\E_{\P}[M(\xi)] = \mathbf{0}$, we have $\theta_{H,M}^* = \theta^*$. By the
non-expansiveness of $\Pi_{\Theta}$,
$\|\hat\theta_{H,M} - \theta^*\|_2 \le \|\hat\theta_{EO} - \theta^*\|_2
+ \|H\|\,\big\|\frac{1}{n}\sum_{i=1}^n M(\xi_i)\big\|_2$,
and the Marcinkiewicz--Zygmund inequality with $\E_{\P}[\|M(\xi)\|_2^4] < \infty$ gives
$\E[\|\frac{1}{n}\sum_i M(\xi_i)\|_2^4] = O(n^{-2})$. 

For $\theta \in \{\hat\theta_{EO}, \hat\theta_{H,M}\}$, combined with
\Cref{asp:regular-risk-rate}, we obtain
$\E_{\Dscr_n}[\|\theta - \theta^*\|_2^4] = O(n^{-2})$. 
By \Cref{asp:theta-star0}, $\nabla_{\theta} Z(\theta^*) = \mathbf{0}$, so a second-order Taylor
expansion of $Z$ at $\theta^*$ yields $0 \le R(\theta) \le C\|\theta - \theta^*\|_2^2$
for some constant $C < \infty$ in a neighborhood of $\theta^*$, and therefore
$\E_{\Dscr_n}[R(\theta)^2] = O(n^{-2})$ for both solutions.

Next, expand $g$ at $Z(\theta^*)$: for each $\theta$ as above, there exists $\zeta$ between
$Z(\theta^*)$ and $Z(\theta)$ such that
\[
g(Z(\theta)) = g(Z(\theta^*)) + g'(Z(\theta^*))\, R(\theta)
+ \frac{1}{2} g''(\zeta)\, R(\theta)^2.
\]
Since $\sup_z |g''(z)| < \infty$, the expectation of the last term is $O(n^{-2})$. Taking expectations in the two expansions and
differencing gives~\eqref{eq:g-reduction}.

Finally, since $\mu_M(\theta^*) = \mathbf{0}$ and $\nabla_{\theta}\mu_M(\theta^*) = \mathbf{0}$,
the analysis of \Cref{thm:improvement-principle} applies to the right-hand side
of~\eqref{eq:g-reduction} with $A = 2\Gamma$ and $B = \Omega_M$: if $\Gamma \neq \mathbf{0}$, the
choice $H^* = -\Gamma\,\Omega_M^{\dagger}$ in~\eqref{eq:optimal-close-h} attains
$\E_{\Dscr_n}[R(\hat\theta_{H^*,M}) - R(\hat\theta_{EO})] = \Theta(1/n) < 0$ and maximizes the
first-order improvement. Multiplying by $g'(Z(\theta^*)) > 0$ preserves both the order and the
sign, which completes the proof.

$\hfill \square$

\section{Additional Details in \Cref{sec:numeric}}\label{app:numerical0}
\subsection{Additional Details in \Cref{subsec:lr}}\label{app:numerical-lr}
\paragraph{Basic Simulation Setting.} For the problem setup in the main body, we fix the feature $X \equiv 1$ and $\theta^* = 1$ across problem instances while varying the distribution of $\epsilon$ with different parametric specifications. The noise term $\epsilon$ is generated i.i.d. across observations following
\begin{equation}\label{eq:noise-spec}
\epsilon_i \overset{i.i.d.}{\sim}
\begin{cases}
\operatorname{Laplace}(0,4),
    & \text{if \texttt{noise\_type = laplace}},\\
\mathcal{N}(0,4^2),
    & \text{if \texttt{noise\_type = normal}},\\
t_4(0,4),
    & \text{if \texttt{noise\_type = t}},
\end{cases}
\end{equation}
where $\operatorname{Laplace}(0,4)$ denotes a Laplace distribution with
location zero and scale parameter $4$, and $t_4(0,4)$ denotes a location-scale
Student's $t$ distribution with four degrees of freedom, location zero, and
scale parameter $4$. In \Cref{subsec:lr}, for each sample size $n$, we perform 200 independent replications. 

In the following, we describe the design choices and practical implementation of the \textbf{EO-Chi2} solution $\hat\theta_{\widehat H, M_1}$ and the \textbf{EO-CVaR} solution $\hat\theta_{\widehat H, M_{\widehat\phi}}$.

\paragraph{Design and Implementation Details for EO-Chi2.} In \textbf{EO-Chi2}, we use the perturbation function
\[
M_1(\theta;\xi)=X\big(Y-\theta^\top X\big)^3.
\]
This is correct side information under symmetric exogenous noise because
$\E[M_1(\theta^*;\xi)]=\E[X\epsilon^3]=\mathbf 0$. It is also the implementable form because, for any candidate $\theta$, it can be evaluated directly from the observed pair $(X,Y)$. In contrast, the expression $\epsilon^3-3\sigma^2\epsilon$ obtained in the univariate calculation in the proof of \Cref{coro:linear-symmetric} is not the primitive moment inserted into the perturbed estimator. Rather, it is the effective first-order fluctuation $\widetilde M_1(\theta^*;\xi)
=M_1(\theta^*;\xi)
+\nabla_\theta\mu_{M_1}(\theta^*)\IFx(\xi)$ that results from evaluating $M_1$ at the estimated EO solution. Indeed, let $\Sigma_X=\E[XX^\top]$. Then $\nabla_\theta M_1(\theta^*;\xi)=-3XX^\top\epsilon^2, \IFx(\xi)=\Sigma_X^{-1}X\epsilon$. Under homoskedastic noise independent of $X$, $\nabla_\theta\mu_{M_1}(\theta^*)=-3\sigma^2\Sigma_X, \widetilde M_1(\theta^*;\xi)
=X\big(\epsilon^3-3\sigma^2\epsilon\big)$. Thus the Hermite-type expression $\epsilon^3-3\sigma^2\epsilon$ described after \Cref{coro:linear-symmetric} is generated automatically by the linearization of the observable residual moment $M_1(\hat\theta_{EO};\xi)$. It should not be used as the primitive moment because neither the population residual $\epsilon=Y-\theta^{*\top}X$ nor $\sigma^2$ is observed.

We next describe its practical implementation. In the univariate mean-estimation case $X\equiv1$, denote:
\[
\hat\epsilon_i=Y_i-\hat\theta_{EO},
\qquad
\hat\sigma^2=\frac1n\sum_{j=1}^n\hat\epsilon_j^2,
\qquad
\widehat M_i=\hat\epsilon_i^3-3\hat\sigma^2\hat\epsilon_i.
\]
Here $\widehat\IFx(\xi_i)=\hat\epsilon_i$ and $\widehat M_i$ is the empirical counterpart of $\widetilde M_1(\theta^*;\xi_i)$. Applying \Cref{alg:analytical-estimate-1} gives
\[
\widehat H
=-\frac{\sum_{i=1}^n\hat\epsilon_i\widehat M_i}
        {\sum_{i=1}^n\widehat M_i^2},
\qquad
\hat\theta_{\widehat H,M_1}
=\hat\theta_{EO}+\frac{\widehat H}{n}\sum_{i=1}^n\hat\epsilon_i^3.
\]
Notice that $\widehat M_i$ is used to estimate the optimal adjustment coefficient $\widehat H$, whereas the final perturbation continues to use the primitive moment $M_1(\hat\theta_{EO};\xi_i)=\hat\epsilon_i^3$.

In the general multivariate case, denote:
\[
\widehat\Sigma_X=\frac1n\sum_{j=1}^nX_jX_j^\top,
\qquad
\hat\epsilon_i=Y_i-\hat\theta_{EO}^{\top}X_i,
\qquad
\widehat\IFx(\xi_i)=\widehat\Sigma_X^\dagger X_i\hat\epsilon_i.
\]
Because $M_1(\theta;\xi)=X(Y-\theta^\top X)^3$, its empirical average derivative is $\frac1n\sum_{j=1}^n\nabla_\theta M_1(\hat\theta_{EO};\xi_j)
=-\frac3n\sum_{j=1}^nX_jX_j^\top\hat\epsilon_j^2$. Consequently, following \Cref{alg:analytical-estimate-1}, we compute:
\begin{equation}\label{eq:eo-chi2-plus}
\widehat H
=-\left(\sum_{i=1}^n\widehat\IFx(\xi_i)\widehat M_i^\top\right)
\left(\sum_{i=1}^n\widehat M_i\widehat M_i^\top\right)^\dagger,
\qquad
\hat\theta_{\widehat H,M_1}
=\hat\theta_{EO}
+\frac{\widehat H}{n}\sum_{i=1}^nX_i\hat\epsilon_i^3,
\end{equation}
where $\widehat M_i:=X_i\hat\epsilon_i^3
-3\left(\frac1n\sum_{j=1}^nX_jX_j^\top\hat\epsilon_j^2\right)
\widehat\IFx(\xi_i)$. As in the univariate case, $\widehat M_i$ accounts for the first-order effect of estimating $\theta^*$ and is used only to compute $\widehat H$; the final directional perturbation is formed from the observable primitive moment $M_1(\hat\theta_{EO};\xi_i)=X_i\hat\epsilon_i^3$.

\paragraph{Design and Implementation Details for EO-CVaR.} In \textbf{EO-CVaR}, we use the following perturbation function:
\[
M_{\phi}(\theta;\xi)=\left(1 - \frac{1}{1-\phi}\cdot \mathbf{1}_{\{\epsilon^2 > q_{\phi}\}}\right)X\epsilon,
\]
where $\epsilon=Y-\theta^{*\top}X$ and $q_\phi$ is the $\phi$-quantile of $\epsilon^2$, which approximates the formula of $M_{\phi}(\xi)$ in Appendix~\ref{app:cvardro}. We highlight that $M_{\phi}(\theta;\xi)$ is a function independent of $\theta$. If we instead replaced $\epsilon$ by the residual $Y-\theta^\top X$ in $M_{\phi}(\theta;\xi)$, the indicator and its quantile threshold would introduce a pointwise nondifferentiable dependence on $\theta$. We therefore treat $M_\phi$ as independent of $\theta$ in the analytical construction, so that $\nabla_\theta M_\phi(\theta;\xi)=\mathbf 0$ and the augmented moment reduces to $\widetilde M_\phi(\theta;\xi)=M_\phi(\theta;\xi)$. Although the population expression involves the unobserved residual $\epsilon$, in practice we use its empirical counterpart $\hat\epsilon_i=Y_i-\hat\theta_{EO}^{\top}X_i$ and replace $q_\phi$ by the empirical $\phi$-quantile of $\{\hat\epsilon_i^2\}_{i\in[n]}$.

We next describe how to select the parameter $\phi \in [0, 1)$ that leads to the largest performance improvement following the principle developed in Appendix~\ref{subsec:comp-h-phi}. Recall $\IFx(\xi) = \E[XX^{\top}]^{-1} X\epsilon$. Applying~\eqref{eq:if-max-unconstrained-true} in \Cref{subsec:comp-h-phi} with the squared loss, we have:
\[\max_{\phi \in [0, 1)}\tr\Big(\E[XX^{\top}] B_{\phi}^{\top} A_{\phi}^{-1} B_{\phi}\Big).\]
For the optimization problem above, when \(X\equiv 1\), i.e., mean estimation, the objective reduces to $\max_{\phi\in[0,1)} \frac{B_\phi^2}{A_\phi}$. Let $\sigma^2:=\E[\epsilon^2]$ and $m_\phi:=\E[\epsilon^2\mid \epsilon^2>q_\phi]$. Then $B_\phi = \sigma^2 - m_\phi$ and $A_\phi= \sigma^2+\frac{2\phi-1}{1-\phi}m_\phi$.
Hence the optimization problem becomes
\[
\max_{\phi\in[0,1)}
\frac{(\sigma^2-m_\phi)^2}
{\sigma^2+\frac{2\phi-1}{1-\phi}m_\phi} .
\]
Equivalently, for $u \in (0,1)$, let \(q_u\) denote the \(u\)-quantile of \(\epsilon^2\). Then $\sigma^2=\int_0^1 q_u\,du$ and $m_\phi=\frac{1}{1-\phi}\int_\phi^1 q_u\,du$. The optimization problem becomes
\[
\max_{\phi\in[0,1)}
\frac{\left(\int_0^1 q_u\,du-\frac{1}{1-\phi}\int_\phi^1 q_u\,du\right)^2}
{\int_0^1 q_u\,du+\frac{2\phi-1}{(1-\phi)^2}\int_\phi^1 q_u\,du}.
\]
After algebraic manipulation, the first-order optimality condition is
\[
(\sigma^2-m_\phi)\left(2\phi\,m_\phi-(2\phi-1)q_\phi\right) = 0.
\]
Empirically, we find $\widehat\phi$ by solving the preceding first-order condition after replacing $q_{\phi}$, $m_{\phi}$, and $\sigma^2$ by $\widehat q_\phi
:=\operatorname{VaR}_{\phi}\big(\{\hat\epsilon_j^2\}_{j\in[n]}\big),\widehat m_\phi
:=\operatorname{CVaR}_{\phi}\big(\{\hat\epsilon_j^2\}_{j\in[n]}\big), \widehat\sigma^2
:=\operatorname{Var}\big(\{\hat\epsilon_j\}_{j\in[n]}\big)$, respectively. 

In the univariate mean-estimation case $X\equiv1$, define
\[
\widehat M_{\phi,i}
:=\left(1-\frac{1}{1-\phi}
\mathbf{1}_{\{\hat\epsilon_i^2>\widehat q_\phi\}}\right)\hat\epsilon_i,
\qquad
\widehat A_\phi:=\frac1n\sum_{i=1}^n\widehat M_{\phi,i}^2,
\qquad
\widehat B_\phi:=\frac1n\sum_{i=1}^n\hat\epsilon_i\widehat M_{\phi,i}.
\]
We then set
\[
\widehat H=-\widehat A_{\widehat\phi}^{-1}\widehat B_{\widehat\phi},
\qquad
\hat\theta_{\widehat H,M_{\widehat\phi}}
=\hat\theta_{EO}
+\frac{\widehat H}{n}\sum_{i=1}^n\widehat M_{\widehat\phi,i}.
\]

\paragraph{Multivariate Data-Generating Process Setting.} We also consider a multivariate data-generating process as follows. We draw the feature $X$ according to a 5-dimensional Gaussian distribution, i.e., $X \sim \Nscr(\mathbf{1}_{D_x}/D_x, I_{D_x\times D_x})$ with $D_x = 5$. We set $\theta^* = \mathbf{1}_{D_x}$. The noise $\epsilon$ is generated as in~\eqref{eq:noise-spec}.

Here, we focus on comparing \textbf{EO-Chi2} (computed in~\eqref{eq:eo-chi2-plus}) and $\chi^2$-EO+ (computed in~\eqref{eq:chi2-eo}) with the standard EO solution. We report the method performance under different noise specifications in \Cref{fig:synthetic_regret_lr_new} and find that our main conclusions in \Cref{subsec:lr} remain valid.

\begin{figure*}[!htb]
\centering
\subfloat{
    \includegraphics[width=0.6\textwidth]{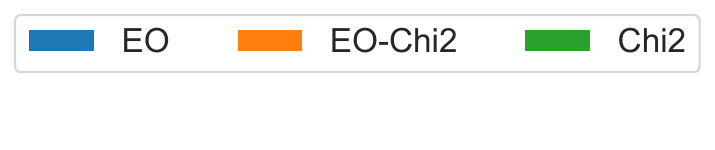}
}\\
\subfloat[Laplace\label{fig:lr_laplace2}]{
    \includegraphics[width=0.32\textwidth]{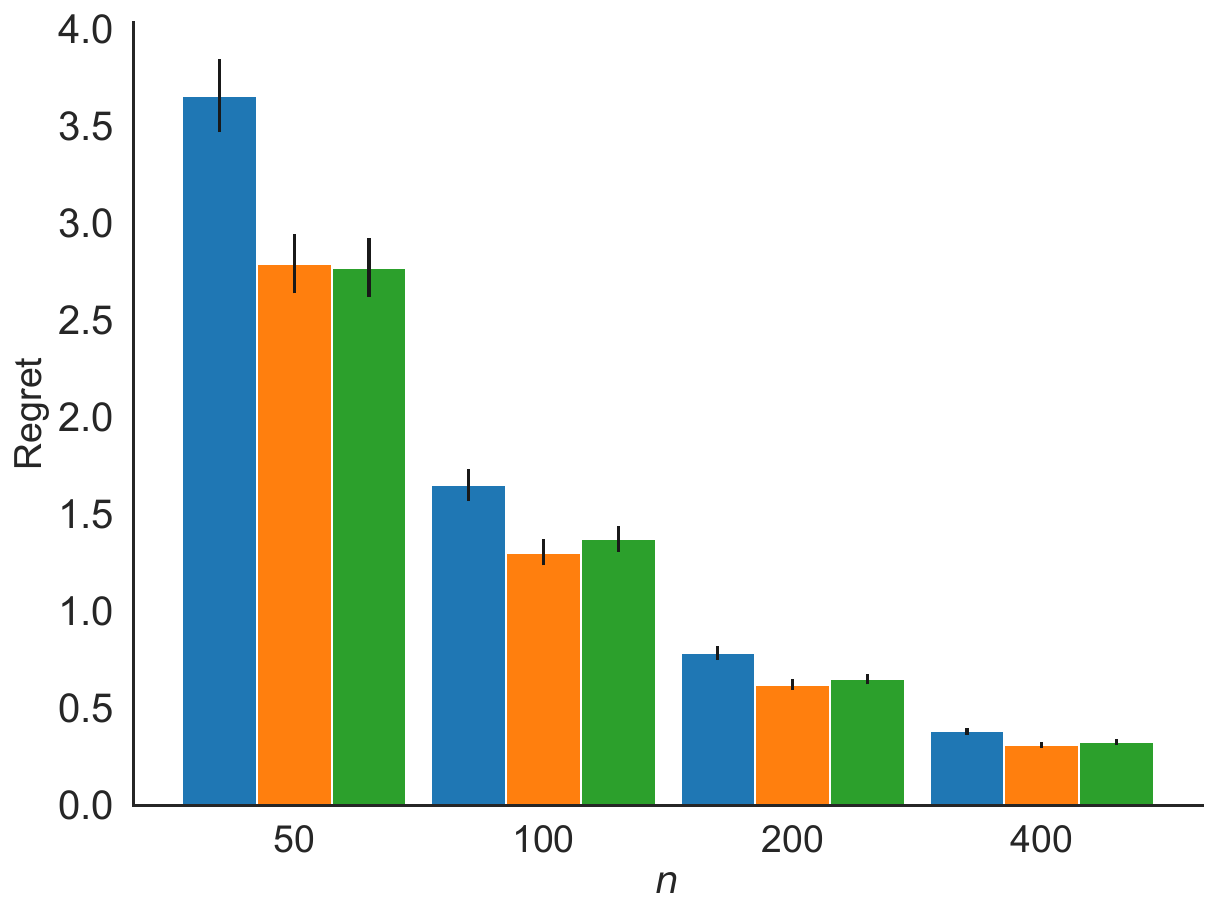}
}
\subfloat[Normal\label{fig:lr_normal2}]{
    \includegraphics[width=0.32\textwidth]{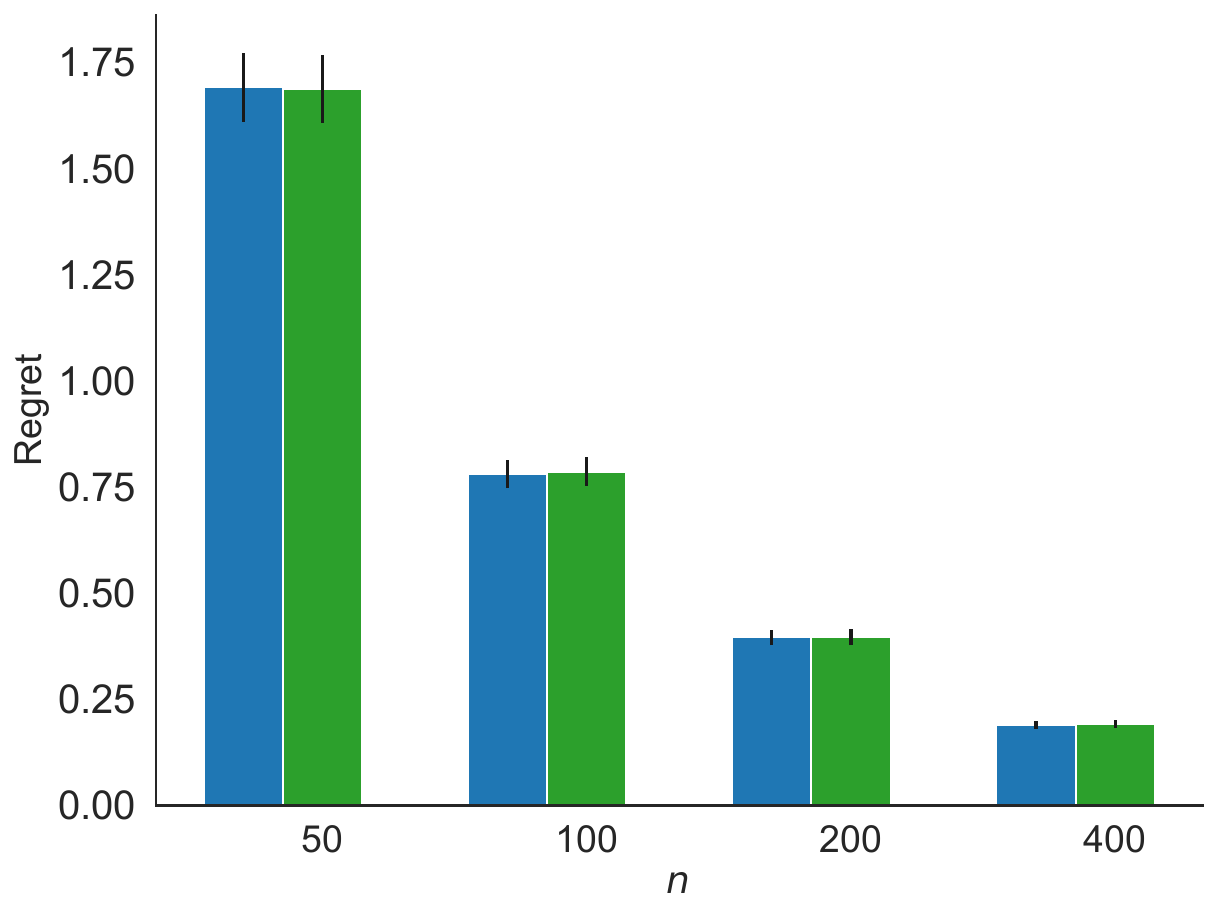}
}
\subfloat[$t$\label{fig:lr_t2}]{
    \includegraphics[width=0.32\textwidth]{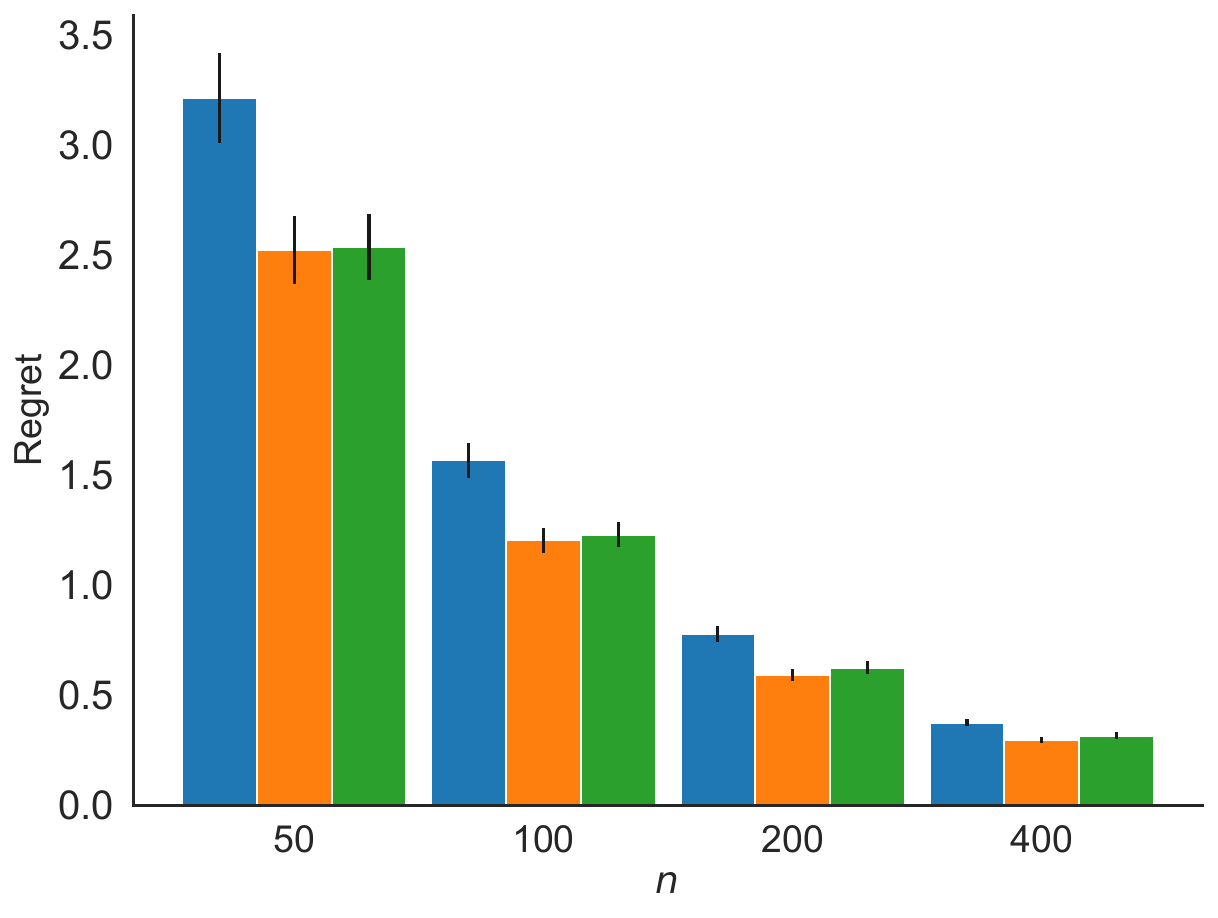}
}
\caption{Expected excess risk comparison among various methods in linear regression under different noise specifications, where the error bars denote the standard errors of the excess risk of each method.}
\label{fig:synthetic_regret_lr_new}
\end{figure*}

\subsection{Additional Details in \Cref{subsec:context-newsvendor}}\label{app:procedure}
We first describe how we construct the optimally perturbed weighted EO solution in the newsvendor problem.

\paragraph{Optimally Perturbed Weighted EO Solution.} For a new covariate $u$, let $\hat{\E}[\xi \mid u]$ denote the estimated conditional mean. Given the weighted empirical solution $\hat\theta_{EO}(u)$ and the corresponding weights ${w_{n,i}(u)}_{i\in[n]}$, we construct the perturbed solution using the analytical approach as follows.
\begin{enumerate}
    \item Estimate the conditional density of $\xi$ given $u$, denoted by $\hat f(\cdot \mid u)$, as an approximation of $\P_{\xi \mid u}$; 
    \item Compute the estimated influence function for each observation $i \in [n]$:
    \[\widehat{\IFx}_{n,u}(u_i, \xi_i) = -w_{n,i}(u)\Para{\frac{c-p}{p \hat f(\hat\theta_{EO}(u))}\mathbf{1}_{\{\xi_i > \hat\theta_{EO}(u)\}} +\frac{c}{p \hat f(\hat\theta_{EO}(u))}\mathbf{1}_{\{\xi_i \leq \hat\theta_{EO}(u)\}}};\]
    \item Compute $\widehat H$ by~\eqref{eq:comp-h-cso};
    \item Output the perturbed solution $\hat\theta_{EO}(u) + \frac{\widehat H}{n} \sum_{i \in [n]}w_{n,i}(u)(\xi_i - \hat\E[\xi|u]).$
\end{enumerate}


\paragraph{Weight and Hyperparameter Configurations.} We describe the weight specification of each nonparametric approach as follows:
\begin{itemize}
    \item kNNW: We set $w_{n,i}(u) = \mathbf{1}_{\{u_i~\text{is a }k_n\text{-NN of}~u\}}$ with the nearest neighbor size $k_n = \lceil~2\sqrt{n}~\rceil$;
\end{itemize}
For DTW and RFW: We use the codebase from \url{https://github.com/opimwue/ddop} with their default configurations, keeping the architecture of decision trees and random forests fixed throughout all experiments. 
\begin{itemize}
    \item DTW: We set \texttt{min\_samples\_split} = 10 and \texttt{min\_samples\_leaf} = 5;
    \item RFW: We set \texttt{n\_estimators} = 100, \texttt{min\_samples\_split} = 10 and \texttt{min\_samples\_leaf} = 5. 
\end{itemize}
 For the benchmark EO+ solutions, we consider the following hyperparameters:
\begin{itemize}
    \item \textbf{$\chi^2$-EO+}: We search 
    \[\lambda \in \{-1, -0.5, -0.1, -0.05, -0.01, -0.005, -0.001, 0.001, 0.005, 0.01, 0.05, 0.1, 0.5, 1\}\]
    and solve the $\chi^2$-DRO problem in \eqref{eq:chi2-eo} for positive $\lambda$. For negative values $\lambda$, we set \(\hat\theta_{\lambda}=2\hat\theta_{EO}-\hat\theta_{-\lambda}\);
    \item \textbf{Objective regularization (Obj-Reg)}: We search
    \[\lambda \in \{-5, -1, -0.5, -0.1, -0.05, -0.01, 0.01, 0.05, 0.1, 0.5, 1, 5\}\]
    and solve the regularized empirical optimization problem~\eqref{eq:reg-method} with $G(\theta)= \|\theta\|_2^2$ for positive $\lambda$. For negative values $\lambda$, we set \(\hat\theta_{\lambda}=2\hat\theta_{EO}-\hat\theta_{-\lambda}\).
\end{itemize}

\subsection{Contextual Newsvendor Studies in the Simulated Dataset}\label{app:newsvendor-simulate}
We consider a simulation environment for the contextual newsvendor problem. We set $c = 1$ and consider $p = 5, 20$. Under $\P$, the feature $u \in \R^5$ is sampled uniformly from $[0, 1]^{5}$, and $\xi \mid u$ is generated as $5 + B^{\top} u + \sin\|u\|_2 + \epsilon$, where $B \in \R^{5}$ is a fixed vector and $\epsilon \sim N(0, 1)$. We consider directionally perturbed weighted EO solutions with the following weight configurations:
\begin{enumerate}
    \item k-Nearest-Neighbor (kNN): $w_{n,i}(u) = \mathbf{1}_{\{u_i~\text{is a }k_n\text{-NN of}~u\}}$ with the number of nearest neighbors $k_n = \lceil~2\sqrt{n}~\rceil$;
    \item Kernel: $w_{n,i}(u) = \mathbf{1}_{\{\|u - u_i\|_2 \leq h_n\}}$ with $h_n = 1.5 n^{-\frac{1}{7}}$.
\end{enumerate}
The estimated conditional mean $\hat\E[\xi|u]$ is constructed in two ways:
\begin{enumerate}
    \item [$M_1$:] A noisy oracle estimator that outputs $\E[\xi|u] + \eta$, where $\eta \sim \Nscr(0, 25/n^2)$ and is independent of other randomness in the system.
    \item [$M_2$:] A random forest regressor, which is trained on $\{(u_i, \xi_i)\}_{i \in [m]}\sim \Q$. Here, $m = 10 n$, $\Q \neq \P$ but $\E_{\Q}[\xi|u] = \E_{\P}[\xi|u]$.
\end{enumerate}
We refer to the two optimally perturbed EO solutions with $M_1$ and $M_2$ as \textbf{EO-M1} and \textbf{EO-M2}, respectively. The two conditional mean estimators represent ``black-box'' AI oracles with a nearly correct mean or related data sources (including different warehouses and products). For completeness, we also compute the optimally perturbed EO solution with an exact conditional mean, i.e., $\hat\E[\xi|u] = \E_{\P}[\xi|u]$, and refer to the corresponding solution as \textbf{EO-}$M^*$.

We benchmark these optimally perturbed weighted EO solutions against the following standard weighted EO+ solutions:
\begin{itemize}
\item \textbf{\(\chi^2\)-EO+}. We search \(\lambda\in (-0.03,0.03]\) on a grid with step size $0.003$ and solve the $\chi^2$-DRO problem in \eqref{eq:chi2-eo} for positive $\lambda$. For negative values \(\lambda\in(-0.03,0)\), we set \(\hat\theta_{\lambda}=2\hat\theta_{EO}-\hat\theta_{-\lambda}\);
\item \textbf{CVaR-EO+.} We search \(\lambda\in(-1,1)\) with step size \(0.05\) and solve the CVaR-DRO problem for positive $\lambda$. For \(\lambda\in(-1,0)\), we compute \(\hat\theta_{\lambda}=2\hat\theta_{EO}-\hat\theta_{-\lambda}\).
\item \textbf{Obj-Reg.} We search \(\lambda\in(-0.3, 0.3)\) with step size \(0.03\) and solve the regularized empirical optimization problem in~\eqref{eq:reg-method} with $G(\theta)= \|\theta\|_2^2$ for positive $\lambda$. For \(\lambda\in(-0.3,0)\), we compute \(\hat\theta_{\lambda}=2\hat\theta_{EO}-\hat\theta_{-\lambda}\).
\end{itemize}

\paragraph{Evaluation.} We evaluate the performance of different methods on 10 randomly generated datasets $\Dscr_n$ for each sample size $n$. For each dataset $\Dscr_n$, the expected excess risk of the data-driven decision $\hat\theta(u)$ is approximated by averaging over 200 covariate points $\{\tilde u_i\}_{i \in [200]}$, i.e., $\frac{1}{200}\sum_{i \in [200]}\Para{\E_{\P_{\xi|\tilde u_i}}[\ell(\hat\theta(\tilde u_i);\xi)] - \min_{\theta \in \Theta}\E_{\P_{\xi|\tilde u_i}}[\ell(\theta;\xi)]}.$ 

\paragraph{Results.} We report the performance in \Cref{fig:synthetic_regret2} across weight configurations, noise distributions, and values of $p$. First, we find that both \textbf{EO-M1} and \textbf{EO-M2} solutions lead to statistically significant performance gains compared with the weighted EO solutions for both kNN and kernel weight estimators. In contrast, none of the standard EO+ solutions achieves comparable improvements. Second, there is little performance difference between the solutions using the noisy side information $M_1, M_2$ and the exact one $M^*$. This indicates that incorporating nearly correct side information with small noise does not lead to much performance degradation. We acknowledge that these findings are specific to the simulation design and are not intended to be universal.

\begin{figure*}[!htb]
\centering
\begin{subfigure}[t]{0.65\textwidth}
\centering
\includegraphics[width=\textwidth]{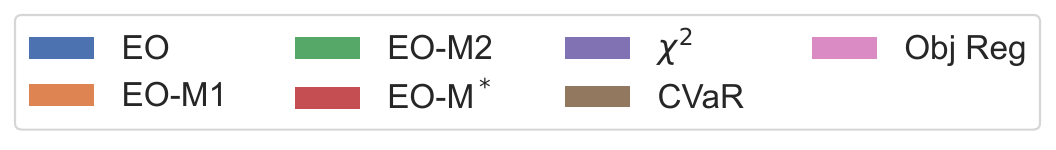}
\end{subfigure}
\begin{subfigure}[t]{0.36\textwidth}
\centering
\includegraphics[width=\textwidth]{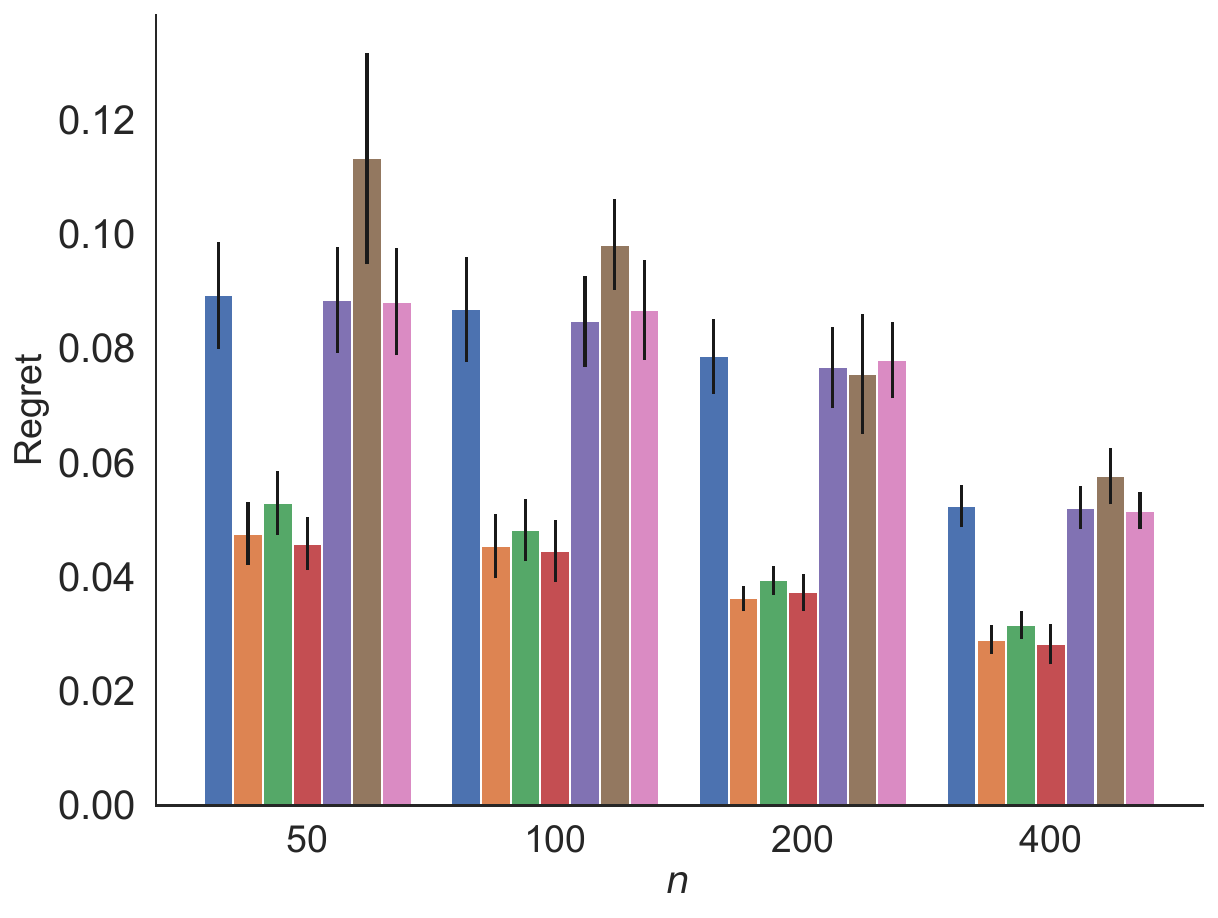}
\caption{kNN (Normal, 5)} %
\label{fig:knn_5_normal}
\end{subfigure}
\begin{subfigure}[t]{0.36\textwidth}
\centering
\includegraphics[width=\textwidth]{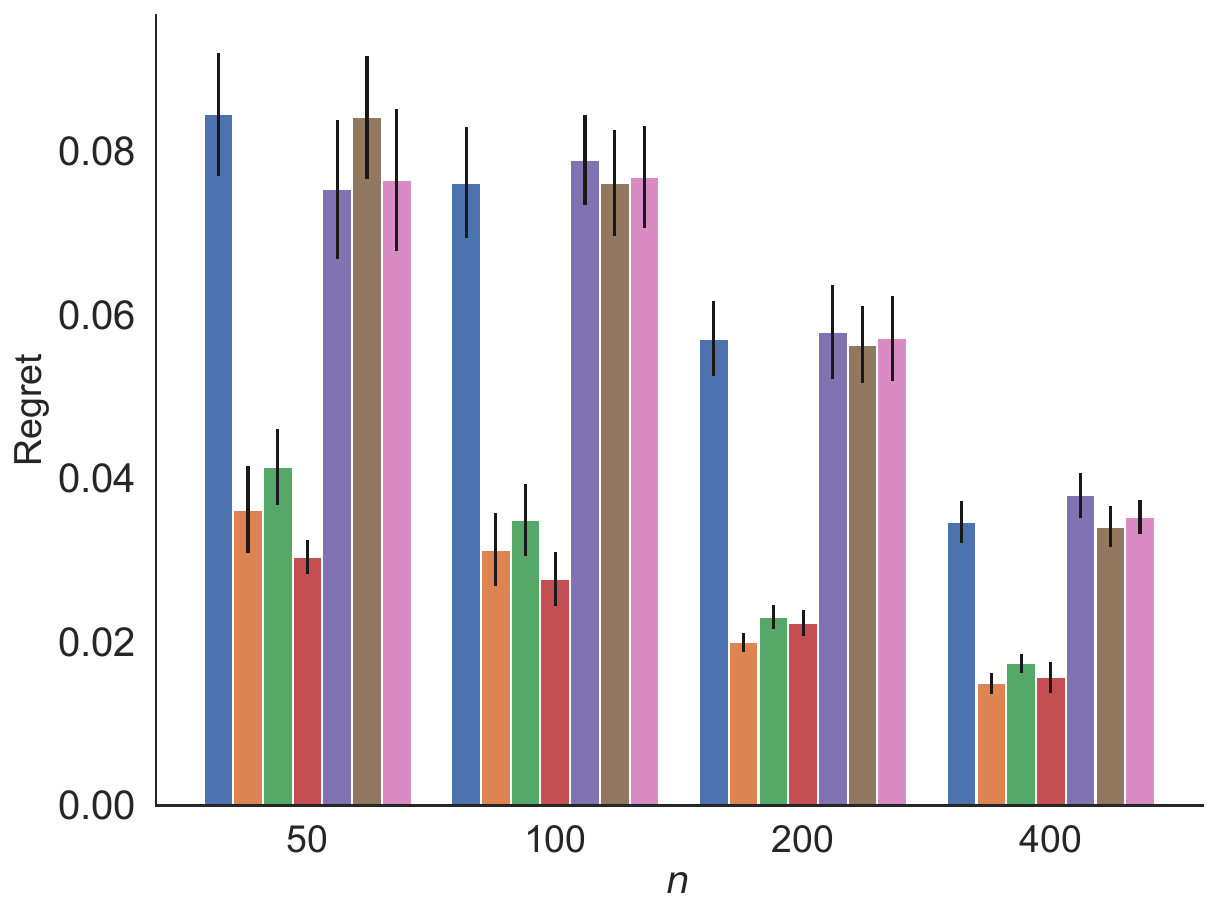}
\caption{Kernel (Normal, 5)}
\label{fig:kernel_5_normal}
\end{subfigure}
\begin{subfigure}[t]{0.36\textwidth}
\centering
\includegraphics[width=\textwidth]{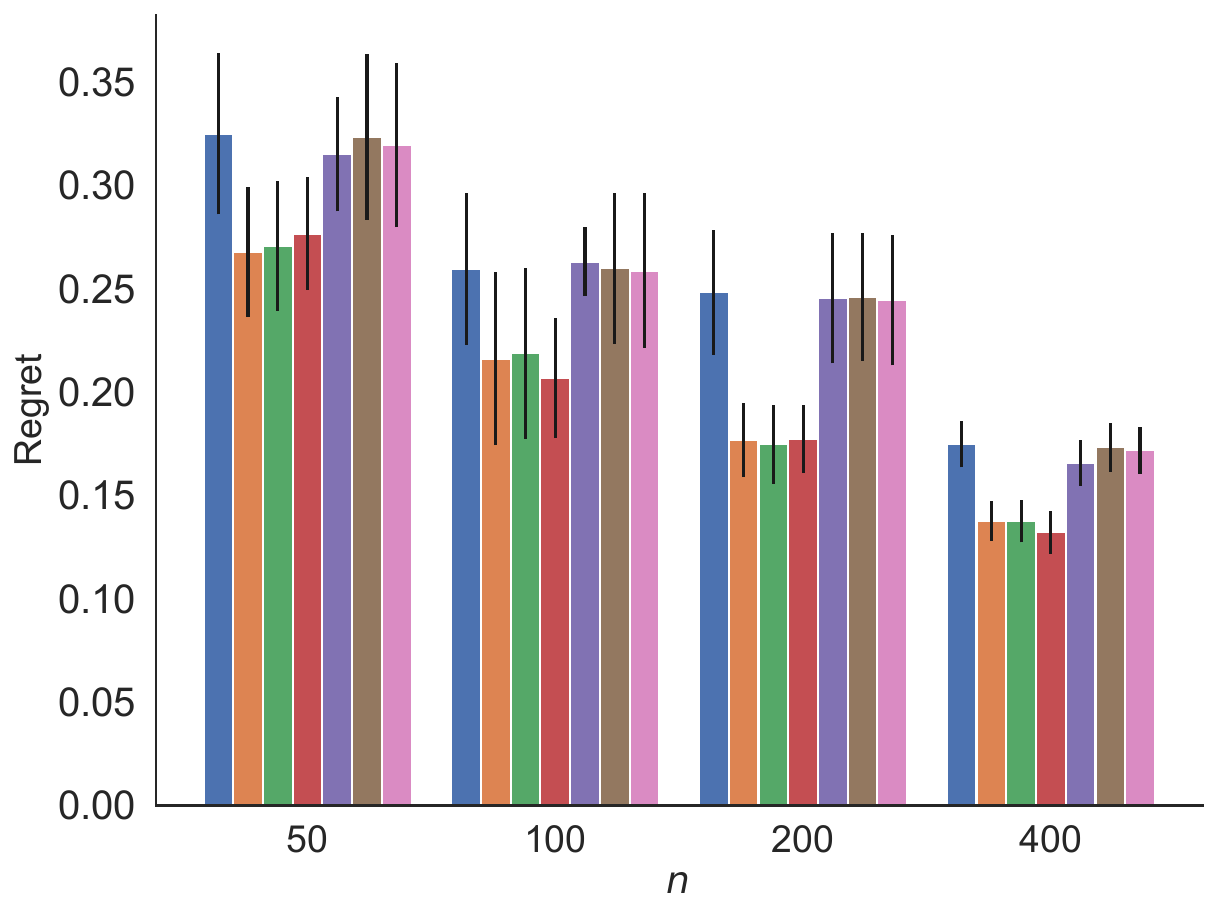}
\caption{kNN (Normal, 20)} %
\label{fig:knn_20_normal}
\end{subfigure}
\begin{subfigure}[t]{0.36\textwidth}
\centering
\includegraphics[width=\textwidth]{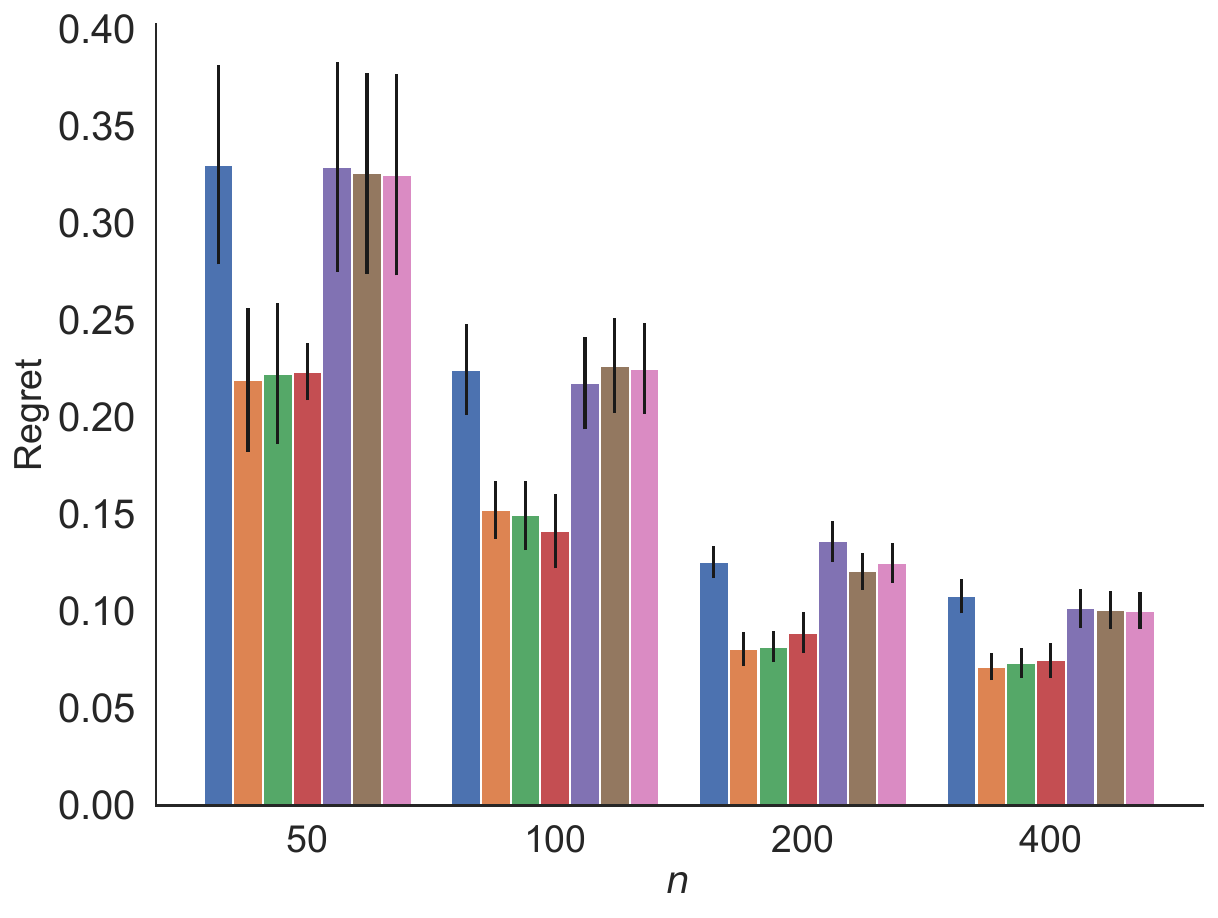}
\caption{Kernel (Normal, 20)}
\label{fig:kernel_20_normal}
\end{subfigure}
\begin{subfigure}[t]{0.36\textwidth}
\centering
\includegraphics[width=\textwidth]{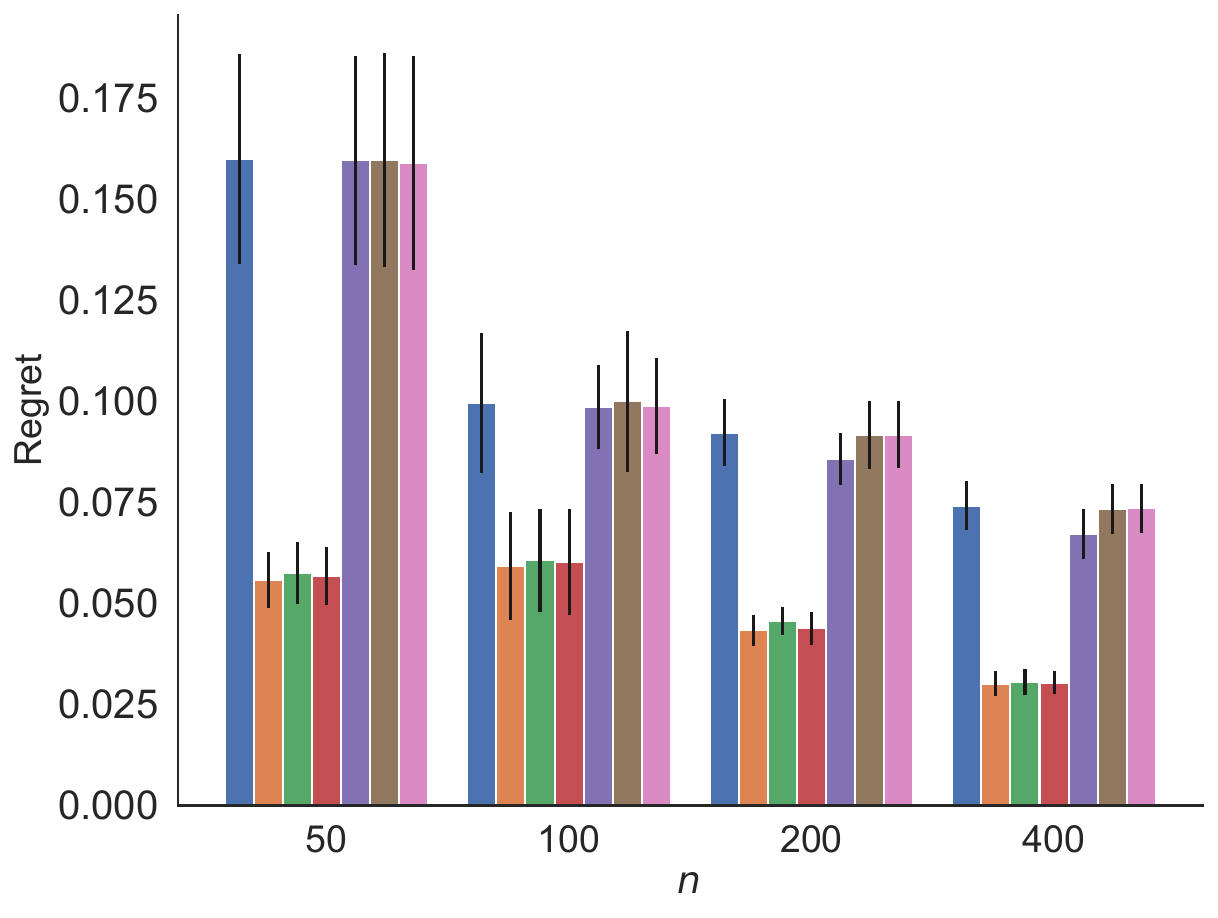}
\caption{kNN (Laplace, 5)} %
\label{fig:knn_5_laplace}
\end{subfigure}
\begin{subfigure}[t]{0.36\textwidth}
\centering
\includegraphics[width=\textwidth]{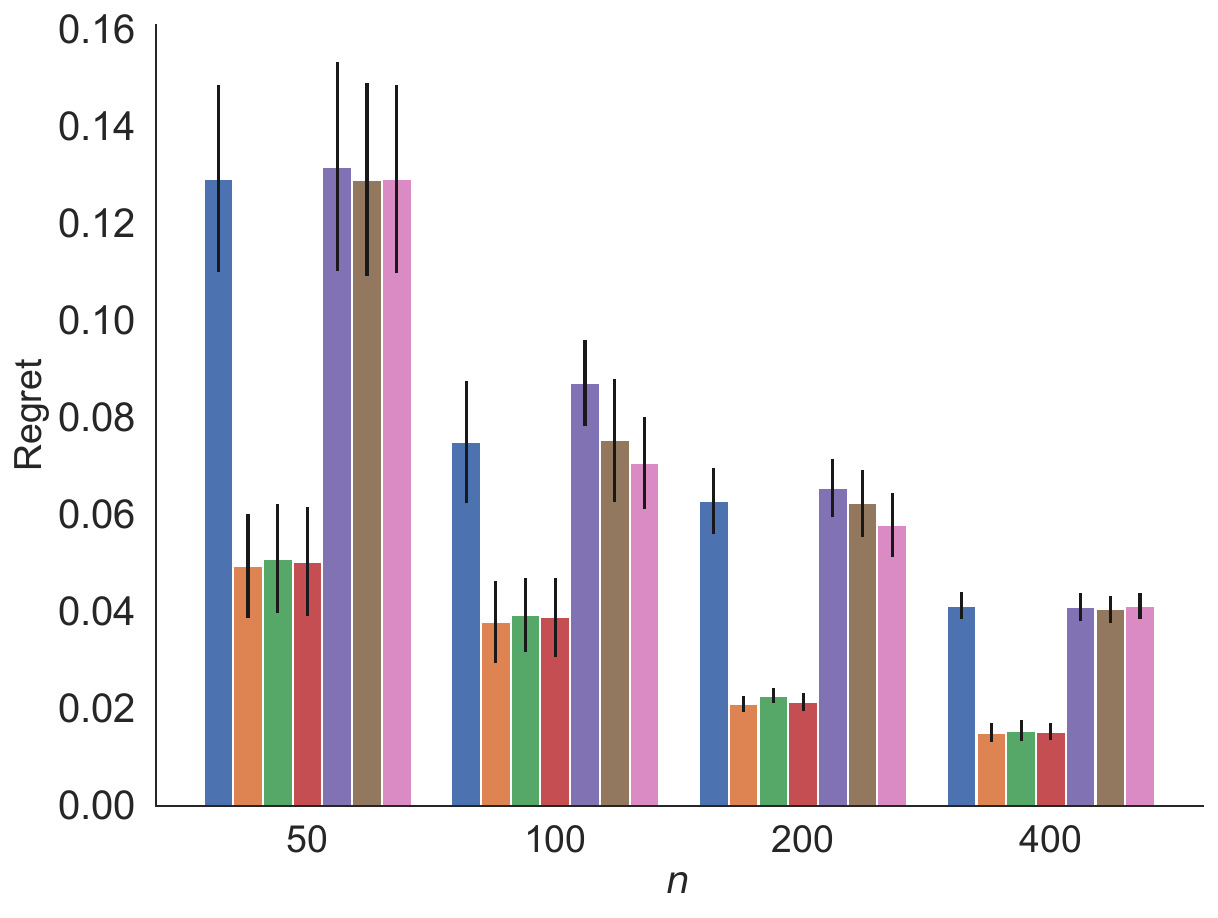}
\caption{Kernel (Laplace, 5)}
\label{fig:kernel_5_laplace}
\end{subfigure}
\begin{subfigure}[t]{0.36\textwidth}
\centering
\includegraphics[width=\textwidth]{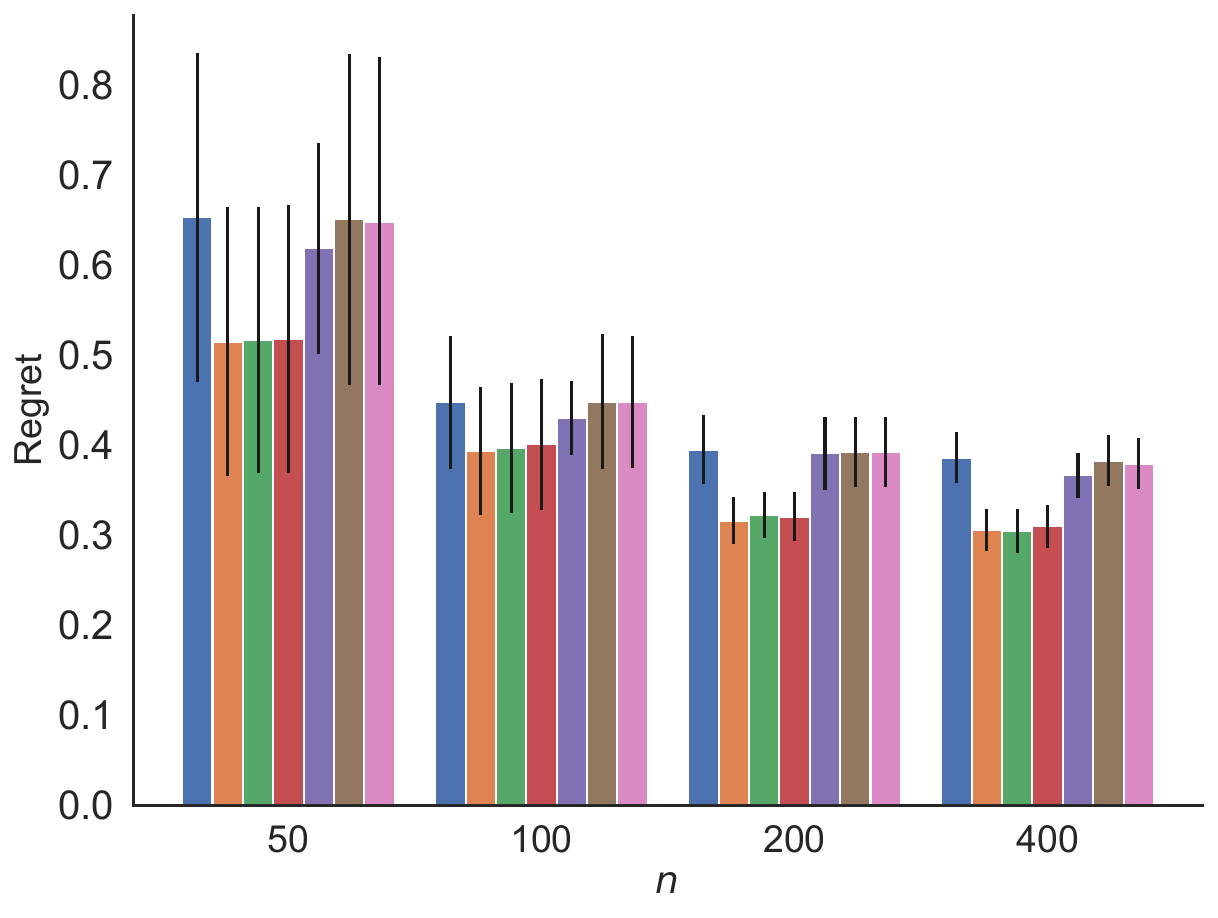}
\caption{kNN (Laplace, 20)} %
\label{fig:knn_20_laplace}
\end{subfigure}
\begin{subfigure}[t]{0.36\textwidth}
\centering
\includegraphics[width=\textwidth]{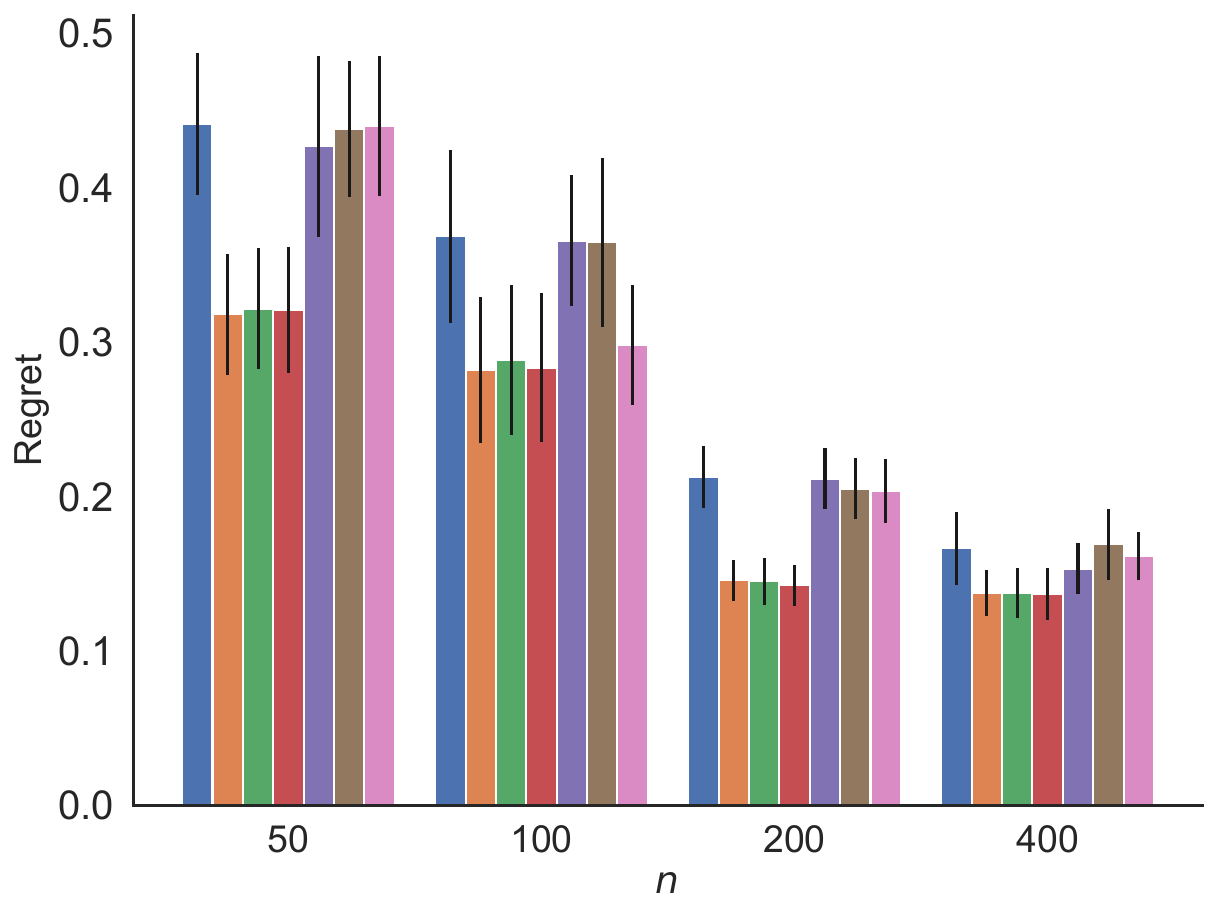}
\caption{Kernel (Laplace, 20)}
\label{fig:kernel_20_laplace}
\end{subfigure}
\caption{Expected regret comparison across noise specifications and values of $p$. Parentheses indicate the noise type $\epsilon$ and the value of $p$.} 
\label{fig:synthetic_regret2}
\end{figure*}

\paragraph{Further Comparison with Other Solutions Incorporating the Side Information.}
To demonstrate the superiority of our optimally perturbed weighted EO+ solutions, we consider two other methods that utilize conditional-mean side information and compare their theoretical and empirical performance. Here, we focus on the kNN weighted EO solution and denote the observed nearby samples as $\{\xi_i\}_{i \in [n_u]}$ (i.e., the historical sample with positive weight when solving the weighted EO solution).

\textbf{Benchmark A: Mean-shift correction.} This method corrects the empirical quantile by shifting it according to the difference between the sample mean and the known mean:
\[
\hat\theta_{1}
=
\hat\theta_{EO} - (\bar\xi - \E[\xi]),
\]
where $\bar\xi = \frac{1}{n_u}\sum_{i=1}^{n_u}\xi_i$ is the sample mean. Intuitively, if the sample mean $\bar\xi$ deviates from the true mean
$\E[\xi]$ due to sampling noise, the empirical quantile estimator inherits a similar shift, and this procedure compensates for this by aligning
the empirical mean with the known population mean. 

In fact, Benchmark A is a special case of the perturbed solution $\hat\theta_{H} = 
\hat\theta_{EO} + H(\bar\xi - \E[\xi])$. The optimal coefficient minimizing the asymptotic variance is $H^*
=
- \frac{\Cov(\IFx_\theta(\xi),\IFx_\mu(\xi))}
{\Var(\IFx_\mu(\xi))}$, where $\IFx_\theta$ and $\IFx_\mu$ denote the influence functions of the quantile estimator and the mean, respectively.

When $\epsilon \sim \Nscr(0,\sigma^2)$, $H^* = -1$, so the mean-shift estimator $\hat\theta_{1}$ is asymptotically optimal. When $\epsilon$ follows a Laplace distribution with mean $0$, letting $\tau = 1-\frac{c}{p}$ denote the optimal newsvendor quantile level, we have:
\[
H^*
=
\begin{cases}
\dfrac{\log(2\tau) - 1}{2}, & \tau \le \tfrac12, \\
\dfrac{\log(2(1-\tau)) - 1}{2}, & \tau \ge \tfrac12.
\end{cases}
\]

\textbf{Benchmark B: Reweighting samples to match the mean.}
This method computes the weighted empirical quantile
\[
\hat\theta_{2} =
\inf\left\{t:
\sum_{i=1}^{n_u} w_i \mathbf{1}\{\xi_i \le t\}
\ge
\tau
\right\},
\]
with weights $w=(w_1,\dots,w_{n_u})$ satisfying $\sum_{i=1}^{n_u} w_i = 1$, and $\sum_{i=1}^{n_u} w_i \xi_i = \E[\xi]$. This procedure adjusts the empirical distribution through moment matching so that the weighted sample mean equals the population mean. However, this procedure is a heuristic without theoretical guarantees.

\paragraph{Setup.} So far, we have demonstrated the optimality of our proposed method theoretically. Next, we report numerical comparisons in \Cref{tab:method_comp_info}. Settings 1--8 correspond to the configurations illustrated in $(a)$--$(h)$ in \Cref{fig:synthetic_regret2}. A-M1 and A-M2 denote Benchmark A with the conditional mean generated as M1 and M2, respectively; B-M1 and B-M2 denote Benchmark B with the conditional mean generated as M1 and M2, respectively.

\paragraph{Results.} In Settings 1--4, where the noise follows a normal distribution, A-M1 and A-M2 achieve costs similar to those of the corresponding perturbed solutions EO-M1 and EO-M2, respectively, and outperform them in some cases. In Settings 5--8, where the noise follows a Laplace distribution, their performance degrades, confirming the theoretical prediction that $H^* \neq -1$ in the non-Gaussian case. Across Settings 1--8, B-M1 and B-M2 consistently underperform the perturbed solutions EO-M1 and EO-M2, confirming that the reweighting procedure is suboptimal. Overall, these results demonstrate the optimality of our proposed method in both theory and practice. 
\begin{table}[!htb]
    \centering
    \caption{Comparison with other benchmarks in the simulated newsvendor problem. Values are reported as means, with standard errors following $\pm$.}
    \label{tab:method_comp_info}
    \resizebox{12cm}{!}{
    \begin{tabular}{cc|cc|cccc}
    \toprule
        Setting & $n$ & EO-M1 & EO-M2 & A-M1 & A-M2 & B-M1 & B-M2 \\
    \midrule
1 & 50
& 0.0362{\scriptsize $\pm 0.0053$}
& 0.0414{\scriptsize $\pm 0.0046$}
& 0.0347{\scriptsize $\pm 0.0058$}
& 0.0382{\scriptsize $\pm 0.0048$}
& 0.0420{\scriptsize $\pm 0.0050$}
& 0.0463{\scriptsize $\pm 0.0049$} \\
& 100
& 0.0313{\scriptsize $\pm 0.0045$}
& 0.0349{\scriptsize $\pm 0.0044$}
& 0.0263{\scriptsize $\pm 0.0036$}
& 0.0299{\scriptsize $\pm 0.0029$}
& 0.0312{\scriptsize $\pm 0.0039$}
& 0.0342{\scriptsize $\pm 0.0034$} \\
& 150
& 0.0200{\scriptsize $\pm 0.0011$}
& 0.0231{\scriptsize $\pm 0.0015$}
& 0.0160{\scriptsize $\pm 0.0010$}
& 0.0194{\scriptsize $\pm 0.0013$}
& 0.0210{\scriptsize $\pm 0.0013$}
& 0.0250{\scriptsize $\pm 0.0015$} \\
& 200
& 0.0150{\scriptsize $\pm 0.0013$}
& 0.0174{\scriptsize $\pm 0.0011$}
& 0.0137{\scriptsize $\pm 0.0016$}
& 0.0139{\scriptsize $\pm 0.0008$}
& 0.0171{\scriptsize $\pm 0.0013$}
& 0.0181{\scriptsize $\pm 0.0010$} \\
\midrule
2 & 50
& 0.0476{\scriptsize $\pm 0.0055$}
& 0.0530{\scriptsize $\pm 0.0055$}
& 0.0372{\scriptsize $\pm 0.0047$}
& 0.0522{\scriptsize $\pm 0.0061$}
& 0.0453{\scriptsize $\pm 0.0042$}
& 0.0580{\scriptsize $\pm 0.0054$} \\
& 100
& 0.0454{\scriptsize $\pm 0.0056$}
& 0.0482{\scriptsize $\pm 0.0054$}
& 0.0340{\scriptsize $\pm 0.0029$}
& 0.0442{\scriptsize $\pm 0.0040$}
& 0.0383{\scriptsize $\pm 0.0027$}
& 0.0489{\scriptsize $\pm 0.0048$} \\
& 150
& 0.0363{\scriptsize $\pm 0.0022$}
& 0.0395{\scriptsize $\pm 0.0025$}
& 0.0291{\scriptsize $\pm 0.0014$}
& 0.0382{\scriptsize $\pm 0.0023$}
& 0.0342{\scriptsize $\pm 0.0023$}
& 0.0425{\scriptsize $\pm 0.0022$} \\
& 200
& 0.0291{\scriptsize $\pm 0.0025$}
& 0.0316{\scriptsize $\pm 0.0024$}
& 0.0271{\scriptsize $\pm 0.0027$}
& 0.0305{\scriptsize $\pm 0.0020$}
& 0.0304{\scriptsize $\pm 0.0027$}
& 0.0350{\scriptsize $\pm 0.0022$} \\
\midrule
3 & 50
& 0.2193{\scriptsize $\pm 0.0369$}
& 0.2226{\scriptsize $\pm 0.0361$}
& 0.1931{\scriptsize $\pm 0.0310$}
& 0.2491{\scriptsize $\pm 0.0456$}
& 0.2068{\scriptsize $\pm 0.0310$}
& 0.2608{\scriptsize $\pm 0.0436$} \\
& 100
& 0.1522{\scriptsize $\pm 0.0148$}
& 0.1496{\scriptsize $\pm 0.0178$}
& 0.1271{\scriptsize $\pm 0.0174$}
& 0.1556{\scriptsize $\pm 0.0174$}
& 0.1327{\scriptsize $\pm 0.0187$}
& 0.1651{\scriptsize $\pm 0.0168$} \\
& 150
& 0.0807{\scriptsize $\pm 0.0088$}
& 0.0819{\scriptsize $\pm 0.0080$}
& 0.0670{\scriptsize $\pm 0.0076$}
& 0.0872{\scriptsize $\pm 0.0051$}
& 0.0735{\scriptsize $\pm 0.0060$}
& 0.0915{\scriptsize $\pm 0.0053$} \\
& 200
& 0.0716{\scriptsize $\pm 0.0069$}
& 0.0737{\scriptsize $\pm 0.0077$}
& 0.0689{\scriptsize $\pm 0.0078$}
& 0.0782{\scriptsize $\pm 0.0059$}
& 0.0738{\scriptsize $\pm 0.0066$}
& 0.0827{\scriptsize $\pm 0.0067$} \\
\midrule
4 & 50
& 0.2680{\scriptsize $\pm 0.0314$}
& 0.2708{\scriptsize $\pm 0.0314$}
& 0.1771{\scriptsize $\pm 0.0214$}
& 0.2483{\scriptsize $\pm 0.0195$}
& 0.1952{\scriptsize $\pm 0.0246$}
& 0.2685{\scriptsize $\pm 0.0201$} \\
& 100
& 0.2163{\scriptsize $\pm 0.0419$}
& 0.2188{\scriptsize $\pm 0.0414$}
& 0.1943{\scriptsize $\pm 0.0254$}
& 0.2384{\scriptsize $\pm 0.0310$}
& 0.1992{\scriptsize $\pm 0.0244$}
& 0.2492{\scriptsize $\pm 0.0313$} \\
& 150
& 0.1770{\scriptsize $\pm 0.0180$}
& 0.1749{\scriptsize $\pm 0.0190$}
& 0.1501{\scriptsize $\pm 0.0134$}
& 0.1703{\scriptsize $\pm 0.0170$}
& 0.1575{\scriptsize $\pm 0.0146$}
& 0.1825{\scriptsize $\pm 0.0167$} \\
& 200
& 0.1380{\scriptsize $\pm 0.0097$}
& 0.1376{\scriptsize $\pm 0.0102$}
& 0.1329{\scriptsize $\pm 0.0119$}
& 0.1561{\scriptsize $\pm 0.0083$}
& 0.1389{\scriptsize $\pm 0.0116$}
& 0.1612{\scriptsize $\pm 0.0119$} \\
\midrule
5 & 50
& 0.0494{\scriptsize $\pm 0.0107$}
& 0.0509{\scriptsize $\pm 0.0112$}
& 0.0538{\scriptsize $\pm 0.0126$}
& 0.0583{\scriptsize $\pm 0.0123$}
& 0.0557{\scriptsize $\pm 0.0124$}
& 0.0582{\scriptsize $\pm 0.0118$} \\
& 100
& 0.0378{\scriptsize $\pm 0.0084$}
& 0.0393{\scriptsize $\pm 0.0076$}
& 0.0386{\scriptsize $\pm 0.0094$}
& 0.0412{\scriptsize $\pm 0.0064$}
& 0.0398{\scriptsize $\pm 0.0087$}
& 0.0434{\scriptsize $\pm 0.0058$} \\
& 150
& 0.0209{\scriptsize $\pm 0.0016$}
& 0.0227{\scriptsize $\pm 0.0015$}
& 0.0199{\scriptsize $\pm 0.0021$}
& 0.0224{\scriptsize $\pm 0.0023$}
& 0.0224{\scriptsize $\pm 0.0018$}
& 0.0255{\scriptsize $\pm 0.0023$} \\
& 200
& 0.0151{\scriptsize $\pm 0.0020$}
& 0.0154{\scriptsize $\pm 0.0022$}
& 0.0163{\scriptsize $\pm 0.0019$}
& 0.0164{\scriptsize $\pm 0.0015$}
& 0.0169{\scriptsize $\pm 0.0021$}
& 0.0186{\scriptsize $\pm 0.0018$} \\
\midrule
6 & 50
& 0.0559{\scriptsize $\pm 0.0069$}
& 0.0575{\scriptsize $\pm 0.0076$}
& 0.0909{\scriptsize $\pm 0.0245$}
& 0.0683{\scriptsize $\pm 0.0126$}
& 0.0925{\scriptsize $\pm 0.0241$}
& 0.0695{\scriptsize $\pm 0.0133$} \\
& 100
& 0.0593{\scriptsize $\pm 0.0134$}
& 0.0607{\scriptsize $\pm 0.0128$}
& 0.0586{\scriptsize $\pm 0.0134$}
& 0.0646{\scriptsize $\pm 0.0072$}
& 0.0601{\scriptsize $\pm 0.0129$}
& 0.0679{\scriptsize $\pm 0.0070$} \\
& 150
& 0.0434{\scriptsize $\pm 0.0039$}
& 0.0457{\scriptsize $\pm 0.0036$}
& 0.0415{\scriptsize $\pm 0.0037$}
& 0.0432{\scriptsize $\pm 0.0035$}
& 0.0444{\scriptsize $\pm 0.0034$}
& 0.0464{\scriptsize $\pm 0.0031$} \\
& 200
& 0.0301{\scriptsize $\pm 0.0031$}
& 0.0306{\scriptsize $\pm 0.0032$}
& 0.0318{\scriptsize $\pm 0.0033$}
& 0.0360{\scriptsize $\pm 0.0033$}
& 0.0324{\scriptsize $\pm 0.0033$}
& 0.0371{\scriptsize $\pm 0.0036$} \\
\midrule
7 & 50
& 0.3699{\scriptsize $\pm 0.0418$}
& 0.3731{\scriptsize $\pm 0.0416$}
& 0.3859{\scriptsize $\pm 0.0472$}
& 0.5296{\scriptsize $\pm 0.0995$}
& 0.3928{\scriptsize $\pm 0.0434$}
& 0.5307{\scriptsize $\pm 0.1048$} \\
& 100
& 0.2086{\scriptsize $\pm 0.0275$}
& 0.2110{\scriptsize $\pm 0.0255$}
& 0.2160{\scriptsize $\pm 0.0279$}
& 0.2466{\scriptsize $\pm 0.0231$}
& 0.2168{\scriptsize $\pm 0.0257$}
& 0.2542{\scriptsize $\pm 0.0249$} \\
& 150
& 0.1611{\scriptsize $\pm 0.0189$}
& 0.1652{\scriptsize $\pm 0.0173$}
& 0.1679{\scriptsize $\pm 0.0213$}
& 0.1726{\scriptsize $\pm 0.0100$}
& 0.1726{\scriptsize $\pm 0.0199$}
& 0.1742{\scriptsize $\pm 0.0087$} \\
& 200
& 0.1116{\scriptsize $\pm 0.0111$}
& 0.1121{\scriptsize $\pm 0.0119$}
& 0.1157{\scriptsize $\pm 0.0108$}
& 0.1313{\scriptsize $\pm 0.0114$}
& 0.1135{\scriptsize $\pm 0.0139$}
& 0.1329{\scriptsize $\pm 0.0136$} \\
\midrule
8 & 50
& 0.5367{\scriptsize $\pm 0.0797$}
& 0.5359{\scriptsize $\pm 0.0803$}
& 0.5501{\scriptsize $\pm 0.0808$}
& 0.6498{\scriptsize $\pm 0.0881$}
& 0.5535{\scriptsize $\pm 0.0798$}
& 0.6505{\scriptsize $\pm 0.0906$} \\
& 100
& 0.3640{\scriptsize $\pm 0.0626$}
& 0.3680{\scriptsize $\pm 0.0594$}
& 0.3726{\scriptsize $\pm 0.0630$}
& 0.3947{\scriptsize $\pm 0.0428$}
& 0.3749{\scriptsize $\pm 0.0614$}
& 0.3988{\scriptsize $\pm 0.0427$} \\
& 150
& 0.2812{\scriptsize $\pm 0.0244$}
& 0.2896{\scriptsize $\pm 0.0235$}
& 0.2918{\scriptsize $\pm 0.0258$}
& 0.3091{\scriptsize $\pm 0.0257$}
& 0.2996{\scriptsize $\pm 0.0253$}
& 0.3206{\scriptsize $\pm 0.0234$} \\
& 200
& 0.2010{\scriptsize $\pm 0.0147$}
& 0.2015{\scriptsize $\pm 0.0146$}
& 0.2073{\scriptsize $\pm 0.0151$}
& 0.2546{\scriptsize $\pm 0.0141$}
& 0.2051{\scriptsize $\pm 0.0169$}
& 0.2580{\scriptsize $\pm 0.0168$} \\
    \bottomrule
    \end{tabular}}
\end{table}

\end{document}